\documentclass[oneside,reqno]{amsart}
\usepackage[greek,english]{babel}
\usepackage{amsthm}
\usepackage{amsbsy}
\usepackage{amsfonts}
\usepackage{graphicx}
\newtheorem{thm}{Theorem}
\newtheorem{cor}{Corrolarry}
\newtheorem{lem}{Lemma}
\newtheorem{pro}{Proposition}
\newtheorem{rem}{Remark}
\newtheorem{definition}{Definition}
\numberwithin{equation}{section} \numberwithin{lem}{section}
\numberwithin{thm}{section} \numberwithin{cor}{section}
\numberwithin{pro}{section} \numberwithin{rem}{section}
\begin{document}
\title[The ground state of a Gross--Pitaevskii energy]{The ground state of a Gross--Pitaevskii energy with general potential in the
Thomas--Fermi limit}

\author{Georgia Karali}

\address{Department of Applied Mathematics, University of
Crete, GR--714 09 Heraklion, Crete, Greece, and Institute of Applied
and Computational Mathematics, IACM, FORTH, Greece.}
\email{gkarali@tem.uoc.gr}
\author{Christos Sourdis}
\address{Department of Applied Mathematics, University of
Crete, GR--714 09 Heraklion, Crete, Greece.}
\email{csourdis@tem.uoc.gr} \maketitle
\begin{abstract}
We study the ground state which minimizes a Gross--Pitaevskii energy
with \emph{general non-radial} trapping potential, under the unit
mass constraint, in the Thomas--Fermi limit where a small parameter
$\varepsilon$ tends to $0$. This ground state plays an important
role in the mathematical treatment of recent experiments on the
phenomenon of Bose--Einstein condensation, and in the study of
various types of solutions of nonhomogeneous defocusing nonlinear
 Schr\"{o}dinger equations. Many of these applications
require delicate estimates for the behavior of the ground state near
the boundary of the condensate, as $\varepsilon\to 0$, in the
vicinity of which the ground state has irregular behavior in the
form of a steep \emph{corner layer}. In particular, the role of this
layer is important in order to detect the presence of vortices in
the small density region of the condensate, understand the
superfluid flow around an obstacle, and also has a leading order
contribution in the energy. In contrast to previous approaches, we
utilize a perturbation argument to go beyond the classical
Thomas--Fermi approximation and \emph{accurately} approximate the
layer by the Hastings--McLeod solution of the Painlev\'{e}--II
equation. This settles an open problem (cf. \cite[pg. 13 or Open
Problem 8.1]{aftalionbook}), answered very recently \emph{only} for
the special case of the model harmonic potential \cite{pelinovsky}.
In fact, we even improve upon previous results that relied heavily
on the radial symmetry of the potential trap. Moreover, we show that
the ground state has the  \emph{maximal regularity} available,
namely it remains uniformly bounded in the $\frac{1}{2}$-H\"{o}lder
norm, which is the \emph{exact} H\"{o}lder regularity of the
singular limit profile, as $\varepsilon\to 0$.
Our study is highly motivated by an interesting open problem posed
recently by Aftalion, Jerrard, and Royo-Letelier
\cite{aftalion-jerrard}, and an open question of Gallo and
Pelinovsky \cite{gallo}, concerning the removal of the radial
symmetry assumption from the potential trap.
%
\end{abstract}
\tableofcontents

\section{Introduction}
\subsection{The problem}\label{secproblem}
This paper is concerned with the analysis of the $\varepsilon \to 0$
limiting behavior of the Gross-Pitaevskii energy
\begin{equation}\label{eqGminimizer}
G_\varepsilon(u)=\int_{\mathbb{R}^2}^{} \left\{\frac{1}{2}|\nabla
u|^2+\frac{1}{4\varepsilon^2}|u|^4+\frac{1}{2\varepsilon^2}W(\textbf{y})|u|^2
\right\}d\textbf{y},
\end{equation}
minimized in
\begin{equation}\label{eqconstraint}
\mathcal{H}\equiv \left\{u\in W^{1,2}(\mathbb{R}^2;\mathbb{C}):\
\int_{\mathbb{R}^2}^{}W(\textbf{y})|u|^2d\textbf{y}<\infty,\ \
\int_{\mathbb{R}^2}^{}|u|^2d\textbf{y}=1\right\},
\end{equation}
 where $\varepsilon>0$ is a small parameter and,
unless specified otherwise, the potential $W$ will satisfy:
\begin{equation}\label{eqV1}
W \ \textrm{is\ nonnegative},\ W\in C^1,
\end{equation}
and
\begin{equation}\label{eqainfinity}
\textrm{there\ exist\ constants}\ C>1,\ p\geq 2\ \textrm{such\
that}\ \ \frac{1}{C}|\textbf{y}|^p\leq W(\textbf{y})\leq
C|\textbf{y}|^p\ \ \textrm{if}\ \ |\textbf{y}|\geq C,
\end{equation}
(see also Remark \ref{remainfinity} below). It is common to refer to
the above problem as the minimization of $G_\varepsilon$ under the
unit mass constraint.

Let $\lambda_0>\inf_{\mathbb{R}^2}W(\textbf{y})$ be  uniquely
determined from the relation
\begin{equation}\label{eqlambda0}
\int_{\mathbb{R}^2}^{}\left(\lambda_0-W(\textbf{y})
\right)^+d\textbf{y}=1,
\end{equation}
where throughout this paper we will denote $f^+\equiv \max\{f,0\}$.
The choice of the value one in the above relation is dictated by the
constraint (\ref{eqconstraint}), see also (\ref{eqn2toA+}) below.
We further assume that the region
\begin{equation}\label{eqD0}
\mathcal{D}_0\equiv \{\textbf{y}\in \mathbb{R}^2:\
W(\textbf{y})<\lambda_0\}
\end{equation}
is a  simply connected
 bounded domain,
containing the origin, with smooth boundary $\partial
\mathcal{D}_0$, such that
\begin{equation}\label{eqVnormal}
\frac{\partial W}{\partial \textbf{n}}>0 \ \ \textrm{on}\ \partial
\mathcal{D}_0,
\end{equation}
where $\textbf{n}=\textbf{n}(\textbf{y})$ denotes the outward unit
normal vector to $\partial \mathcal{D}_0$. This last assumption can
be viewed as a non--degeneracy condition. We point out that these
hypotheses admit physically relevant
examples, used to model certain experiments (see the next
subsection). We stress that the simply connectedness assumption
 is assumed  for convenience purposes
only (see Remark \ref{remconnected} below), and so is the fact that
the setting is two-dimensional (see Remark \ref{remuniq} below).

 It follows from \cite[Prop. 1]{seiringerLNF} (see also
\cite{aftalion-jerrard}, \cite{ignat}, \cite{lieb}) that the
functional $G_\varepsilon$ has a unique real valued minimizer
$\eta_\varepsilon>0$ in $\mathcal{H}$ (all complex valued minimizers
are of the form $\eta_\varepsilon e^{i\alpha}$, where $\alpha$ is a
constant). The function $\eta_\varepsilon$ satisfies
\begin{equation}\label{eqlagrange}
-\Delta
\eta_\varepsilon+\frac{1}{\varepsilon^2}\eta_\varepsilon\left(W(\textbf{y})+\eta_\varepsilon^2
\right)=\frac{1}{\varepsilon^2}\lambda_\varepsilon
\eta_\varepsilon,\ \ \eta_\varepsilon>0\ \ \textrm{in}\
\mathbb{R}^2,\  \eta_\varepsilon\to 0\ \ \textrm{as}\ \
|\textbf{y}|\to \infty,
\end{equation}
where $\frac{1}{\varepsilon^2}\lambda_\varepsilon$ is the Lagrange
multiplier, which is also necessarily unique. The point being that
(\ref{eqainfinity}) ensures that minimizing sequences of
$G_\varepsilon$ in $\mathcal{H}$ cannot have their mass escaping at
infinity (see also \cite{rabinowitz}); in fact the imbedding
$\mathcal{H}\hookrightarrow L^2(\mathbb{R}^2,\mathbb{C})$ is compact
(see \cite{ignat}, \cite{zhangZAMP}, or more generally \cite[Lemma
3.1]{bartsch}). (One can also ignore the mass constraint, and
instead minimize the functional
$G_\varepsilon(u)-\frac{\lambda}{2\varepsilon^2}\|u\|^2_{L^2(\mathbb{R})}$,
with $\lambda>\min_{\mathbb{R}^2} W$, in which case the minimizer
would satisfy (\ref{eqlagrange}) with
$\lambda_\varepsilon=\lambda$).

The real issue is the study of the asymptotic behavior of the
minimizer $\eta_\varepsilon$ (or more generally of the critical
points) of $G_\varepsilon$ as the parameter $\varepsilon$ tends to
zero. Following Aftalion and Rivi\`{e}re \cite{aftalionriviere},
letting
\begin{equation}\label{eqA}
 A=\lambda_0-W,
 \end{equation} the functional
$G_\varepsilon$ can be rewritten as
\begin{equation}\label{eqGepsrenorintro}
G_\varepsilon(\eta)=\int_{\mathbb{R}^2}^{} \left\{\frac{1}{2}|\nabla
\eta|^2+\frac{1}{4\varepsilon^2}(\eta^2-A^+)^2+\frac{1}{2\varepsilon^2}A^-\eta^2
\right\}d\textbf{y} + \frac{1}{2\varepsilon^2}\left(
\lambda_0-\frac{1}{2}\int_{\mathbb{R}^2}^{}(A^+)^2d\textbf{y}\right)
\end{equation}
if $\eta \in \mathcal{H}$ is real valued, where
$A^+\equiv\max\{A,0\}$ and $A^-\equiv-\min\{A,0\}$. Let
$G_\varepsilon^1(\eta)$ denote the first integral above. Since
$\eta_\varepsilon$ clearly minimizes $G_\varepsilon^1$ in
$\mathcal{H}$, by constructing a suitable competitor based on
$\sqrt{A^+}$, it is easy to see that
\begin{equation}\label{eqG1intro}
G_\varepsilon^1(\eta_\varepsilon)\leq C |\ln \varepsilon|,
\end{equation}
 for some
constant $C>0$, provided $\varepsilon$ is small (see \cite{alama},
\cite{aftalion-jerrard}, \cite{jerard}, and Remark \ref{remlagrange}
herein). (We remark that the logarithmic term appears because
(\ref{eqD0}), (\ref{eqVnormal}) imply that $\nabla
\left(\sqrt{A^+}\right)$ is not square--integrable near $\partial
\mathcal{D}_0$). Hence, for small $\varepsilon>0$, we have
\begin{equation}\label{eqcompetitor}
\int_{\mathbb{R}^2}^{}
\left\{(\eta_\varepsilon^2-A^+)^2+A^-\eta_\varepsilon^2
\right\}d\textbf{y}\leq C\varepsilon^2|\ln \varepsilon|,
\end{equation}
which suggests that $\eta_\varepsilon^2$ should be close, in some
sense, to $A^+$ as $\varepsilon\to 0$. Indeed, it can be shown that
\begin{equation}\label{eqn2toA+}\eta_\varepsilon\to \sqrt{A^+}\ \
\textrm{uniformly\ in}\  \mathbb{R}^2\  \textrm{as}\ \varepsilon\to
0,
\end{equation}
see the references in Subsection \ref{secknown} below. Therefore,
loosely speaking, the minimizer $\eta_\varepsilon$ develops a steep
\emph{corner layer} along $\partial{D}_0$, as $\varepsilon\to 0$
(recall (\ref{eqD0}), (\ref{eqVnormal})). Note also that
$\sqrt{A^+}$ is the global minimizer of the ``limit'' functional
\begin{equation}\label{eqGminimizer000}
G_\varepsilon^0(\eta)=\int_{\mathbb{R}^2}^{}
\left\{\frac{1}{4\varepsilon^2}\eta^4+\frac{1}{2\varepsilon^2}W(\textbf{y})\eta^2
\right\}d\textbf{y},
\end{equation}
among real functions such that $\|\eta\|_{L^2(\mathbb{R}^2)}=1$ and
$W\eta^2\in L^1(\mathbb{R}^2)$. In the context of Bose-Einstein
condensates, see the following subsection, the function $\sqrt{A^+}$
is known as the \emph{Thomas-Fermi approximation}.

The estimates that are available in the literature  for the
convergence in (\ref{eqn2toA+}), see Subsection \ref{secknown}
below, fail to encapsulate important information which is often
required in interesting applications (see Section
\ref{secmotivation} below). As an illustrative example, let us
mention that a lower bound for $\eta_\varepsilon$, sufficient to
imply that the second variation
$\partial^2G_\varepsilon(\eta_\varepsilon)$ is coercive (this is
easy to prove but hard to estimate), does not seem to be known (in
Remark \ref{remfuscoextension} we will establish a spectral bound
for $\partial^2G_\varepsilon(\eta_\varepsilon)$, and as a matter of
fact one may even calculate sharp constants). Our main goal in this
paper is to provide  crucial details, missing from the known results
that describe the statement (\ref{eqn2toA+}) quantitatively, placing
special emphasis on \emph{how} $\eta_\varepsilon$ converges to the
``singular limit'' $\sqrt{A^+}$ near $\partial \mathcal{D}_0$, as
$\varepsilon\to 0$, and proving that it converges in a self-similar
fashion as conjectured in \cite{aftalionbook}. Although it is a
variational problem, our approach will be based more on partial
differential equation and functional analysis tools. Our treatment
is concise and systematic, and can be used to treat in a unified
manner problems with similar features.

As will be apparent from one glance in the references of this paper
and the following discussion, the current interest in the minimizer
$\eta_\varepsilon$, and in problems that have it on their
background, is phenomenal.
\subsection{Motivation for the current work}\label{secmotivation}
The motivation for the current work is threefold:
\subsubsection{Minimization of a
Gross-Pitaevskii energy describing a Bose-Einstein condensate in a
potential trap} Among the many experiments on Bose-Einstein
condensates (which we abbreviate BEC), one consists in rotating the
trap holding the atoms in order to observe a superfluid behavior:
the appearance of quantized vortices (see the books
\cite{aftalionbook}, \cite{pethick}, \cite{pitaevskibook} and the
references that follow). This takes place for sufficiently large
rotational velocities. On the contrary, at low rotation, no vortex
is detected in the bulk of the condensate. In a BEC, all the atoms
occupy the lowest energy state so that they can be described by the
same complex valued wave function. The latter is at the same time
the macroscopic quantum wave function of the condensate and
minimizes a Gross-Pitaevskii type energy. A vortex corresponds to
zeroes of the wave function with phase around it.
In two dimensions, the Gross-Pitaevskii energy considered in
\cite{aftaliondu}, \cite{aftalion-jerrard}, \cite{ignat},
\cite{ignat2} has the form:
\begin{equation}\label{eqEfunctional}
E_\varepsilon(v)=\int_{\mathbb{R}^2}^{} \left\{\frac{1}{2}|\nabla
v|^2+\frac{1}{4\varepsilon^2}|v|^4+\frac{1}{2\varepsilon^2}W(\textbf{y})|v|^2
-\Omega \textbf{y}^\bot \cdot (iv,\nabla v)\right\}d\textbf{y},\ \
v\in \mathcal{H},
\end{equation}
where $\Omega$ is the angular velocity,
$\textbf{y}=(\textbf{y}_1,\textbf{y}_2)$,
$\textbf{y}^\bot=(-\textbf{y}_2,\textbf{y}_1)$, $\varepsilon>0$ is a
small parameter that corresponds to the Thomas-Fermi approximation
\cite{fermi, thomas}, the trapping potential $W$ belongs in the
class described in the previous subsection, and $(iv,\nabla
v)=iv\nabla v^*-iv^*\nabla v$. It is clear that $\eta_\varepsilon$
is the unique real valued minimizer of $E_\varepsilon$ in
$\mathcal{H}$.

 For mathematical studies in the case where the
condensate has an annular shape, we refer to \cite{alama} and
\cite{corregiCMP}, whilst for studies in a three-dimensional setting
to \cite{alamamontero}, \cite{baldojerrard}, and
\cite{nonlinearanalysisChinese}. For numerics we refer to the review
article \cite{baoREVIEW}.

 The density of the condensate is significant in
$\mathcal{D}_0$ (keep in mind (\ref{eqn2toA+})), which is typically
a disc or an annulus, and gets exponentially small outside of this
domain. The case of harmonic trapping potential
\begin{equation}\label{eqharmonic}W(\textbf{y}_\textbf{1},\textbf{y}_\textbf{2})=\textbf{y}_\textbf{1}^2+\Lambda^2\textbf{y}_\textbf{2}^2,
\end{equation}
for a fixed parameter $0<\Lambda\leq 1$, has been considered in
experiments in \cite{harmonic1, harmonic2}. In recent experiments,
in which a laser beam is superimposed upon the magnetic trap holding
the atoms, the trapping potential $W$ is of a different type
\cite{lazer1, lazer2, lazer3}: \begin{equation}\label{eqlazer}
W(r)=r^2+ae^{-br^2},\ \ \
r^2=\textbf{y}_\textbf{1}^2+\textbf{y}_\textbf{2}^2,\ \ a,b>0.
\end{equation}
(By choosing $a,b$ accordingly, the domain $\mathcal{D}_0$ is either
a disc or an annulus).

The energy $E_\varepsilon$ bears a formal resemblance to the
well-studied Ginzburg-Landau functional
\[
J_\varepsilon(u)=\int_{\mathbb{R}^2}^{} \left\{\frac{1}{2}|\nabla
u|^2+\frac{1}{4\varepsilon^2}\left(|u|^2-1 \right)^2-\Omega
\textbf{y}^\bot \cdot (iu,\nabla u)\right\}d\textbf{y},
\]
 used to model superconductivity \cite{bethuel} (with $\Omega=0$), and superfluidity \cite{alamabronsardmillot}, \cite{serfatyESAIM}.
In their influential monograph \cite{bethuel}, Bethuel, Brezis and
H\'{e}lein have developed the main tools for studying vortices in
``Ginzburg-Landau type'' problems. As we have already pointed out,
the singular behavior of $\sqrt{A^+}$ near $\partial \mathcal{D}_0$
induces a cost of order $|\ln \varepsilon|$ in the energy. This
causes a mathematical difficulty in detecting vortices by energy
methods, since any vortex has precisely the same cost (see
\cite{bethuel}). In other words, the possible presence of vortices
will be hidden by the energetic cost of the corner layer. This
difficulty is common in problems of Ginzburg-Landau type when the
zero Dirichlet boundary condition is imposed (see for instance
\cite{serfatyESAIM}). Fortunately, this difficulty can be surpassed
in an elegant way by an idea that goes back to the work of Lassoued
and Mironescu \cite{lassouedmironescu}, and Andr\'{e} and Shafrir
\cite{andre2}. By a remarkable identity, for any $v$, the energy
$E_\varepsilon(v)$, for any $\Omega$, splits into two parts, the
energy $G_\varepsilon(\eta_\varepsilon)$ of the density profile and
a reduced energy of the complex phase $w = v/\eta_\varepsilon$:
\begin{equation}\label{eqenergysplitting}
E_\varepsilon(v) = G_\varepsilon(\eta_\varepsilon)+F_\varepsilon(w),
\end{equation}
where
\begin{equation}\label{eqFepsilonfuncitonal}
F_\varepsilon(w)=\int_{\mathbb{R}^2}^{}\left\{\frac{\eta_\varepsilon^2}{2}|\nabla
w|^2+\frac{\eta_\varepsilon^4}{4\varepsilon^2}\left(|w|^2-1\right)^2
-\eta^2_\varepsilon\Omega \textbf{y}^\bot \cdot (iw,\nabla w)
\right\}d\textbf{y},
\end{equation}
(see also \cite{ignat}). In particular, the potential $W$ only
appears in $G_\varepsilon$.
 In (\ref{eqenergysplitting}), the term
$G_\varepsilon(\eta_\varepsilon)$ carries the energy of the singular
layer near $\partial{D}_0$, and thus one may detect vortices from
the reduced energy $F_\varepsilon$ by applying the Ginzburg-Landau
techniques to the energy
\[
\tilde{F}_\varepsilon(w)=\int_{\mathbb{R}^2}^{}\left\{\frac{\eta_\varepsilon^2}{2}|\nabla
w|^2+\frac{\eta_\varepsilon^4}{4\varepsilon^2}\left(|w|^2-1\right)^2
\right\}d\textbf{y}.
\]
The difficulty will arise in the small density region of the
condensate, namely $\mathbb{R}^2\backslash\mathcal{D}_0$, where
$\eta_\varepsilon$ is small. This kind of splitting of the energy is
by now standard in the rigorous analysis of functionals such as
$E_\varepsilon$ (see also \cite{lin}). It clearly brings out the
need for the study of the minimizer $\eta_\varepsilon$ of
$G_\varepsilon$, which is the subject of the current work.
  In particular, as will also be apparent from the discussion in
Subsection \ref{secknown} below, estimating $\eta_\varepsilon$ near
$\partial \mathcal{D}_0$ is essential for adapting the powerful
Gamma-convergence techniques, developed for $J_\varepsilon$ (see
\cite{serfatybook} and the references therein), to the study of
$E_\varepsilon$ (concerning issues of vortices, vortex lines
\cite{alamabronsardmillot}, etc). Obtaining these delicate
estimates, \emph{without imposing any symmetry assumptions} on the
trapping potential, is the main contribution of the present paper.

\subsubsection{Semi--classical states of the defocusing nonlinear
Schr\"{o}dinger equation}\label{secscrodmotiv} Elliptic problems of
the form (\ref{eqlagrange})
arise directly when seeking standing wave solutions
\begin{equation}\label{eqphase}
u(\textbf{y},t)=e^{-i\lambda t/\varepsilon}\eta(\textbf{y}),
\end{equation}
for the famous nonlinear Schr\"{o}dinger equation (NLS):
\begin{equation}\label{eqNLS}
i\varepsilon \frac{\partial u}{\partial t}+\varepsilon^2 \Delta
u-W(\textbf{y})u\pm|u|^{q-1}u=0,\ \ \textbf{y}\in \mathbb{R}^N,\
t>0,\ \ u:\mathbb{R}^N\to \mathbb{C},
\end{equation}
where $N\geq 1$ and $q>1$ (in the plus sign case, the potential $W$
may vary from that described previously). (See also (\ref{eqvortex})
below). For small $\varepsilon>0$, these standing-wave solutions are
referred to as semi-classical states. The plus sign in (\ref{eqNLS})
gives rise to the focusing NLS (attractive nonlinearity), while the
minus sign to the defocusing NLS (repulsive nonlinearity) which is
also known as the Gross--Pitaevskii equation (GP). It is also quite
common to use the name of the Gross-Pitaevskii equation if
(\ref{eqNLS}) has a nonzero potential $W$, and the name of the
nonlinear Schr\"{o}dinger equation if $W\equiv 0$. Keep in mind that
potentials of quadratic growth, as $|\textbf{y}|\to \infty$, are the
highest order potentials for local well-posedness of (\ref{eqNLS}),
see \cite{oh}.

At low enough temperature, neglecting the thermal and quantum
fluctuations, a Bose condensate can be represented by a complex wave
function $u(\textbf{y}, t)$ that obeys the dynamics of the NLS
equation, see the excellent review article \cite{kevrekidisREVIEW}
and the references that follow. In particular, solutions of
(\ref{eqlagrange}) provide, via (\ref{eqphase}), standing wave
solutions for the GP equation with $N=2,\ q=3$.  Let us mention that
the minimizer of $G_\varepsilon$, considered in  the entire space or
in  a bounded domain with zero boundary conditions (as in Remark
\ref{remdirichlet} below), also plays an important role in the study
of multi-component BECs (see \cite{karalikevrekidisefr}, \cite{liu}
and the references therein); in the dynamics of vortices confined in
$\mathcal{D}_0$ under the flow of the Gross-Pitaevskii equation (see
\cite{kevrekidisvortices}); in the construction of traveling wave
solutions with a stationary or traveling vortex ring to the
Gross-Pitaevskii equation (see \cite{vortexringsARMA},
\cite{weipitaevskii}, \cite{weivortex-2}); in the study of excited
states of the GP equation (see \cite{kevrekidisexcitedradial},
\cite{pelinovskynonlinear}), and in Bose-Einstein condensates with
weak localized impurities (see \cite{frantzeskakisPurity}).

In considering typical BEC experiments and in exploring the
unprecedented control of the condensates through magnetic and
optical ``knobs'', a mean-field theory is applied to reduce the
quantum many--atom description to a scalar nonlinear Schr\"{o}dinger
equation  (see \cite{lieb}). The NLS equation is a well established
model in optical and plasma physics as well as in fluid mechanics,
where it emerges out of entirely different physical considerations
\cite{ablowitz, sulem}. In particular, for instance in optics, it
emerges due to the so-called Kerr effect, where the material
refractive index depends linearly on the intensity of incident
light. The widespread use of the NLS equation stems from the fact
that it describes, to the lowest order, the nonlinear dynamics of
envelope waves.

 Ground state solutions of the NLS are
standing wave solutions, of the form (\ref{eqphase}), such that
$\eta$ is positive, $\eta \in W^{1,2}(\mathbb{R}^N)$, and satisfies
\begin{equation}\label{eqground}
\varepsilon^2\Delta
\eta-\left(W(\textbf{y})-\lambda\right)\eta\pm|\eta|^{q-1}\eta=0\ \
\textrm{in}\ \mathbb{R}^N,\ \eta\to 0\ \textrm{as}\ |\textbf{y}|\to
\infty.
\end{equation}
The condition $u\in W^{1,2}(\mathbb{R}^N)$ is required to obtain
solutions with physical interest.
 (Sometimes we will refer to
positive solutions $\eta$ of $(\ref{eqground})$ as  ground states of
(\ref{eqNLS})). In the subcritical case where $1<q<\frac{N+2}{N-2}$
if $N\geq 3$, $q>1$ if $N=1,2$, ground states of the defocusing
equation $(\ref{eqground})_-$ correspond to global minimizers of
$\mathcal{G}_-$ (these are nontrivial if $\varepsilon>0$ is
sufficiently small \cite{ignat}, see also \cite[Example
5.11]{malchiodicambridge} and \cite[Lemma 2.1]{delPinoscrew}), while
ground states of the focusing equation $(\ref{eqground})_+$
correspond to mountain passes of $\mathcal{G}_+$, where
\begin{equation}\label{eqmountain}
\mathcal{G}_{\pm}(\eta)=\int_{\mathbb{R}^N}^{}\left\{\frac{\varepsilon^2}{2}|\nabla
\eta|^2+\left(W(\textbf{y})-\lambda \right)\frac{\eta^2}{2}\mp
\frac{|\eta|^{q+1}}{q+1}\right\}d\textbf{y},
\end{equation}
in $W^{1,2}(\mathbb{R}^N)$ with
$\int_{\mathbb{R}^N}^{}W(\textbf{y})\eta^2d\textbf{y}<\infty$ (see
for instance \cite{malchiodibook}, \cite{rabinowitz}).

In the focusing case, following the pioneering work of Floer and
Weinstein \cite{floer}, there have been enormous investigations on
\emph{spike layer} solutions for problem $(\ref{eqground})_+$, for
small $\varepsilon>0$, typically under the conditions
$\lambda<\inf_{\mathbb{R}^N} W(\textbf{y})<\liminf_{|\textbf{y}|\to
\infty}W(\textbf{y})$, and $1<q<\frac{N+2}{N-2}$ if $N\geq 3$, $q>1$
if $N=1,2$. The ``critical'' case, where
$\inf_{\mathbb{R}^N}W(\textbf{y})=\lambda$, has also received
 attention, see \cite{byeon}. Actually, the latter case is related to the discussion following Definition \ref{defmaximal}
below.
 We refer the interested reader to \cite{malchiodibook}, \cite{del
pino cpam}, and the references therein.

In the defocusing case $(\ref{eqground})_-$, assuming that $W$
satisfies the assumptions of the previous subsection with $\lambda$
in place of $\lambda_0$, we see that the corresponding \emph{limit
algebraic equation}, obtained by formally letting $\varepsilon=0$ in
$(\ref{eqground})_-$, has the compactly supported continuous
solution
\begin{equation}\label{eqeta0}
\eta_0=\left[(\lambda-W)^+ \right]^{\frac{1}{q-1}}.
\end{equation}
Obviously, there is also the solution $-\eta_0$, and the trivial
one. In fact, if $q> 2$ then the ``singular limit'' $\eta_0$ is
merely H\"{o}lder continuous with exponent $1/(q-1)$ (recall
(\ref{eqD0}), (\ref{eqVnormal})). In particular, it is easy to see
that if $q\geq 3$ then $\nabla \eta_0$ is not square-integrable near
$\partial \mathcal{D}_0$.
 If $q=2$ and
$\eta>0$, then $(\ref{eqground})_-$ becomes the well known scalar
logistic equation \cite{cantell}, and  $\eta_0$ is Lipschitz
continuous. If $1<q<2$, then $\eta_0$ is at least differentiable. In
the language of bifurcation theory \cite{ioss}, the solution set of
the corresponding limit algebraic equation to $(\ref{eqground})_-$
undergoes a supercritical pitchfork bifurcation at $\partial
\mathcal{D}_0$ if $q>2$; a transcritical bifurcation if $q=2$; a
subcritical pitchfork bifurcation if $1<q<2$. Notice also that
$\eta_0$ is asymptotically stable (as an approximate equilibrium)
with respect to the parabolic dynamics that correspond to
(\ref{eqground}). The question whether $\eta_0$ perturbs, for small
$\varepsilon$, to a solution $\eta_\varepsilon$ of
$(\ref{eqground})_-$, and keeping track of the maximal regularity
available (to be defined in a moment), is a source of current
mathematical interest. Note that, if $q\geq 2$, such a solution
would have a \emph{corner layer} along $\partial\mathcal{D}_0$.
 The bifurcation that
occurs at $\partial\mathcal{D}_0$ takes the problem off from the
classical setting, where the roots of the corresponding algebraic
equation (for fixed $\textbf{y}$) are non-intersecting (see
\cite{delpinoCPDE}).

The following definition is adapted from \cite{cafaroqjofr}:
\begin{definition}\label{defmaximal}
Let $\alpha\in (0,1]$ be the largest number such that
$\|\eta_0\|_{C^{0,\alpha}(\mathbb{R}^N)}<\infty$, we say that a
solution $\eta_\varepsilon$ has maximal H\"{o}lder regularity if
$\|\eta_\varepsilon\|_{C^{0,\alpha}(\mathbb{R}^N)}$ remains
uniformly bounded as $\varepsilon\to 0$.
\end{definition}
If $N=1$, problem $(\ref{eqground})_-$ can be rewritten as a
homoclinic connection problem for a 3-dimensional slow--fast system
of ordinary differential equations \cite{jones}. At the points that
correspond to $\partial \mathcal{D}_0$ we have loss of normal
hyperbolicity of the corresponding slow manifold, due to a pitchfork
or transcritical bifurcation. This fact \emph{prohibits} the use of
standard geometric singular perturbation theory \cite{fenichel,
jones} in order to deduce the persistence of the ``singular
homoclinic orbit'' $\eta_0$, for small $\varepsilon>0$.
 At
present, much work in geometric singular perturbation theory deals
with such situations, often using the ``blowing-up'' construction
\cite{krupaszmolyanMN}. This approach has been successfully applied
recently in \cite{schecter-sourdis} in a heteroclinic connection
problem for a 4-dimensional slow-fast Hamiltonian system, sharing
similar features with $(\ref{eqground})_-$ with $q=3$, arising from
the study of crystalline grain boundaries (see also \cite{alikakos,
fifeunpublished, sourdis-fife}). Let us mention that there is an
abundance of non-hyperbolic points in applications,  see for
instance \cite{krupaszmolyanMN}, that have been successfully treated
using this approach. In particular, singularly perturbed
one-dimensional second-order elliptic systems, involving loss of
normal hyperbolicity, arise in the study of the Dafermos
regularization for singular shocks \cite{dafermosshocks},
\cite{schecterdafermos}. Although elegant, the arguments of this
approach are intrinsically one-dimensional.

 Elliptic systems where the
singular limit has merely H\"{o}lder or Lipschitz regularity
typically describe phase separation or spatial segregation, and have
attracted a lot of current mathematical research, see
\cite{cafaroqjofr, teraccini,weiweth} and the references therein.
These  type of problems have been tackled in the latter references
either by weak convergence arguments, yielding weak convergence in
the Sobolev space $W^{1,2}$ and strong in $L^2$ (see also
\cite{dancer-hilhorst}), or the method of upper and lower solutions,
yielding uniform convergence (see also \cite{butuzov jde2},
\cite{hutson-lou-mischaikow-jde}, \cite{iida}), as $\varepsilon\to
0$.
 The question of maximal
regularity of the convergence to the singular limit profile is then
addressed using a delicate analysis, based on monotonicity
properties, blow-up techniques and Liouville-type theorems (see also
Remark \ref{remCAffarelli} below). \emph{To the best of our knowledge, for
 these systems, the maximal regularity property has only been proven
 in one--dimensional cases, see \cite{berestycki-wei2012}.}
In the case at hand, since $\nabla\eta_0$ is not square-integrable
near $\partial \mathcal{D}_0$ if $q\geq 3$, it is not clear how to
use standard weak convergence arguments. Furthermore, it seems to be
hard to construct a good pair of upper and lower solutions
(especially) near $\partial \mathcal{D}_0$. A motivation for the
current work is to show that the perturbation approach to such
problems, we initiated in
\cite{karalisourdisradial,karalisourdisresonance}, can be adapted to
 treat problem $(\ref{eqground})_-$ with general potential.
 \emph{We emphasize that the perturbation method seems to be the only
 one available at the moment
 that yields the maximal regularity property in higher-dimensional singular perturbation problems.}
 (See the main theorem of \cite{karalisourdisresonance}, and Corollary \ref{cormax}
 herein).
 Among other advantages of the perturbation approach is that it
 provides finer estimates which, as we already stated, imply the maximal regularity
 property, it be applied in the study of non-minimizing solutions in
 systems of equations which in general  lack the maximum principle (see assumption (B4) in \cite{butuzovJDE2}),
 and in supercritical problems that cannot be treated variationally (see for example Remark \ref{remhastingsUniq}). Most importantly, for the problem at hand, it
 applies equally well without the assumption of radial symmetry on the
 equation, see the discussion in the following subsection.

\subsubsection{Applications to related problems}\label{secandre} A strong motivation
behind the current work is  the possibility of adapting our approach
to treat other interesting problems with similar features.

Variational problems of a similar type to (\ref{eqEfunctional}),
with spatially varying coefficients, have also been introduced to
model vortex pinning due to material inhomogeneities in a
superconductor (see the references in \cite{alamaPinning},
\cite[Sec. 6]{Dureview} and \cite[Subsec. 14.1.4,
14.3.5]{serfatybook}). In \cite{andrebaumanphilips}, the authors
considered the case where the corresponding pinning potential
$\mathcal{W}(\textbf{y})$ is nonnegative but vanishes at a finite
number of points (see also \cite{sun}). Minimizers such that
$|v_\varepsilon|\to \sqrt{\mathcal{W}}$ in $W^{1,2}$, as
$\varepsilon\to 0$, were analyzed by variational methods. In their
result, it was important that $\sqrt{\mathcal{W}}$ is in $W^{1,2}$,
which is not the case here (also recall the previous discussion
concerning weak convergence methods). In that context, the minimizer
has corner layer behavior at points rather than curves. It would be
of interest to study this type of problems from the perturbation
viewpoint of the present paper which, in particular, does not
require that the singular limit belongs in $W^{1,2}$.

A singularly perturbed elliptic equation of the form
(\ref{eqEllipticIntro}) below, considered in a bounded domain with
zero boundary conditions, where the corresponding limit algebraic
equation admits a fold bifurcation at the boundary of the domain,
appears in the proof of the Lazer--Mckenna conjecture for a
superlinear elliptic problem of Ambrosetti-Prodi type, see
\cite{dancer-lazer}, \cite{danceryanCrtitic}. In that case, the
corresponding minimizer (without the mass constraint) has a steep
corner layer along the boundary of the domain, see also the old
paper \cite{turcotte}. This situation is qualitatively similar to
the problem considered here. In \cite{dancer-lazer}, by adapting
variational techniques from \cite{danceryanCVPDE}, the behavior of
the minimizer, as $\varepsilon\to 0$, was estimated in compact sets
away from the boundary of the domain. We believe that, employing the
perturbation techniques of the present paper, one can obtain fine
estimates for the minimizer \emph{all the way up to the boundary}.
In turn, these could potentially lead to the construction of new
type of solutions on top of the minimizer, for example solutions
having small peaks near the boundary, or bifurcating from symmetry
(as in \cite{karalisourdisradial,karalisourdisresonance}). Another
reason for developing  perturbation arguments for these problems is
that unstable solutions are hard to find or describe accurately
through purely variational methods, see Remark \ref{remhastingsUniq}
for more details.

\subsection{Known results}\label{secknown} In this
subsection we gather some  known properties of the real valued
minimizer $\eta_\varepsilon$ of $G_\varepsilon$ in $\mathcal{H}$,
under the assumptions described in Subsection \ref{secproblem}.
These have been studied in various contexts (see for instance
\cite{alama,aftalion-jerrard,corregiCMP,ignat,monterohodge}). As we
will see in this paper, some are actually far from optimal.

The corresponding Lagrange multiplier satisfies, for small
$\varepsilon>0$, the estimate:
\begin{equation}\label{eqlambdaepsilon}
|\lambda_\varepsilon-\lambda_0|\leq C |\ln
\varepsilon|^\frac{1}{2}\varepsilon.
\end{equation}

The real valued minimizer $\eta_\varepsilon$ satisfies the following
estimates:
\begin{equation}\label{eqmontero}
0<\eta_\varepsilon\leq \sqrt{A^+}+C\varepsilon^\frac{1}{3}\ \
\textrm{in}\ \mathbb{R}^2,
\end{equation}
\begin{equation}\label{eqJ1} \eta_\varepsilon(\textbf{y})\leq
C\varepsilon^\frac{1}{3}\exp\left\{-c\varepsilon^{-\frac{2}{3}}\textrm{dist}(\textbf{y},\partial\mathcal{D}_0)\right\}\
\ \textrm{in}\ \mathbb{R}^2\backslash\mathcal{D}_0,
\end{equation}
where $\textbf{y}\mapsto \textrm{dist}(\textbf{y},\partial
\mathcal{D}_0)$ denotes the Euclidean distance in $\mathbb{R}^2$
from $\textbf{y}$ to $\partial \mathcal{D}_0$,
\begin{equation}\label{eqJ2}
|\eta_\varepsilon-\sqrt{A^+}|\leq
C\varepsilon^\frac{1}{3}\sqrt{A^+}\ \ \textrm{if}\ \textbf{y}\in
\mathcal{D}_0\ \textrm{and} \
\textrm{dist}(\textbf{y},\partial\mathcal{D}_0)\geq\varepsilon^\frac{1}{3},
\end{equation}
and
\begin{equation}\label{eqJ3}
\|\nabla \eta_\varepsilon\|_{L^\infty(\mathbb{R}^2)}\leq
C\varepsilon^{-1},
\end{equation}
for some constants $c,C>0$, if $\varepsilon$ is sufficiently small.
 Relation (\ref{eqlambdaepsilon}) follows from
(\ref{eqG1intro}), see \cite{aftalion-jerrard}, \cite{ignat}.
Relations (\ref{eqmontero}), (\ref{eqJ1}) have been shown in
\cite{ignat} (see also \cite[Lemma B.1]{monterohodge}) by
constructing a suitable upper-solution of (\ref{eqlagrange}), and
using the uniqueness of positive solutions of the latter equation in
bounded domains with zero boundary conditions \cite{brezis-oswald}.
Relation (\ref{eqJ2}) can be traced back to \cite{alama}, and
follows from lower-- and
 upper--solution arguments in equation (\ref{eqlagrange}) based on
\cite{andre2} (also keep in mind Remark 3 in \cite{alama}). Note
that, in particular, estimates (\ref{eqmontero})--(\ref{eqJ2}) yield
(\ref{eqn2toA+}). Lastly, estimate (\ref{eqJ3}) on the gradient
follows from the equation and a Gagliardo-Nirenberg type inequality
as in \cite{bethuelCVPDE} (see Lemma \ref{lembrezis} herein).

In the special case where the potential trap $W$ is additionally
assumed to be radially symmetric, it follows from its uniqueness
that the real valued minimizer $\eta_\varepsilon>0$ of
$G_\varepsilon$ in $\mathcal{H}$ is also radially symmetric. In
particular, if $\mathcal{D}_0$ is a ball of radius $R$, it has been
shown recently in \cite{aftalion-jerrard} that
\begin{equation}\label{eqJ4}
\eta_\varepsilon'(r)\leq 0\ \ \textrm{in}\ (R-\delta_0,R+\delta_0),
\end{equation}
for some small constant $\delta_0>0$, if $\varepsilon$ is small. The
radial symmetry was used heavily by the authors of
\cite{aftalion-jerrard} in order to establish (\ref{eqJ4}), using a
maximum principle due to Berestycki, Nirenberg, and Varadhan
\cite{berestycki} (see also the discussion following
(\ref{eqberestycki}) herein) together with an
intersection--comparison type of argument, mostly taking advantage
of (\ref{eqVnormal}). The importance of the positivity of
$\eta_\varepsilon$ and the radial symmetry of $W$ in deriving
(\ref{eqJ4}) can be naively seen from the following consideration.
If $W'(r)\geq 0$ for all $r>0$, using that $\eta_\varepsilon>0$ and
the method of moving planes \cite{gidas}, we can infer that
$\eta_\varepsilon'(r) \leq 0$ for all $r>0$ (see also Proposition
2.1 in \cite{lieb} for an approach via a radially-symmetric
rearrangement argument which takes advantage of the minimizing
character of $\eta_\varepsilon$). We point out that
it is not clear how the aforementioned arguments apply in the case
where $W$ is radially symmetric but the set $\{W-\lambda_0<0\}$ is
an annulus, considered in \cite{alama}, \cite{alamamontero}.
Relation (\ref{eqJ4}) was used in an essential way in
\cite{aftalion-jerrard} for estimating uniformly  the auxiliary
function
\begin{equation}\label{eqfepsilon}
f_\varepsilon(r)\equiv\xi_\varepsilon(r)/\eta_\varepsilon^2(r),\
\textrm{where}\
\xi_\varepsilon(r)\equiv\int_{r}^{\infty}s\eta^2_\varepsilon(s)ds,
\end{equation}
near the circle $|\textbf{y}|=R$. The function $f_\varepsilon$ plays
a crucial role in the study of the functional $E_\varepsilon$, see
\cite{aftalion-jerrard} and Section \ref{secfepsilon} below.
(Actually, estimate (\ref{eqJ4}) was needed in a region of the form
$(R-\delta, R+C\varepsilon^\frac{2}{3})$ for some constants $\delta,
C>0$). Making use of the previously mentioned estimates on
$f_\varepsilon$ near the boundary of $\mathcal{D}_0$, and of some
new estimates away from $\mathcal{D}_0$,  the authors of
\cite{aftalion-jerrard} proved that if the angular velocity $\Omega$
is below a critical speed
\[\Omega_c=\omega_0|\ln \varepsilon|-\omega_1\ln|\ln \varepsilon|,\]
for some constants $\omega_0,\ \omega_1>0$, and $\varepsilon$ is
sufficiently small, then the rotation has absolutely no effect on
the minimizer. In other words, all minimizers $v_\varepsilon$ of
$E_\varepsilon$ in $\mathcal{H}$ are of the form
$v_\varepsilon=\eta_\varepsilon e^{i\alpha}$ in $\mathbb{R}^2$,
where $\alpha$ is a constant. In particular, at low velocity, there
are no vortices in the condensate. We remark that this last
assertion was previously known to hold true  only in the bulk of the
condensate, see \cite{ignat}.
It was left as an interesting open problem in
\cite{aftalion-jerrard} to see to what extent their analysis
continues to hold if the assumption of radial symmetry on $W$ is
dropped (see also Open Problem 8.1 in \cite{aftalionbook}, and the
open questions in the presentation \cite{gallo}). Our results on the
minimizer $\eta_\varepsilon$, which hold \emph{without any symmetry
assumption} on $W$, may represent a major step in the answering of
this question.

Let us close this subsection by mentioning that the case where the
potential is homogeneous of some order $s > 0$, i.e., $W(\lambda
\textbf{y}) = \lambda^sW(\textbf{y})$ for all $\lambda>0,\
\textbf{y}\in \mathbb{R}^2$ (see (\ref{eqharmonic}) for an example
with $s=2$), and locally H\"{o}lder continuous has been studied in
the work of E. H. Lieb and his collaborators in \cite{lieb},
\cite{liebBook}. By employing scaling and variational arguments, it
was shown in the latter references that, as $\varepsilon\to 0$, the
minimizer $\eta_\varepsilon$ converges to $\sqrt{A^+}$ in the strong
$L^2(\mathbb{R}^2)$ sense. In the special case of the model harmonic
potential, described by (\ref{eqharmonic}) with $\Lambda=1$, a
complete analysis has been carried in \cite{pelinovsky} (see the
discussion following the statement of our main theorem for more
details). By generalizing the divergence-free WKB method, a
uniformly valid approximation for the condensate density of an
ultra-cold Bose gas confined in a harmonic trap, that extends into
the classically forbidden region (near and outside of $\partial
\mathcal{D}_0$), has been established very recently in
\cite{salman}.

\subsection{Statement of the main result}

 In order to state our main
result, we need some definitions:

Let
\begin{equation}\label{eqaepsilon}
a_\varepsilon(\textbf{y})=\lambda_\varepsilon-W(\textbf{y}),\ \
\textbf{y}\in \mathbb{R}^2.
\end{equation}
By virtue of (\ref{eqD0}), (\ref{eqVnormal}),
(\ref{eqlambdaepsilon}), and the implicit function theorem
\cite{ambrozetti-proddi}, the domain $\mathcal{D}_0$ perturbs
smoothly, for $\varepsilon\geq 0$ small, to a simply connected
domain $\mathcal{D}_\varepsilon$ such that
\begin{equation}\label{eqa>b}
a_\varepsilon>0\ \ \textrm{in}\ \mathcal{D}_\varepsilon,\ \
a_\varepsilon<0\ \ \textrm{in}\
\mathbb{R}^2\backslash\bar{\mathcal{D}_\varepsilon},\ \
\frac{\partial a_\varepsilon}{\partial \nu_\varepsilon}<-c\ \
\textrm{on}\
\partial \mathcal{D}_\varepsilon,
\end{equation}
for some constant $c>0$, where
$\nu_\varepsilon=\nu_\varepsilon(\textbf{y})$ denotes the outward
unit normal to $\partial \mathcal{D}_\varepsilon$. Let
$\Gamma_\varepsilon$ be the simple, smooth closed curve defined by
$\partial \mathcal{D}_\varepsilon$, and let $\ell_\varepsilon=
|\Gamma_\varepsilon|$ denote its total length. Note that
$\Gamma_\varepsilon$ is inside of an $\mathcal{O}(|\ln
\varepsilon|^\frac{1}{2}\varepsilon)$--tubular neighborhood of
$\partial \mathcal{D}_0$.
We consider the natural parametrization
$\gamma_\varepsilon=\gamma_\varepsilon (\theta)$ of
$\Gamma_\varepsilon$ with positive orientation, where $\theta$
denotes an arc length parameter measured from a chosen point of
$\Gamma_\varepsilon$. Slightly abusing notation, we let
$\nu_\varepsilon(\theta)$ denote the outward unit normal to
$\Gamma_\varepsilon$ (as in (\ref{eqa>b})). Points $\textbf{y}$ that
are $\delta_0$-close to $\Gamma_\varepsilon$, for sufficiently small
$\delta_0>0$ (independent of small $\varepsilon$), can be
represented in the form
\begin{equation}\label{eqFermi}\textbf{y}=\gamma_\varepsilon(\theta)+t\nu_\varepsilon(\theta),\ \ |t | < \delta_0,\ \ \theta\in [0,\ell_\varepsilon),
\end{equation}
 where the map
$\textbf{y}\mapsto (t, \theta)$ is a local diffeomorphism (see
\cite[Sec. 14.6]{Gilbarg-Trudinger}). Note that $t<0$ in
$\mathcal{D}_\varepsilon$. Abusing notation, frequently we will
denote points $\textbf{y}$ near $\Gamma_\varepsilon$ plainly by
their image $(t,\theta)$ under the above mapping.
 From (\ref{eqaepsilon}), (\ref{eqa>b}), we have
\begin{equation}\label{eqbt>at}
-(a_\varepsilon)_t(0,\theta)=W_t(0,\theta)\geq c,\ \ \ \theta\in
\left[0,\ell_\varepsilon\right),
\end{equation}
for some constant $c>0$ and small $\varepsilon$.
 We define
\begin{equation}\label{eqbita}
\beta_\varepsilon(\theta)=\left(-a_t(0,\theta)\right)^\frac{1}{3}>0,\
\ \ \theta\in [0,\ell_\varepsilon).
\end{equation}
It might be useful to point out that for the harmonic potential,
described in (\ref{eqharmonic}), one can derive explicitly that
\[
\left[\beta_\varepsilon(\theta)\right]^3=2\sqrt{\lambda_\varepsilon}\Lambda\sqrt{\Lambda^{-2}\cos^2\left(
\frac{2\pi}{\ell_\varepsilon}\theta\right)+\sin^2\left(
\frac{2\pi}{\ell_\varepsilon}\theta\right)},\ \ \ \theta\in
[0,\ell_\varepsilon).
\]

We will also make use of the \emph{Hastings-McLeod solution}
\cite{hastings} of the Painlev\'{e}--II equation \cite{fokas},
namely the unique solution $V$ of the boundary value problem:
\begin{equation}\label{eqpanintro}
v_{xx}-v(v^2+x)=0,\ \ \ x\in \mathbb{R},\ \ v(x)-\sqrt{-x}\to 0,\
x\to -\infty;\ \ v(x)\to 0,\ x\to \infty.
\end{equation}
It is useful, in relation with (\ref{eqJ4}), to point out here that
$V_x<0,\ x\in \mathbb{R}$. The importance of the Hastings-Mcleod
solution is that it will ``lead'' the minimizer $\eta_\varepsilon$
across $\partial \mathcal{D}_0$.

We can now state our main result:

 \begin{thm}\label{thmmain}
If $\varepsilon>0$ is sufficiently small, the unique real valued
minimizer of $G_\varepsilon$ in $\mathcal{H}$ satisfies
\begin{equation}\label{eqestim1}
\eta_{\varepsilon}(\textbf{y})=\varepsilon^\frac{1}{3}\beta_\varepsilon(\theta)V\left(\beta_\varepsilon(\theta)\frac{t}{\varepsilon^\frac{2}{3}}
\right)+\mathcal{O}(\varepsilon+|t|^\frac{3}{2})
\end{equation}
uniformly in $\left\{-d\leq t \leq 0,\ \ \theta\in
[0,\ell_\varepsilon) \right\}$,
\begin{equation}\label{eqestim1+}
\eta_{\varepsilon}(\textbf{y})=\varepsilon^\frac{1}{3}\beta_\varepsilon(\theta)V\left(\beta_\varepsilon(\theta)\frac{t}{\varepsilon^\frac{2}{3}}
\right)+\mathcal{O}(\varepsilon e^{-c\varepsilon^{-\frac{2}{3}}t})
\end{equation}
uniformly in $\left\{0\leq t \leq d,\ \ \theta\in
[0,\ell_\varepsilon) \right\}$, where $c,d>0$ are some small
constants.

Given $D>0$, if $\varepsilon>0$ is sufficiently small, we have
\begin{equation}\label{eqestim2}
\eta_\varepsilon(\textbf{y})
-\sqrt{\lambda_\varepsilon-W(\textbf{y})}=\mathcal{O}\left(\varepsilon^2|t|^{-\frac{5}{2}}
\right)
\end{equation}
uniformly in $\left\{-d \leq t\leq -D \varepsilon^\frac{2}{3},\ \
\theta\in [0,\ell_\varepsilon) \right\}$,
\begin{equation}\label{eqestim3}
\eta_\varepsilon(\textbf{y})-\sqrt{\lambda_\varepsilon-W(\textbf{y})}=\mathcal{O}(\varepsilon^2)
\end{equation}
uniformly in $\mathcal{D}_\varepsilon \backslash \left\{-d < t< 0,\
\ \theta\in [0,\ell_\varepsilon) \right\}$,
 and
\begin{equation}\label{eqestim2+}
0<\eta_\varepsilon(\textbf{y}) \leq C
\varepsilon^\frac{1}{3}\exp\{-c\varepsilon^{-\frac{2}{3}}\textrm{dist}(\textbf{y},\partial\mathcal{D}_0)\}
\end{equation}
in $\mathbb{R}^2 \backslash \mathcal{D}_\varepsilon$, for some
constants $c,C>0$ independent of $\varepsilon$, where
$\mathcal{O}(\cdot)$ denotes Landau's symbol (see Section
\ref{secnotation} for the precise definition).

The potential of the associated linearized operator
\begin{equation}\label{eqLreal}
\textbf{L}_\varepsilon(\varphi)\equiv\varepsilon^2\Delta
\varphi-\left(3\eta_\varepsilon^2(\textbf{y})+W(\textbf{y})-\lambda_\varepsilon
\right)\varphi,
\end{equation}
satisfies
\begin{equation}\label{equappotentialthmmain}
3\eta_\varepsilon^2(\textbf{y})+W(\textbf{y})-\lambda_\varepsilon\geq
\left\{
\begin{array}{ll}
  c \varepsilon^\frac{2}{3}+c|t|, &  \textrm{if} \ \ |t|\leq \delta, \\
    &   \\
  c+c|\textbf{y}|^p,  & \textrm{otherwise}.
\end{array}\right.
\end{equation}

The Lagrange multiplier $\varepsilon^{-2}\lambda_\varepsilon$
satisfies
\begin{equation}\label{eqlambdaepsilonnew}
\lambda_\varepsilon-\lambda_0=\mathcal{O}(|\ln
\varepsilon|\varepsilon^2),
\end{equation}
while the energy of $\eta_\varepsilon$ satisfies
\begin{equation}\label{eqenergysharp}
G_\varepsilon(\eta_\varepsilon)=\left(\frac{\lambda_0}{2}-\frac{1}{4}\int_{\mathbb{R}^2}^{}(A^+)^2d\textbf{y}\right)\varepsilon^{-2}
+\frac{1}{12}\left(\int_{0}^{\ell_0}\beta_0^3(\theta)d\theta
\right)|\ln \varepsilon|+\mathcal{O}(1),
\end{equation}
as $\varepsilon\to 0$.
\end{thm}

The main highlight of our result is that we \emph{rigorously} prove
that, close to $\partial \mathcal{D}_0$, the minimizer
$\eta_\varepsilon$ behaves like (\ref{eqestim1}). We emphasize that
 the rigorous derivation of the Pailev\'{e}-II equation from (\ref{eqlagrange}) was an unsettled
open problem, see \cite[pg. 13]{aftalionbook} (see also the
discussion below on the recent paper \cite{pelinovsky}). In turn, as
was noted in Section 8.1.3 of the book \cite{aftalionbook}, this
information is required as a stepping stone towards the open problem
mentioned in Subsection \ref{secknown} (see also Open Problem 8.1 in
\cite{aftalion-jerrard}), in order to obtain a lower bound for
$\eta_\varepsilon$ in $\mathbb{R}^2\backslash \mathcal{D}_0$. In the
current paper, under the additional assumption that $W_{tt}$ is
strictly positive on $\partial \mathcal{D}_0$, we contribute towards
this direction by obtaining a \emph{sharp} lower bound in the
strip-like domain of $\mathbb{R}^2\backslash \mathcal{D}_0$
described by $\textrm{dist}(\mathbf{y},\partial \mathcal{D}_0)\ll
\varepsilon^\frac{2}{5}$ (see Remark \ref{remexpimproved} below).

\emph{We believe that our result opens the road for the rigorous
description of the Painlev\'{e} region in recent experiments on
three-dimensional Bose-Einstein condensates where a laser beam,
modeled by a cylinder along the $z$ direction, is translated in the
$x$ direction along the condensate (see
\cite{aftaliondu-painleve,aftalionpainleve}), and to understand the
superfluid flow around an obstacle (see \cite{aftalionbook}  and the
references therein).}

As we have already mentioned, our proof is based on  perturbation
arguments rather than variational ones or the method of upper and
lower solutions (as was hoped for in \cite[Sec.
8.3.1]{aftalionbook}).
It relies mainly on the existence and asymptotic stability of the
Hastings-McLeod solution (in the usual sense) in order to construct
a solution $\textbf{u}_\varepsilon$ of problem (\ref{eqlagrange}),
``close'' to $\sqrt{A^+}$, for small $\varepsilon>0$ (with  the
Lagrange multiplier $\lambda_\varepsilon$ treated as a known
coefficient). Then, using the fact that the latter problem has a
unique solution which follows from ideas of Brezis and Oswald
\cite{brezis-oswald} (see Remark \ref{remuniq} below), namely
$\eta_\varepsilon$, we infer that $\textbf{u}_\varepsilon\equiv
\eta_\varepsilon$. Actually, with the obvious modifications, an
analogous result  holds true for the minimizer of $G_\varepsilon$ in
arbitrary dimensions. Furthermore, our method of proof can be
adapted to treat the case where
$G_\varepsilon(u)-\frac{\lambda}{2}\|u\|^2_{L^2(\mathbb{R}^2)}$ is
minimized in $W^{1,2}_0(\mathbb{D})$, where $\mathbb{D}$ is a
bounded annular domain such that $W=\lambda$ on the outer or inner
part of its boundary, as in \cite{alama}, \cite{alamamontero}, see
Remark \ref{remdirichlet} below. For further generalizations we
refer to Remarks \ref{remconnected}, \ref{remainfinity} below.

It follows from the above estimates that the convergence of
$\eta_\varepsilon$ to $\sqrt{A^+}$ is the most regular possible (see
Corollary \ref{cormax} below). If we assume that $W_{tt}\geq c>0$ on
$\Gamma$, then we will show in Proposition \ref{prorefine} that
bound (\ref{eqestim1+}) can be replaced by
\[
\eta_{\varepsilon}(\textbf{y})=\varepsilon^\frac{1}{3}\beta_\varepsilon(\theta)V\left(\beta_\varepsilon(\theta)\frac{t}{\varepsilon^\frac{2}{3}}
\right)\left[1+\mathcal{O}(\varepsilon^\frac{2}{3})\left(\frac{t}{\varepsilon^\frac{2}{3}}\right)^\frac{5}{2}
\right].
\]
 Estimate
(\ref{eqlambdaepsilonnew}) improves (\ref{eqlambdaepsilon}), and was
previously established in \cite{ignat} in the special case of the
model harmonic potential (\ref{eqharmonic}) by exploiting a scaling
property of the corresponding equation (\ref{eqlagrange}), see
Remark \ref{remignat} below, which is not available under our
general assumptions.  The above theorem has some other interesting
consequences, which will be explored in Sections \ref{secfurther},
\ref{secfepsilon}: We can prove an analogous monotonicity property
to (\ref{eqJ4}) for $\eta_\varepsilon$ near $\partial \mathcal{D}_0$
\emph{without} the simplifying assumption of radial symmetry  on the
potential $W$. We will see that estimates (\ref{eqJ2}) and
(\ref{eqJ3}) are actually far from optimal. In addition, restricting
ourselves to the radially symmetric case with $\mathcal{D}_0$ a
ball, we can improve and sharpen the new estimates of
\cite{aftalion-jerrard} for the important auxiliary function
$f_\varepsilon$, as described in (\ref{eqfepsilon}). In fact, we
believe that the estimates of Theorem \ref{thmmain} can be utilized
in estimating the function $f_\varepsilon$ in the non-radial case,
through equation (\ref{eqxidiv}) below, which may ultimately lead to
the resolution of the open problem raised in \cite{aftalion-jerrard}
(recall the discussion in Subsection \ref{secknown}).

To further illustrate the importance of our result, we emphasize
that its method of proof can be adapted to produce similar results
for semiclassical standing wave solutions of the defocusing NLS
$(\ref{eqNLS})_-$, with $N\geq 1$ and $q>2$, assuming that $W$ has
the features described Subsection \ref{secproblem} with $\lambda$ in
place of $\lambda_0$ and $\mathcal{D}_0$ a domain with
$(N-1)$--dimensional boundary (see also Remark \ref{remexponent}
below). (Recall that the minimizer $\eta_\varepsilon$ solves
(\ref{eqlagrange}), and the discussion in the second part of
Subsection \ref{secmotivation}). We also emphasize that, when
dealing directly with $(\ref{eqground})_-$, our approach does not
make any use of the techniques of Brezis and Oswald
\cite{brezis-oswald}, which were needed in previous approaches for
establishing (\ref{eqmontero}). In fact, our approach may produce
sign changing solutions of $(\ref{eqground})_-$, satisfying
estimates analogous to those of Theorem \ref{thmmain} (see Remark
\ref{remconnected} and Section \ref{secopen} below).

A rigorous connection between semiclassical ground states of the
defocusing nonlinear Schr\"{o}dinger equation $(\ref{eqNLS})_-$, in
one space dimension, and solutions of the Painlev\'{e}-II equation
(\ref{eqpanintro}) has been established recently in
\cite{schecter-sourdis,sourdis-fife}, for a related Hamiltonian
system. We refer to the physical works \cite{bandmalomed},
\cite{konotopkevrekidis}, \cite{lundh}, \cite{pethick} for formal
expansions in one space dimension or radially symmetric cases, and
to \cite{wu-marg} for higher dimensions (see also
\cite{margetis2012}).
 In the
 case of the model harmonic potential $W(\textbf{y})=|\textbf{y}|^2,\
\textbf{y}\in \mathbb{R}^N$, $N\geq 1$, $\lambda=1$, and $q=3$, the
problem of uniform asymptotic approximations of the ground state of
the defocusing NLS $(\ref{eqNLS})_-$, using the Hastings-McLeod
solution of the Painlev\'{e}-II equation, has been established on a
rigorous level very recently in \cite{pelinovsky}. However, the
approach of \cite{pelinovsky} relies crucially on the specific form
of the model harmonic potential, which allows for a suitable global
change of independent variables in equation $(\ref{eqground})_-$
(see Remark \ref{rempeli} below). The real delicacy of our result is
not that it successfully connects $\eta_\varepsilon$ with $V$, but
that we do so in a way that yields fine estimates, as can already be
seen from (\ref{eqenergysharp}), and the proof of
(\ref{eqlambdaepsilonnew}) (see also the detailed estimates of
\cite{pelinovsky} in the case of the model harmonic potential). The
optimality of estimates (\ref{eqestim1}), (\ref{eqestim2}) is also
suggested by Remark \ref{remyan} below.

The analog of relation (\ref{equappotentialthmmain}) may be used to
study the spectrum of the linearization of the Gross--Pitaevskii
equation $(\ref{eqNLS})_-$ at the corresponding standing wave
solution (\ref{eqphase}), which is defined by the eigenvalue
problem, in $L^2(\mathbb{R}^N)\times L^2(\mathbb{R}^N)$,
\begin{equation}\label{eqgrillakis}
\left\{\begin{array}{lll}
                                     -\varepsilon^2 \Delta \phi+(W-\lambda
+q|\eta_\varepsilon|^{q-1})\phi & = & -\mu \psi \\
                                      &   &   \\
                                     -\varepsilon^2\Delta
\psi+(W-\lambda+|\eta_\varepsilon|^{q-1})\psi & = & \mu \phi.
                                   \end{array}
\right.
\end{equation}
The above eigenvalue problem determines the spectral stability of
the standing wave, with respect to the time evolution of the GP
equation,  and gives preliminary information for nonlinear analysis
of orbital stability or more generally about the flow of
$(\ref{eqNLS})_-$ in a neighborhood of the  standing wave (see
\cite{batesjones}, \cite{jonesindiana}, \cite{grillakis},
\cite{schlagBook}). More complex phenomena, such as those of pinned
vortices (dark solitons) on top of the ground state
\cite{kevrekidisREVIEW}, \cite{pelikevrekidisGryllakis},
 and the construction of traveling wave
solutions with a stationary or traveling vortex ring to the GP
equation \cite{weipitaevskii}, \cite{weivortex-2}, can also be
understood from the analysis of (\ref{eqgrillakis}). In particular,
relation (\ref{equappotentialthmmain}) plays an important role in
the construction of excited states for $(\ref{eqground})_-$ (see
\cite{pelinovskynonlinear} and Section \ref{secopen} herein). It
seems that it was previously known \emph{only} in the special case
of the model harmonic potential (see \cite{pelinovsky}). In fact,
even under the assumption of general radial symmetry, it does not
follow from the recent estimates of \cite{aftalion-jerrard}.

 Observe that the equation in $(\ref{eqground})_-$, with $q=3$, is
equivalent, for $\varepsilon>0$, to
\begin{equation}\label{eqallenCahninhomog}
\Delta u+\left(\lambda-W(\varepsilon y) \right)u-u^3,\ \ y\in
\mathbb{R}^N,
\end{equation}
which when setting $\varepsilon=0$, and re-scaling appropriately (we
assume  that $\lambda>W(0)$),  becomes the well known Allen-Cahn
equation
\begin{equation}\label{eqAllenCahn} \Delta
v-v(v^2-1)=0,\ \ y\in \mathbb{R}^N .\end{equation} The above problem
has received an enormous attention, see for instance
\cite{delpinoAnnals} and the references therein.
 It seems
plausible that our result can be combined with existing ones for the
Allen-Cahn equation, and produce new interesting solutions for the
Gross-Pitaevskii equation $(\ref{eqground})_-$. We will elaborate
more on this in Section \ref{secopen} below.


\begin{rem}\label{remconnected}
Analogous assertions to those of Theorem \ref{thmmain} hold true if
$\mathcal{D}_0$ is assumed to be the union of finitely many bounded
smooth domains. In the latter case, one may construct sign changing
solutions of $(\ref{eqground})_-$, whose absolute value converges
uniformly to $\sqrt{A^+}$, as $\varepsilon\to 0$. The assumption
that $\mathcal{D}_0$ is simply connected plays an important role
only in Section \ref{secfepsilon} below. The degenerate case where
some connected components of $\mathcal{D}_0$ ``touch'' is left as an
open problem is Section \ref{secopen}.
\end{rem}
 \begin{rem}\label{remainfinity}
With only minor modifications in the proof, all the assertions of
Theorem \ref{thmmain} remain  true if (\ref{eqainfinity}) is
replaced by $\frac{1}{C}\leq W(\textbf{y})-\lambda_0\leq
C|\textbf{y}|^p,\ |\textbf{y}|\geq C$, for some $p>0$ (the second
branch in (\ref{equappotentialthmmain}) would have to be replaced by
a constant). The fact that $p\geq 2$ in (\ref{eqainfinity}) plays an
important role only in Proposition \ref{profeps} below.
\end{rem}
\begin{rem}\label{remdist}
Let us keep in mind that $|t|=\textrm{dist}(\textbf{y},
\partial \mathcal{D}_0)+\mathcal{O}(|\ln
\varepsilon|^\frac{1}{2}\varepsilon)$, wherever defined, and
a-posteriori $|t|=\textrm{dist}(\textbf{y},\partial
\mathcal{D}_0)+\mathcal{O}(|\ln \varepsilon|\varepsilon^2)$, as
$\varepsilon\to 0$ (recall (\ref{eqlambdaepsilon}),
(\ref{eqlambdaepsilonnew})).
\end{rem}
\begin{rem}\label{remGeneral}
An analogous result continuous to hold for the singularly perturbed
elliptic equation
\begin{equation}\label{eqEllipticIntro}
\varepsilon^2\Delta u=F(u,\textbf{y}),
\end{equation}
considered in the entire space $\mathbb{R}^N$ or in a bounded
domain, where $F$ is such that the zero set of $F(u,\textbf{y})=0$
undergoes a supercritical pitchfork bifurcation as the variable
$\textbf{y}$ crosses some surface (such conditions on $F$ can be
found in \cite{schecter-sourdis}). In order to bring out clearly the
underline ideas we refrain from any such generalization. In
(\ref{eqestim1+}) the convergence is exponential because, in the
case at hand, zero is a solution of $F=0$. However, there are no
delay phenomena present, as $\textbf{y}$ crosses $\partial
\mathcal{D}_0$, in contrast to the first order equation $\varepsilon
\dot{u}=F(u,t)$, as $t$ crosses the point corresponding to
$\partial\mathcal{D}_0$ (see \cite{benoit}).
\end{rem}

\subsection{Outline of the proof and structure of the
paper}\label{secoutline} The proof of Theorem \ref{thmmain} consists
in showing that there exists a genuine solution of
(\ref{eqlagrange}) ``near'' a suitably constructed smooth
approximate solution, which in turn is ``near'' $\sqrt{A^+}$,
provided the parameter $\varepsilon$ is small enough. \emph{We emphasize
that we will consider the Lagrange multiplier
$\varepsilon^{-2}\lambda_\varepsilon$,
 corresponding to the minimization
of $G_\varepsilon$ in $\mathcal{H}$, as a known coefficient in
(\ref{eqlagrange}) which, as we have already remarked, is known to
satisfy (\ref{eqlambdaepsilon}).}
 Then, by uniqueness
(see also Remark \ref{remuniq} below), we will conclude that the
obtained solution is actually the minimizer $\eta_\varepsilon$.
Actually, we will prefer to work with the equivalent (for
$\varepsilon>0$) problem in stretched variables
$y=\varepsilon^{-\frac{2}{3}}\textbf{y}$, see (\ref{eqEqstretched})
below. The main steps of the proof are the following:
\begin{enumerate}
  \item [(i)]
Firstly, we construct a ``good'' smooth approximate solution for the
(stretched) problem which we call $u_{ap}$. The function $u_{ap}$ is
carefully built, along the lines set in
\cite{karalisourdisresonance},
throughout Subsections \ref{secnearcurve}-\ref{secmatching} in the
following steps: Starting from the Hastings-McLeod solution $V$,
described above, we  construct an inner approximation $u_{in}$ that
is valid \emph{only} in
  a tubular neighborhood of the (stretched)
  curve $\varepsilon^{-\frac{2}{3}}\Gamma_\varepsilon$. Then, we
  obtain the desired (global) approximate solution $u_{ap}$ by
  patching $u_{in}$ with a subtle modification of the outer approximation
  $\sqrt{a_\varepsilon(\varepsilon^\frac{2}{3}y)^+}$.
  We emphasize that the use of this latter modification seems to be a key point in the whole construction.
In fact, the matching of the inner approximation with the outer one
is the major difficulty in the current problem,  mainly due to the
algebraic decay of $V$ to $\sqrt{-x}$ as $x\to -\infty$ (see
(\ref{eqvasympto}) below).  For more details on this point, see the
discussion in
  Subsection \ref{secaway} and Remark \ref{remmatching} below.
  This difficulty was not present in \cite{pelinovsky}, where the
  case of the model harmonic potential was considered, since in that case the
 special form of the equation allowed for the inner solution to be used globally, and thus no matching was necessary.
\item[(ii)] Next, in Subsection \ref{seclinear}, we study the linearized operator
\[
\mathcal{L}_\varepsilon\equiv\Delta
-\varepsilon^{-\frac{2}{3}}\left(3u_{ap}^2-a(\varepsilon^\frac{2}{3}y)\right),
\]
about the approximate solution $u_{ap}$, and invert it in carefully
chosen weighted spaces. We exploit a recent observation, due to
Gallo and Pelinovsky \cite[Lemma 2.2]{pelinovsky}, which says that
the potential $3V^2+x$ of the linearization of (\ref{eqpanintro})
about the Hastings-McLeod solution is bounded below by some positive
constant. In turn, we will show that this latter property implies
that the same assertion holds true for the potential of
$-\mathcal{L}_\varepsilon$, if $\varepsilon>0$ is small. In
particular, for small $\varepsilon$, the operator
$\mathcal{L}_\varepsilon$ is invertible. We point out that the
previously mentioned lemma of \cite{pelinovsky} is of technical
nature and can be bypassed at the expense of a more involved, but
rather standard, analysis (see Remark \ref{rempotential} and
Appendix \ref{ap2} below). Our choice of weighted spaces, a variant
of those considered in \cite{pacard-riviere}, is mainly motivated
from the error term (\ref{eqremuouttilda}) below.

\item[(iii)]
Finally, in Subsection \ref{secexiste}, we look for a genuine
solution of the stretched problem (\ref{eqEqstretched}) in the form
\[
u_\varepsilon = u_{ap} + \varphi,
\]
where $\varphi$ is a correction.
 At this stage, we show that we can rephrase the
problem as a fixed point problem for $\varphi$, which can easily be
solved, if $\varepsilon>0$ is sufficiently small, using the fixed
point theorem for contraction mappings. Then, in Subsection
\ref{secproofofmain}, taking advantage of the recent uniqueness
result of \cite{ignat} for problem (\ref{eqEqstretched})
(see also Remark \ref{remuniq} below), we infer that the unique real
valued minimizer of $G_\varepsilon$ in $\mathcal{H}$ satisfies
$\eta_\varepsilon(\textbf{y})=u_\varepsilon(\varepsilon^{-\frac{2}{3}}\textbf{y})$,
$\textbf{y}\in \mathbb{R}^2$. From this property, and the estimates
derived from the construction of $u_\varepsilon$, we can deduce the
validity of Theorem \ref{thmmain}. In particular, the proof of
estimate (\ref{eqlambdaepsilonnew}) builds on
(\ref{eqlambdaepsilon}), and uses the fact that
$\|\eta_\varepsilon\|_{L^2(\mathbb{R}^2)}=1$.
\end{enumerate}

The outline of the paper is the following: In Section
\ref{secnotation} we will introduce notation and standard concepts
that we will use throughout the paper. In Section \ref{secproofmain}
we will present the proof of our main result, as outlined above. In
Section \ref{secfurther}, as a byproduct of our construction, we
will establish an analogous monotonicity property to (\ref{eqJ4})
for the general case (with improvements), show that
$\eta_\varepsilon$ has the maximal H\"{o}lder regularity available,
improve bound (\ref{eqJ3}), and generalize and considerably improve
estimate (\ref{eqJ2}). In Section \ref{secfepsilon}, assuming that
the potential $W$ is radially symmetric with $\mathcal{D}_0$  a
ball, we will mainly rely on the results of Section \ref{secfurther}
to refine and improve the recent estimates of
\cite{aftalion-jerrard} for the important auxiliary function
$f_\varepsilon$ in (\ref{eqfepsilon}). In Section \ref{secopen}
  we will present some  interesting open   problems that are  related to the current study.
 We
will close the paper with two appendixes. In Appendix \ref{ap2} we
will reprove our main result concerning the linearized operator
$\mathcal{L}_\varepsilon$ based on the asymptotic stability of the
Hastings-McLeod solution (in the usual sense), rather than making
use of the recent lemma of \cite{pelinovsky} which may not hold in
other problems. Lastly, in Appendix \ref{appenpainleve} we will
provide a new more flexible and simple proof of the existence, and
related properties, of the Hastings-McLeod solution through the
study of problem (\ref{eqvdirichlet}) below. The results of the
latter study seem to be new and to have interesting applications
(see Remark \ref{remdirichlet} below).

\section{Notation}\label{secnotation} In the sequel, we will often suppress the obvious dependence on
$\varepsilon$ of various functions and quantities. Furthermore by
$c/C$ we will denote small/large generic constants, independent of
$\varepsilon$, whose value will change from line to line. The value
of $\varepsilon$ will constantly decrease so that all previous
relations hold. The Landau symbol $\mathcal{O}(1), \ \varepsilon\to
0$, will denote quantities that remain uniformly bounded  as
$\varepsilon\to 0$, whereas $o(1)$ will denote quantities that
approach zero as $\varepsilon\to 0$. By $C^{k,\alpha}$ or
equivalently $C^{k+\alpha}$,  $W^{k,p}$, $k,p \in \mathbb{Z}\cup
\{\infty\},\ \alpha\in [0,1)$, we will denote the usual H\"{o}lder,
and Sobolev spaces respectively (see for instance
\cite{Gilbarg-Trudinger}). Frequently, we will denote the minus case
of $(\ref{eqground})$ below plainly by $(\ref{eqground})_-$, e.t.c.

We will identify the tubular neighborhood $
B_\delta(\Gamma_\varepsilon)=\{\textbf{y}\in  \mathbb{R}^2\ :\
\textrm{dist}(\textbf{y},\Gamma_\varepsilon)<\delta\}, $ where
$\delta<\delta_0$ (recall (\ref{eqFermi})), with $
\Omega_\delta(\Gamma_\varepsilon)=\{|t|<\delta,\ \theta\in
[0,\ell_\varepsilon) \}$, denoted simply by $\{|t|<\delta\}$, and
$\textbf{y}$ in $B_\delta(\Gamma_\varepsilon)$ by the corresponding
pair $(t,\theta)\in \Omega_\delta(\Gamma_\varepsilon)$ as determined
via (\ref{eqFermi}).
\section{Proof of the main result}\label{secproofmain}

\subsection{Setup near the Curve}\label{secnearcurve} In this
subsection, suitably blowing up (\ref{eqlagrange}) around the curve
$\Gamma_\varepsilon$, we will construct an inner approximation which
is valid \emph{only} near the curve, for small $\varepsilon>0$,
whose profile normal to the curve will be that of a (scaled)
Hastings-McLeod solution (see also \cite{wu-marg}). To this end, we
will follow the general lines set in \cite{del pino cpam} which
dealt with the focusing case $(\ref{eqground})_+$.

Formally neglecting the term $\varepsilon^2\Delta u$ in
  (\ref{eqlagrange}), we get the outer approximation
  $\sqrt{a_\varepsilon(\textbf{y})^+}$. However, the Laplacian of the latter is not
  even bounded on the boundary of $\mathcal{D}_\varepsilon$. Hence, the outer approximation fails in the vicinity
  of $\partial \mathcal{D}_0$.
   An
  inner approximation is thus needed, playing the role of a
  ``bridge'' as $\textbf{y}$ crosses that boundary.

 In the coordinates $(t,\theta)$, near
$\Gamma_\varepsilon$, the metric can be parameterized as
\[
g_{t,\theta}=dt^2+(1+kt)^2d\theta^2,
\]
and the Laplacian operator becomes
\[
\Delta_{t,\theta}=\frac{\partial^2}{\partial
t^2}+\frac{1}{(1+kt)^2}\frac{\partial^2}{\partial \theta^2} +
\frac{k}{1+kt}\frac{\partial}{\partial
t}-\frac{k't}{(1+kt)^3}\frac{\partial}{\partial \theta},
\]
where $k_\varepsilon(\theta)$ is the curvature of
$\Gamma_\varepsilon$ (see for instance \cite{stefanopoulos},
\cite{fife-arma}). Note that $k_\varepsilon$ and its derivatives
depend smoothly on $\varepsilon\geq 0$.
\begin{rem}\label{remfermi}
$(t,\gamma_\varepsilon(\theta))$ are known in the literature as
Fermi coordinates, see for instance \cite{klidenberg}, and  are
frequently employed in the study of perturbation problems involving
solutions concentrating on manifolds (see for example \cite{del pino
cpam}). Interestingly enough, they owe their name to the physicist
E. Fermi in the title of the current paper!
\end{rem}

In stretched variables \[y=\varepsilon^{-\frac{2}{3}}\textbf{y},\]
problem (\ref{eqlagrange}) becomes
\begin{equation}\label{eqEqstretched}
\Delta
u-\varepsilon^{-\frac{2}{3}}u\left(u^2-a_\varepsilon(\varepsilon^\frac{2}{3}y)
\right)=0,\ \ u>0  \ \textrm{in}\
  \mathbb{R}^2,\ \ \lim_{|{y}|\to \infty}u({y})=0,
\end{equation}
where $a_\varepsilon$ was defined in (\ref{eqaepsilon}).

In the sequel we will denote
\begin{equation}\label{eqDepsilon}\tilde{\mathcal{D}}_\varepsilon=\varepsilon^{-\frac{2}{3}}\mathcal{D}_\varepsilon\
\ \textrm{and}\ \
\tilde{\Gamma}_\varepsilon=\varepsilon^{-\frac{2}{3}}\Gamma_\varepsilon.\end{equation}
Let
\begin{equation}\label{eqszfinal}(s, z) =
\varepsilon^{-\frac{2}{3}}(t, \theta)
\end{equation}
 be natural stretched
coordinates associated to the curve $\tilde{\Gamma}_\varepsilon$,
now defined for
\begin{equation}\label{eqFermi-stretched}
s\in \left(-\delta_0\varepsilon^{-\frac{2}{3}},\
\delta_0\varepsilon^{-\frac{2}{3}} \right),\ \ \ z\in
\left[0,\varepsilon^{-\frac{2}{3}}\ell_\varepsilon\right).
\end{equation}
In the coordinates $(s,z)$, near $\tilde{\Gamma}_\varepsilon$, the
metric can be written as
\begin{equation}\label{eqmetric}
g_{s,z}=ds^2+(1+\varepsilon^\frac{2}{3}ks)^2dz^2,
\end{equation}
and the Laplacian for $u$ expressed in these coordinates becomes
\begin{equation}\label{eqLaplace}
\Delta u=u_{zz}+u_{ss}+B_1(u),
\end{equation}
where
\begin{equation}\label{eqB1}
B_1(u)=-\left[1-\frac{1}{\left(1+\varepsilon^\frac{2}{3}k(\varepsilon^\frac{2}{3}z)s\right)^2}
\right]u_{zz}+\frac{\varepsilon^\frac{2}{3}k(\varepsilon^\frac{2}{3}z)}{1+\varepsilon^\frac{2}{3}k(\varepsilon^\frac{2}{3}z)s}u_s-
\frac{\varepsilon^\frac{4}{3}sk'(\varepsilon^\frac{2}{3}z)}{\left(1+\varepsilon^\frac{2}{3}k(\varepsilon^\frac{2}{3}z)s\right)^3}u_z.
\end{equation}
Hence, equation (\ref{eqEqstretched}) takes the form
\begin{equation}\label{eqEqfermi}
R(u)\equiv
u_{zz}+u_{ss}+B_1(u)-\varepsilon^{-\frac{2}{3}}u\left(u^2-a(\varepsilon^\frac{2}{3}s,\varepsilon^\frac{2}{3}z)
\right)=0,
\end{equation}
in the region (\ref{eqFermi-stretched}). Observe that all terms in
the operator $B_1$ have $\varepsilon^\frac{2}{3}$ as a common
factor, more precisely we can write
\begin{equation}\label{eqB0}
B_1(u)=
\varepsilon^\frac{2}{3}a_0(\varepsilon^\frac{2}{3}s,\varepsilon^\frac{2}{3}z)u_s+
\varepsilon^\frac{4}{3}sa_1(\varepsilon^\frac{2}{3}s,\varepsilon^\frac{2}{3}z)u_z+
\varepsilon^\frac{2}{3}sa_2(\varepsilon^\frac{2}{3}s,\varepsilon^\frac{2}{3}z)u_{zz},
\end{equation}
for certain smooth functions $a_j (t, \theta),\ j = 0, 1, 2$.

We now consider a further change of variables in equation
(\ref{eqEqfermi}) with the property that it (roughly) replaces at
main order the function $a$ by one that has constant gradient along
$\tilde{\Gamma}_\varepsilon$.
Let $\beta$ be as in (\ref{eqbita}), then
we define $v(x, z)$ by the relation
\begin{equation}\label{eqvsz}\left\{\begin{array}{l}
                             u(s,z)=\varepsilon^\frac{1}{3}\beta(\varepsilon^\frac{2}{3}z)v(x,z),
 \\
                              \\
                             x=\beta(\varepsilon^\frac{2}{3}z)s.
                           \end{array}\right.
\end{equation}
Choosing a smaller $\delta_0$, if necessary, we may assume that the
coordinates $(x,z)$ are also defined for $|x|\leq
\delta_0\varepsilon^{-\frac{2}{3}}$, $z\in
[0,\varepsilon^{-\frac{2}{3}}\ell_\varepsilon)$. We want to express
equation (\ref{eqEqfermi}) in terms of these new coordinates. We
compute:
\begin{equation}\label{equs-uss-uz-uzz}
\left\{
\begin{array}{lll}
  u_s &=& \varepsilon^\frac{1}{3}\beta^2v_x, \\
  && \\
  u_{ss}&=&\varepsilon^\frac{1}{3}\beta^3v_{xx}, \\
   &&\\
  u_z &=& \varepsilon\beta'v+\varepsilon^\frac{1}{3}\beta v_z+\varepsilon \beta'x v_x, \\
  && \\
  u_{zz}&=&
\varepsilon^\frac{5}{3}\beta''v+2\varepsilon^\frac{5}{3}(\beta')^2\beta^{-1}xv_x+2\varepsilon\beta'v_z+2\varepsilon\beta'xv_{xz}+\varepsilon^\frac{1}{3}\beta v_{zz}\\
&&\\
&&
+\varepsilon^\frac{5}{3}\beta''xv_x+\varepsilon^\frac{5}{3}(\beta')^2\beta^{-1}x^2v_{xx}.
\end{array}\right.
\end{equation}
In order to write down the equation, it is also convenient to expand
\begin{equation}\label{eqa-taylor}
a(\varepsilon^\frac{2}{3}s,\varepsilon^\frac{2}{3}z)=a_t(0,\varepsilon^\frac{2}{3}z)\varepsilon^\frac{2}{3}s
+a_3(\varepsilon^\frac{2}{3}s,\varepsilon^\frac{2}{3}z)\varepsilon^\frac{4}{3}s^2\stackrel{(\ref{eqbita})}{=}
-\beta^2\varepsilon^\frac{2}{3}x+a_3(\varepsilon^\frac{2}{3}s,\varepsilon^\frac{2}{3}z)\varepsilon^\frac{4}{3}s^2,
\end{equation}
for some bounded function $a_3(t,\theta)$.
 It turns out that
$u$ solves (\ref{eqEqfermi}) if and only if $v$, defined by
(\ref{eqvsz}), solves
\begin{equation}\label{eqS(v)=0}
S(v)\equiv \varepsilon^{-\frac{1}{3}}\beta^{-3}R(u)=
v_{xx}-v(v^2+x)+B_2(v)=0,
\end{equation}
where $B_2(v)$ is a differential operator defined by
\begin{equation}\label{eqB2}
  B_2(v)=
\varepsilon
  ^\frac{2}{3}\tilde{a}_3x^2v+\varepsilon^{-\frac{1}{3}}\beta^{-3}u_{zz}+\varepsilon^{-\frac{1}{3}}\beta^{-3}B_1(u),
\end{equation}
where the bounded function  $\tilde{a}_3 (t, \theta)$ is evaluated
at $(\varepsilon^\frac{2}{3}\beta^{-1}x,\varepsilon^\frac{2}{3}z)$,
and $B_1$ is the differential operator in (\ref{eqB1}) where
derivatives are expressed in terms of formulas
(\ref{equs-uss-uz-uzz}) and $s$ replaced by $\beta^{-1}x$. (Note
that $ B_2(v)=\beta^{-2}v_{zz}+\tilde{B}_2(v), $ where all the terms
in the operator $\tilde{B}_2$ carry $\varepsilon^\frac{2}{3}$ in
front of them).

Our first criterion for choosing an approximate solution $v$ of
(\ref{eqS(v)=0}) is that $S(v)$ is small. In view of
(\ref{eqS(v)=0}), (\ref{eqB2}), it is natural to choose a $v$ that
depends only on $x$, and solves the second--order non--autonomous
ordinary differential equation:
\begin{equation}\label{eqpainleve}
v_{xx}-v(v^2+x)=0,\ \ \ x\in \mathbb{R},
\end{equation}
which is known as the Painlev\'{e}-II equation, a particular case of
the second Painlev\'{e} transcendent \cite{fokas}. Then, keeping in
mind that the inner profile $\varepsilon^\frac{1}{3}\beta v(\beta
s)$ of (\ref{eqEqstretched}) should match \cite[Chpt. 5]{vandyke}
with the outer profile $\sqrt{a(\varepsilon^\frac{2}{3}y)^+}$, as
$s\to \pm \infty$, it is easy to see that the appropriate asymptotic
behavior of
 $v$ should be
\begin{equation}\label{eqvboundary}
v(x)-\sqrt{-x}\to 0,\ x\to -\infty;\ \ v(x)\to 0,\ x\to \infty.
\end{equation}
(Observe that the $s$ variable is defined in
$(-\delta_0\varepsilon^{-\frac{2}{3}},
\delta_0\varepsilon^{-\frac{2}{3}})$, but the natural domain for
this variable is the infinite line). More precisely, recalling
(\ref{eqa-taylor}), we get that
\[
\sqrt{a_\varepsilon(\varepsilon^\frac{2}{3}y)^+}=\varepsilon^\frac{1}{3}\beta
\sqrt{-\beta s}+\mathcal{O}(\varepsilon |s|^\frac{3}{2})
\]
if $-\delta_0\varepsilon^{-\frac{2}{3}}\leq s \leq 0$, as
$\varepsilon\to 0$, while on the other side it holds that
$a_\varepsilon(\varepsilon^\frac{2}{3}y)=0$ if $0\leq s\leq
\delta_0\varepsilon^{-\frac{2}{3}}$.

The following proposition holds:
\begin{pro}\label{prohastings}
Problem (\ref{eqpainleve})--(\ref{eqvboundary}) has a unique
solution $V$, called \emph{Hastings-McLeod solution}. Furthermore,
we have
\begin{equation}\label{eqVx<0}
V_x<0\ \ \textrm{in}\ \ \mathbb{R}.
\end{equation}

The following estimates hold:
\begin{equation}\label{eqvasympto}
-C|x|^{-\frac{5}{2}}<V(x)-\sqrt{-x}<0,\  x\to -\infty;\ \
0<V(x)<Cx^{-\frac{1}{4}}e^{-\frac{2}{3}x^\frac{3}{2}},\ x\geq 1,
\end{equation}
for some constant $C>0$, and
\begin{equation}\label{eqvestims}
\left\{\begin{array}{lll}
  V_x=-\frac{1}{2}(-x)^{-\frac{1}{2}}+\mathcal{O}(|x|^{-\frac{7}{2}}), & V_{xx}=-\frac{1}{4}(-x)^{-\frac{3}{2}}+\mathcal{O}(|x|^{-\frac{9}{2}}), & x\to -\infty, \\
    &   &   \\
  V_xV=-\frac{1}{2}+\mathcal{O}(|x|^{-3}), &  V^2+x=\mathcal{O}(|x|^{-2}), & x\to -\infty, \\
    &   &   \\
|V_x|+|V_{xx}|\leq C e^{-cx^\frac{3}{2}},    &   &  x>0.
\end{array}
\right.
\end{equation}

 The potential of the associated
linearized operator
\begin{equation}\label{eqMcal}
\mathcal{M}(\varphi)\equiv\varphi_{xx}-\left(3V^2(x)+x\right)\varphi
\end{equation}
satisfies
\begin{equation}\label{eqvpeli}
3V^2(x)+x\geq c>0,\ \ x\in \mathbb{R}.
\end{equation}
\end{pro}
\begin{proof}
We will provide a sketch of proof, underlining the main ideas, and
refer the interested reader to the original works.

Existence and uniqueness of a solution $V$ of
(\ref{eqpainleve})--(\ref{eqvboundary}) have been shown by Hastings
and McLeod \cite{hastings}, \cite{hastingsbook} using a ``shooting''
technique (they also mention an unpublished proof of Conley).
Recently, a new proof of existence has been given in \cite{alikakos}
by the method of upper and lower solutions. Motivated from some
problems that we mentioned in Section \ref{secmotivation} (see also
Remark \ref{remdirichlet} below), we will provide in Appendix
\ref{appenpainleve} a new different proof of existence and
uniqueness  which seems to be simpler and more flexible. Another
different  and rather short proof has  been given very recently in
\cite[Lemma 2.4]{weipitaevskii}. Relation (\ref{eqVx<0}) has also
been shown in the aforementioned references, and in Appendix
\ref{appenpainleve} herein. We note that, by looking at the equation
for $Z\equiv(-x)^{-\frac{1}{2}}V$, one can show that $V$ has a
unique inflection point (see \cite{hastings}); notice also that
$V-\sqrt{-x}$ is convex whenever nonnegative.

The asymptotic behavior of $V$, as $x\to \pm \infty$, is described
in great detail in Theorem 11.7 of \cite{fokas} (see also
\cite{itsLNM}). In particular, it follows that relations
(\ref{eqvasympto})--(\ref{eqvestims}) hold true. Let us provide some
intuition behind these, at first sight, complex formulae. The first
relation of (\ref{eqvasympto}) can be formally derived as follows:
Assume that $V-\sqrt{-x}\sim \alpha(-x)^{-\beta}$ as $x\to -\infty$,
for some $\alpha\in \mathbb{R},\ \beta>0$; plugging this ansatz in
(\ref{eqpainleve}), keeping in mind that $V+\sqrt{-x}\sim 2
\sqrt{-x}$ and (the expectation) that $V_{xx}\sim
-\frac{1}{4}(-x)^{-\frac{3}{2}}$ as $x\to -\infty$, we readily find
that $\alpha=-\frac{1}{8},\ \beta=\frac{5}{2}$ (see also \cite[pg.
160]{pethick}). A rigorous derivation can be given by simply writing
down  the equation for the function $V-\sqrt{-x}$ (or better yet for
$V-\sqrt{-x}+\frac{1}{8}(-x)^{-\frac{5}{2}}$), and then applying in
the resulting identity the following simple lemma, which can be
proven by a standard barrier argument:

\textbf{Lemma} (cf. \cite[Lemma 2.1]{pelinovsky} or \cite[Lemma
3.10]{karalisourdisradial}) \emph{Suppose that $\phi\in C^2,\ q\in C$
satisfy
\[
-\phi''+q(x)\phi=\mathcal{O}(|x|^{-\alpha}),
\]
$\phi\to 0$, and $q(x)\geq c|x|$, as $x \to -\infty$, for some
constants $\alpha,\ c>0$. Then we have
\[
\phi=\mathcal{O}(|x|^{-\alpha-1})\ \ \textrm{as}\ x\to -\infty.
\]}
In passing, we note that the above lemma can also be proven by
rewriting the equation in terms of the new independent variable
$\xi=-(-x)^{\frac{3}{2}}$ and then applying a lemma of H\'{e}rve and
H\'{e}rve \cite[pg. 435]{herve}, see also \cite[Prop. 3.1]{helfer}.
 On the other side, because of the second condition in
(\ref{eqvboundary}), the nonlinear term in (\ref{eqpainleve}) can be
(formally) neglected for $x\to \infty$, yielding Airy's equation
\cite[pg. 100]{BenderO}, namely
\begin{equation}\label{eqairy}
y''=xy,
\end{equation}
predicting that
\begin{equation}\label{eqVsimAiry}V(x)\sim \gamma\textrm{Ai}(x)\ \  \textrm{as}\  x\to \infty,
\end{equation}
 for some
constant $\gamma>0$, where $\textrm{Ai}(\cdot)$ denotes the standard
Airy's function (see also (\ref{eqairyasymptotic}) below). For
future reference, we recall from \cite{BenderO} that the two
independent solutions of Airy's equation can be taken to have the
asymptotic behavior
\begin{equation}\label{eqairyasymptotic}
\textrm{Ai}(x)\sim
\frac{1}{2}\pi^{-\frac{1}{2}}x^{-1/4}e^{-2x^{3/2}/3}\ \
\textrm{and}\ \ \textrm{Bi}(x)\sim
\pi^{\frac{1}{2}}x^{-1/4}e^{2x^{3/2}/3},\ \ x\to \infty,
\end{equation}
see also the discussion leading to formulae (\ref{eqBOa}) below.
 Relation (\ref{eqVsimAiry}), which clearly
implies the validity of the second estimate in (\ref{eqvasympto}),
can be established rigorously directly from the variation of
constants formula
\begin{equation}\label{eqvariationofcons}
V(x)=\gamma\textrm{Ai}(x)+2\int_{x}^{\infty}\left\{\textrm{Ai}(x)\textrm{Bi}(t)-\textrm{Bi}(x)\textrm{Ai}(t)
\right\}V^3(t)dt,
\end{equation}
see also \cite[Lemma B.1]{alfimov}.
 It is worthwhile to note that the exact value of $\gamma$ was
determined to be
\begin{equation}\label{eqgamma1}
\gamma=1
\end{equation}
 in \cite{hastings}, using an integral equation
derived by inverse scattering techniques of Ablowitz and Segur. An
implication of this beautiful formula, for the physics of
Bose--Einstein condensation, has been pointed out in
\cite{margetispainleve}. The estimates in (\ref{eqvestims}) can be
established in a similar manner. For example, the asymptotic
behavior of $V_x$ as $x\to -\infty$ follows by writing down the
equation for the function $V_x+\frac{1}{2}(-x)^{-\frac{1}{2}}$ and
then applying the above lemma.

 Finally, the lower bound (\ref{eqvpeli}) has been proven
recently in \cite[Lemma 2.2]{pelinovsky}; despite of its simple
appearance, its proof takes almost three pages! Actually, we were
surprised to find that such an apparently simple relation proves to
be so recalcitrant. We emphasize that, for the purposes of the
current paper, it can be bypassed (something that we did not do) at
the expense of some extra work (see Remark \ref{rempotential}
below). On the other side, it is quite painless to show that the
operator $\mathcal{M}$ still satisfies the maximum principle
\emph{without} knowledge of the latter lemma. In other words:
\begin{equation}\label{eqwWords}\textrm{``whenever}\  w\in
C^2(\mathbb{R}),\ \mathcal{M}(w)\leq 0,\ \textrm{and}\
\liminf_{|x|\to \infty}w(x)\geq 0,\ \textrm{then\ necessarily}\
w\geq0\textrm{''}.\end{equation} For this, recall that the typical
way towards establishing the maximum principle for an elliptic
operator $L$ in an open set $\Omega$
 is to first show that
\begin{equation}\label{eqCabrewords}\textrm{``there\ exists\ a\ positive\ upper--solution}\ \phi\ \textrm{of}\ L(\phi) =
0\ \textrm{in}\ \Omega\textrm{''},\end{equation} see
\cite{saddlecabre3solo} for more details.
 If this holds, then
adding one of various additional assumptions on $\phi$ (the simplest
one being $\phi\geq c$ with $c$ a positive constant), it does
guarantee the maximum principle to hold, see \cite{berestycki}.
Indeed, in bounded domains, (\ref{eqCabrewords}) is a necessary and
almost sufficient condition for the maximum principle to hold, see
Corollary 2.1 of \cite{berestycki}. However, in unbounded domains
one has to be more careful to deal with infinity. Having this in
mind, firstly note that
\begin{equation}\label{eqberestycki}
\mathcal{M}(-V_x)=-V<0;\ \ \ \mathcal{M}(V)=-2V^3<0,
\end{equation}
and recall that $-V_x>0,\ V>0$ in $\mathbb{R}$. So, condition
(\ref{eqCabrewords}), with $L=-\mathcal{M}$, is satisfied by either
one of $-V_x$ or $V$. Nevertheless, observe that one faces a
difficulty when proceeding as in \cite{berestycki}, namely applying
the standard maximum principle in the equation satisfied by
$\frac{w}{-V_x}$ or $\frac{w}{V}$ (where $w$ is as in
(\ref{eqwWords})). This is because $-V_x$ and $V$ vanish as $x\to
\pm \infty$ and $x\to \infty$ respectively, making the signs of
$\liminf_{x\to \pm \infty}\frac{w}{-V_x}$ and $\liminf_{x\to \infty}
\frac{w}{V}$ unclear (in the case where one of $\liminf_{x\to\pm
\infty}w$ equals zero). (This subtle point seems to have been
overlooked in \cite{weipitaevskii}). Instead, we consider the
function
\[
\varphi_\epsilon=\frac{w-\epsilon V_x}{V},\ \ \epsilon>0,
\]
noting that, thanks to (\ref{eqairy})--(\ref{eqairyasymptotic}) (see
also (\ref{eqairy'}) below), we have $\lim_{x\to
-\infty}\varphi_\epsilon\geq 0$ and $\lim_{x\to
\infty}\varphi_\epsilon= \infty$. Using (\ref{eqwWords}),
(\ref{eqberestycki}), and applying the standard maximum principle in
the equation satisfied by $\varphi_\epsilon$, we obtain that
$\varphi_\epsilon>0$ in $\mathbb{R}$, for every $\epsilon>0$. Then,
letting $\epsilon\to 0$ we infer that $w\geq 0$ in $\mathbb{R}$
which is the desired assertion in (\ref{eqwWords}).

The sketch of the proof of the proposition is complete.
\end{proof}

We can now define the inner solution of problem
(\ref{eqEqstretched}), in the neighborhood of
$\tilde{\Gamma}_\varepsilon$ described by $\{|x|\leq \delta_0
\varepsilon^{-\frac{2}{3}},\ \ z\in
[0,\varepsilon^{-\frac{2}{3}}\ell_\varepsilon)\}$, as
\begin{equation}\label{equin}
u_{in}(y)=\varepsilon^\frac{1}{3}\beta(\varepsilon^\frac{2}{3}z)V(x),
\end{equation}
(recall (\ref{eqvsz})).

Let $L>0$ be a large constant to be determined independently of
small $\varepsilon>0$. We consider the neighborhood of
$\tilde{\Gamma}_\varepsilon$ described by
\begin{equation}\label{eqStripepsilon}
U_\varepsilon=\{-2L \leq x\leq \delta_0 \varepsilon^{-\frac{2}{3}},\
\ z\in [0,\varepsilon^{-\frac{2}{3}}\ell_\varepsilon)\}.
\end{equation}

The following proposition contains the main estimate regarding
$u_{in}$.
\begin{pro}\label{prouin}
If $\varepsilon$ is sufficiently small, the inner approximation
$u_{in}$, defined in (\ref{equin}), satisfies
\begin{equation}\label{eqremuin}
\|\Delta
u_{in}-\varepsilon^{-\frac{2}{3}}u_{in}\left(u_{in}^2-a(\varepsilon^\frac{2}{3}y)
\right)\|_{L^\infty(U_\varepsilon)}\leq C\varepsilon.
\end{equation}
\end{pro}
\begin{proof}
From the calculation leading to (\ref{eqS(v)=0}), and (\ref{eqB2}),
we find that
\begin{equation}\label{equinremexplicit}
\begin{array}{lll}
\left|  \Delta
u_{in}-\varepsilon^{-\frac{2}{3}}u_{in}\left(u_{in}^2-a(\varepsilon^\frac{2}{3}y)
\right)\right| & = &|R(u_{in})|= \varepsilon^\frac{1}{3}\beta^3|S(V)| \\
   &  &  \\
   & \leq &C \varepsilon\left[(x^2+\varepsilon^\frac{2}{3})V+|V_x|+|x||V_{xx}|\right]
\end{array}
\end{equation}
pointwise in $U_\varepsilon$. The desired estimate (\ref{eqremuin})
now follows at once from the above relation, via (\ref{eqvasympto}),
(\ref{eqvestims}) and (\ref{eqStripepsilon}).

The proof of the proposition is complete.
\end{proof}

\begin{rem}\label{remgeometric}
From the geometric singular perturbation viewpoint
\cite{schecter-sourdis}, the Hastings-McLeod solution  corresponds
to a special trajectory of the (de-singularized) blown-up system
that connects two equilibria on a sphere. The fact that this
connection is transverse, which allows for a perturbation argument,
follows from the non-degeneracy of the Hastings-McLeod solution (as
defined in Remark \ref{rempotential} below). In this regard, see
also Remark \ref{remcornerlemma} below.
\end{rem}
\begin{rem}\label{rempainleveApplications}
In the context of singular perturbation problems, the
Hastings-McLeod solution first served as a basis for an inner
solution in plasma physics (see \cite{deboer}). Since then, it has
 been (formally) used to describe layered structures in problems involving crystalline interphase boundaries
\cite{alikakos,schecter-sourdis,sourdis-fife},  patterns of
convection in rectangular platform containers \cite{painleveFluid},
self-similar parabolic optical solitary waves \cite{btnn}, and the
Navier-Stokes and continuity equations for axisymmetric flow
\cite{painlevegotler}.
\end{rem}
\begin{rem}\label{remfisher}
In well known singularly perturbed elliptic problems, such as the
focusing NLS (recall the discussion  in Subsection
\ref{secscrodmotiv}) or the spatially inhomogeneous Allen-Cahn
equation (see for instance \cite{guilayer},
\cite{nakashimaNonradial}, \cite{weicluster}), the corresponding
inner profile is determined by special solutions of
\emph{autonomous} second order elliptic equations (posed in less or
equal dimensions). In the former case the corresponding equation is
\[
\Delta u-u+|u|^{q-1}u=0,
\]
while in the latter it is equation (\ref{eqAllenCahn}). In contrast,
in the problem at hand the corresponding equation (\ref{eqpainleve})
is \emph{non-autonomous}, as was the case in \cite{byeon},
\cite{dancer-lazer}, \cite{danceryanCrtitic},
\cite{karalisourdisradial}, and \cite{karalisourdisresonance}. The
interested reader can verify that  similar situations also occur in
the singularly perturbed Fisher's equation \cite[Chpt. 10,
Exc.3]{henry} (see also \cite{nakashimaFisher}), and in the
spatially inhomogeneous Allen-Cahn equation that we mentioned
previously, treated in \cite{guilayer}, \cite{nakashimaNonradial},
if the spatial inhomogeneity is not strictly positive but vanishes
at certain points (or submanifolds) of the domain (have in mind the
first part of Subsection \ref{secandre}).
\end{rem}
\subsection{Set up away from the curve}\label{secaway} In this
subsection,  adapting an idea of \cite{karalisourdisradial} (see
also \cite{karalisourdisresonance}), we suitably perturb, in
$\tilde{\mathcal{D}}_\varepsilon$,
the outer approximation
\begin{equation}\label{equout}u_{out}\equiv\sqrt{a(\varepsilon^\frac{2}{3}y)^+}
\end{equation}
to an improved outer approximation $\tilde{u}_{out}$, which is
closer to the inner approximation $u_{in}$ near the curve
$\tilde{\Gamma}_\varepsilon$, for small $\varepsilon>0$. We
emphasize that this is a key step in our construction of an
approximate solution for (\ref{eqEqstretched}). Our underlying plan,
carried out in the following subsection, is to smoothly interpolate
between $\tilde{u}_{out}$ and $u_{in}$,
and $u_{in}$ and zero (in $\mathbb{R}^2\backslash
\tilde{\mathcal{D}}_\varepsilon$), near the curve
$\tilde{\Gamma}_\varepsilon$, in order to obtain an approximate
solution $u_{ap}$ of (\ref{eqEqstretched}) that is valid in all of
$\mathbb{R}^2$. Interpolating directly between $u_{out}$, as defined
in (\ref{equout}), and $u_{in}$ in $\tilde{\mathcal{D}}_\varepsilon$
is not standard due to the following obstructions: The inner
approximation $u_{in}$ leaves a remainder in (\ref{eqEqstretched})
that grows with respect to the (negative) distance from
$\tilde{\Gamma}_\varepsilon$, as can be seen from
(\ref{eqvasympto}), (\ref{equinremexplicit}); $V$ converges
algebraically slowly to $\sqrt{-x}$ as $x\to -\infty$, recall
(\ref{eqvasympto}).
(For more details on this subtle point we refer to Remark
\ref{remmatching} below). In the one-dimensional case, an elegant
solution to this can be given by geometric singular perturbation
theory, see \cite{schecter-sourdis} and Remark \ref{remcornerlemma}
below. In contrast, thanks to the (super--) exponentially fast
convergence of $V$ to zero, as $x\to \infty$, interpolating directly
between $u_{in}$ and zero in
$\mathbb{R}^2\backslash\tilde{\mathcal{D}}_\varepsilon$ is rather
standard. (This last situation occurs in the construction of
spike-layered solutions of the focusing $(\ref{eqground})_+$, see
\cite{malchiodibook}, \cite{del pino cpam}, \cite{floer},  and
transition-layered solutions of  Allen-Cahn type equations, see
\cite{delpinoarma}, \cite{fifegreenlee}).

Let $\delta < \frac{\delta_0}{100(1+\max \beta)}$ be a fixed number.
We consider a smooth cutoff function
\begin{equation}\label{eqndelta-cuttoff}
n_\delta(t)=\left\{\begin{array}{ll}
                     1, & |t|\leq \delta, \\
                       &   \\
                     0, & |t|\geq 2\delta.
                   \end{array}
 \right.
\end{equation}
Denote as well
\begin{equation}\label{eqcutoffchi}
\chi_\delta(x)=n_\delta(\varepsilon^\frac{2}{3}x),
\end{equation}
where $x$ is the coordinate in (\ref{eqvsz}).

We define our outer approximation in the domain
$\tilde{\mathcal{D}}_\varepsilon \backslash\{-L<\beta s<0,\ z\in
[0,\varepsilon^{-\frac{2}{3}}\ell_\varepsilon)\}$ to be
\begin{equation}\label{equouttilda}
\tilde{u}_{out}(y)\equiv\left\{a(\varepsilon^\frac{2}{3}y)+\varepsilon^\frac{2}{3}\beta^2
\chi_\delta(\beta s)\left[\beta s+V^2(\beta
s)\right]\right\}^\frac{1}{2}.
\end{equation}
 Note that, thanks to (\ref{eqa>b}), (\ref{eqvestims}), if $L$ is
sufficiently large then $\tilde{u}_{out}$ is well defined for small
$\varepsilon$.

The following proposition makes delicate use of estimates
(\ref{eqvasympto}), (\ref{eqvestims}), and contains the main
properties of $\tilde{u}_{out}$. In some sense, it is the ``heart''
of the present paper.
\begin{pro}\label{prououttilda}
We have
\begin{equation}\label{eqremuouttilda}
\Delta
\tilde{u}_{out}-\varepsilon^{-\frac{2}{3}}\tilde{u}_{out}\left(\tilde{u}_{out}^2-a(\varepsilon^\frac{2}{3}y)\right)=\mathcal{O}(\varepsilon
|s|^{-\frac{1}{2}})
\end{equation}
uniformly in $\{-\delta\varepsilon^{-\frac{2}{3}} \leq \beta s\leq
-L, \ z\in [0,\varepsilon^{-\frac{2}{3}}\ell_\varepsilon)\}$, and
\begin{equation}\label{eqremuouttilda2}
\Delta
\tilde{u}_{out}-\varepsilon^{-\frac{2}{3}}\tilde{u}_{out}\left(\tilde{u}_{out}^2-a(\varepsilon^\frac{2}{3}y)\right)=\mathcal{O}(\varepsilon^{\frac{4}{3}})
\end{equation}
uniformly in $\tilde{\mathcal{D}}_\varepsilon\backslash
\{-\delta\varepsilon^{-\frac{2}{3}} \leq \beta s<0, \ z\in
[0,\varepsilon^{-\frac{2}{3}}\ell_\varepsilon)\}$, as
$\varepsilon\to 0$.
 Moreover, if $\varepsilon>0$ is sufficiently small, we have
\begin{equation}\label{equout-uingradients}
|\tilde{u}_{out}-u_{in}|\leq C\varepsilon |s|^\frac{3}{2},\ \
\left|\nabla(\tilde{u}_{out}-u_{in})\right|\leq C \varepsilon
|s|^\frac{1}{2},\ \ \left|\Delta(\tilde{u}_{out}-u_{in})\right|\leq
C\varepsilon |s|^{-\frac{1}{2}}
\end{equation}
in $\{-\delta\varepsilon^{-\frac{2}{3}}\leq \beta s\leq -L, \ z\in
[0,\varepsilon^{-\frac{2}{3}}\ell_\varepsilon)\}$, and
\begin{equation}\label{equouttilda-a}
\left|\tilde{u}_{out}-\sqrt{a(\varepsilon^\frac{2}{3}y)}\right|\leq
C\varepsilon^\frac{1}{3}|s|^{-\frac{5}{2}}
\end{equation}
in $\{-2\delta\varepsilon^{-\frac{2}{3}}\leq \beta s\leq -L, \ z\in
[0,\varepsilon^{-\frac{2}{3}}\ell_\varepsilon)\}$.
\end{pro}
\begin{proof}
In $\{-\delta\varepsilon^{-\frac{2}{3}} \leq \beta s\leq -L, \ z\in
[0,\varepsilon^{-\frac{2}{3}}\ell_\varepsilon)\}$, we have
$\chi_\delta=1$ and we can compute that

\begin{equation}\label{equouttildainpro}
\begin{array}{rl}
  -\varepsilon^{-\frac{2}{3}}\tilde{u}_{out}\left(\tilde{u}_{out}^2-a(\varepsilon^\frac{2}{3}y) \right) & =
   -\beta^2\left[a(\varepsilon^\frac{2}{3}y)+\varepsilon^\frac{2}{3}\beta^3
s+\varepsilon^\frac{2}{3} \beta^2
V^2\right]^{\frac{1}{2}}\left[\beta s+
V^2(\beta s) \right] \\
   &   \\
\textrm{via}\ (\ref{eqa>b}),\ (\ref{eqbita})\ &=
-\beta^2\left[\mathcal{O}(\varepsilon^\frac{4}{3}s^2)+\varepsilon^\frac{2}{3}
\beta^2 V^2\right]^{\frac{1}{2}}\left[\beta s+ V^2(\beta s) \right]\\
& \\
&=-\varepsilon^\frac{1}{3}\beta^3 V(\beta
s)\left[1+\varepsilon^\frac{2}{3}\mathcal{O}(s^2V^{-2})
\right]^\frac{1}{2}
\left[\beta s+ V^2(\beta s) \right]\\
 & \\
\textrm{via}\ (\ref{eqvasympto})\ &=-\varepsilon^\frac{1}{3}\beta^3
V(\beta s)\left[1+\varepsilon^\frac{2}{3}\mathcal{O}(|s|) \right]
\left[\beta s+ V^2(\beta s) \right]\\
 & \\
\textrm{via}\ (\ref{eqvasympto}),\ (\ref{eqvestims})\ &  =
-\varepsilon^\frac{1}{3}\beta^3 V(\beta s)\left[\beta s+V^2(\beta s)
    \right]+\mathcal{O}(\varepsilon |s|^{-\frac{1}{2}}), \\
\end{array}
\end{equation}
uniformly, as $\varepsilon\to 0$. Moreover, in the same region, we
have:
\begin{equation}\label{equouttilda-s}
(\tilde{u}_{out})_s=\frac{\varepsilon^\frac{2}{3}}{2}\left[a_t(\varepsilon^\frac{2}{3}s,\varepsilon^\frac{2}{3}z)+\beta^3+2\beta^3
V_x(\beta s)V(\beta s)
\right]\left[a(\varepsilon^\frac{2}{3}y)+\varepsilon^\frac{2}{3}\beta^3
s+\varepsilon^\frac{2}{3} \beta^2 V^2(\beta
s)\right]^{-\frac{1}{2}},
\end{equation}
\begin{equation}\label{equouttilda-ss}
\begin{array}{lll}
  (\tilde{u}_{out})_{ss} & = &
\frac{\varepsilon^\frac{2}{3}}{2}\left[\varepsilon^\frac{2}{3}a_{tt}(\varepsilon^\frac{2}{3}s,\varepsilon^\frac{2}{3}z)
    +2\beta^4(V_{xx}V+
V_x^2)
\right]\left[a(\varepsilon^\frac{2}{3}y)+\varepsilon^\frac{2}{3}\beta^3
s+\varepsilon^\frac{2}{3} \beta^2 V^2\right]^{-\frac{1}{2}}

 \\
    &   &   \\
    &   &  -\frac{\varepsilon^\frac{4}{3}}{4}\left[a_t(\varepsilon^\frac{2}{3}s,\varepsilon^\frac{2}{3}z)+\beta^3+2\beta^3
V_xV
\right]^2\left[a(\varepsilon^\frac{2}{3}y)+\varepsilon^\frac{2}{3}\beta^3
s+\varepsilon^\frac{2}{3} \beta^2 V^2\right]^{-\frac{3}{2}},
\end{array}
\end{equation}
\begin{equation}\label{equouttilde-z}
(\tilde{u}_{out})_z=\frac{1}{2}\left[\varepsilon^\frac{2}{3}a_\theta+3\varepsilon^\frac{4}{3}
\beta'\beta^2s+2\varepsilon^\frac{4}{3}\beta'\beta
V^2+2\varepsilon^\frac{4}{3}\beta'\beta^2 sV_xV
\right]\left[a+\varepsilon^\frac{2}{3}\beta^3
s+\varepsilon^\frac{2}{3} \beta^2 V^2\right]^{-\frac{1}{2}},
\end{equation}
and
\begin{equation}\label{equouttilda-zz}
\begin{array}{lll}
  (\tilde{u}_{out})_{zz} & = &
\frac{1}{2}\left[\varepsilon^\frac{4}{3}a_{\theta\theta}
    +6\varepsilon^2(\beta')^2\beta s+3\varepsilon^2\beta''\beta^2s+2\varepsilon^2V^2\left(\beta''\beta
+(\beta')^2\right)+8\varepsilon^2(\beta')^2\beta sV_xV
    \right.\\
    && \\
    && \left.+2\varepsilon^2\beta''\beta^2 s V_xV
+2\varepsilon^2(\beta')^2\beta^2s^2(V_x^2+V_{xx}V)
    \right]\left[a+\varepsilon^\frac{2}{3}\beta^3
s+\varepsilon^\frac{2}{3} \beta^2 V^2\right]^{-\frac{1}{2}}\\
&& \\
  & &-\frac{1}{4}
  \left[\varepsilon^\frac{2}{3}a_\theta+3\varepsilon^\frac{4}{3}\beta' \beta^2s
+2\varepsilon^\frac{4}{3}\beta'\beta( V^2+ \beta sV_xV)
\right]^2\left[a+\varepsilon^\frac{2}{3}\beta^3
s+\varepsilon^\frac{2}{3} \beta^2 V^2\right]^{-\frac{3}{2}},
\end{array}
\end{equation}
where the functions $a,\ a_{\theta}\ a_{\theta\theta}$ are evaluated
at $\varepsilon^\frac{2}{3}y$, $V$ at $\beta s$, and $\beta$ at
$\varepsilon^\frac{2}{3}z$.
 Estimating as in (\ref{equouttildainpro}), we conclude that in the region
described by
 $\{-\delta\varepsilon^{-\frac{2}{3}}\leq \beta s\leq -L,\ z\in [0,\varepsilon^{-\frac{2}{3}}
 \ell_\varepsilon)\}$ we have:
\begin{equation}\label{equouts1}
\begin{array}{ll}
  (\tilde{u}_{out})_s=\varepsilon^\frac{1}{3}\beta^2 V_x(\beta
s)+\mathcal{O}(\varepsilon |s|^\frac{1}{2}), &
 (\tilde{u}_{out})_{ss}=\varepsilon^\frac{1}{3}\beta^3 V_{xx}(\beta s)+\mathcal{O}(\varepsilon |s|^{-\frac{1}{2}}),  \\
   \\
  (\tilde{u}_{out})_z=\mathcal{O}(\varepsilon |s|^\frac{1}{2}), &
  (\tilde{u}_{out})_{zz}=\mathcal{O}(\varepsilon |s|^{-\frac{1}{2}}), \\
\end{array}
\end{equation}
uniformly as $\varepsilon\to 0$. We point out that, when estimating
derivatives in $z$, we also made use of the bound
\[|a_\theta(\varepsilon^\frac{2}{3}s,\varepsilon^\frac{2}{3}z)|\leq
C\varepsilon^\frac{2}{3}|s|,\] which follows directly from
(\ref{eqa>b}).
In order to elucidate the various cancelations of powers of $|s|$
involved, let us carefully present the proof of one of the estimates
in (\ref{equouts1}):
\[
\begin{array}{lll}
  \left|(\tilde{u}_{out})_{zz}\right| & \leq &
C\left(\varepsilon^\frac{4}{3}+ \varepsilon^2 |s|+\varepsilon^2
|s||s|^{-\frac{1}{2}}|s|^{\frac{1}{2}}
+\varepsilon^2s^2(|s|^{-1}+|s|^{-\frac{3}{2}}|s|^\frac{1}{2})
    \right)\varepsilon^{-\frac{1}{3}}|s|^{-\frac{1}{2}}\\
&& \\
  & &+C
  \left(\varepsilon^\frac{4}{3}|s|
+ \varepsilon^\frac{4}{3} |s||s|^{-\frac{1}{2}}|s|^{\frac{1}{2}}
\right)^2\varepsilon^{-1}|s|^{-\frac{3}{2}}\\
&&\\
&\leq& C
\left(\varepsilon^\frac{4}{3}+\varepsilon^2|s|\right)\varepsilon^{-\frac{1}{3}}|s|^{-\frac{1}{2}}+C\varepsilon^\frac{8}{3}|s|^2\varepsilon^{-1}|s|^{-\frac{3}{2}}\\
&&\\
&\leq& C\varepsilon
|s|^{-\frac{1}{2}}+C\varepsilon^\frac{5}{3}|s|^\frac{1}{2}\leq
C\varepsilon
|s|^{-\frac{1}{2}}+C\varepsilon^\frac{5}{3}|s||s|^{-\frac{1}{2}}\leq
C \varepsilon |s|^{-\frac{1}{2}}.
\end{array}
\]

In $\{-2\delta\varepsilon^{-\frac{2}{3}} \leq \beta s\leq
-\delta\varepsilon^{-\frac{2}{3}}, \ z\in
[0,\varepsilon^{-\frac{2}{3}}\ell_\varepsilon)\}$, making again use
of (\ref{eqvasympto}), (\ref{eqvestims}), we can show that
\begin{equation}\label{equouttildainpro2}
\begin{array}{lll}
  \tilde{u}_{out}^2-a(\varepsilon^\frac{2}{3}y)=\mathcal{O}(\varepsilon^2),&(\tilde{u}_{out})_s=\mathcal{O}(\varepsilon^\frac{2}{3}), & (\tilde{u}_{out})_{ss}=\mathcal{O}(\varepsilon^\frac{4}{3}), \\
    &   &\\
  &(\tilde{u}_{out})_{z}=\mathcal{O}(\varepsilon^\frac{2}{3}), &
  (\tilde{u}_{out})_{zz}=\mathcal{O}(\varepsilon^\frac{4}{3}),
   \end{array}
\end{equation}
uniformly as $\varepsilon\to 0$. Estimate (\ref{eqremuouttilda}) now
follows from (\ref{eqLaplace}), (\ref{eqB0}), the fact that $V$
solves (\ref{eqpainleve}), (\ref{eqvestims}),
(\ref{equouttildainpro}), and (\ref{equouts1}); estimate
(\ref{eqremuouttilda2}) follows at once from (\ref{eqLaplace}),
(\ref{eqB0}), (\ref{equouttilda}), and (\ref{equouttildainpro2}).
The first relation in (\ref{equout-uingradients}) can be shown by
working as in (\ref{equouttildainpro}). The remaining estimates in
(\ref{equout-uingradients}) follow  directly from
(\ref{eqvasympto}), (\ref{eqvestims}), (\ref{equin}),
(\ref{equouts1}), and the fact that
\begin{equation}\label{eqsz}
u_{y_1}^2+u_{y_2}^2=u_s^2+(1+\varepsilon^\frac{2}{3}ks)^2u_z^2,
\end{equation}
for $y=(y_1,y_2)\in \{(s,z)\ |\ |s|\leq
\delta_0\varepsilon^{-\frac{2}{3}},\ \ z\in
[0,\varepsilon^{-\frac{2}{3}}\ell_\varepsilon)\}$, and any smooth
function $u$ defined in this region (recall (\ref{eqmetric})).
Finally, estimate (\ref{equouttilda-a}) follows readily from
(\ref{eqa>b}), and (\ref{eqvestims}).

 The proof
of the proposition is complete.
\end{proof}
\subsection{The approximate solution $u_{ap}$}\label{secmatching}
In this subsection we will construct a smooth approximate solution
for problem (\ref{eqEqstretched}) that is valid in all of
$\mathbb{R}^2$. This will be achieved by smoothly interpolating
between $\tilde{u}_{out}$ and $u_{in}$ in
$\tilde{\mathcal{D}}_\varepsilon$, and between $u_{in}$ and zero in
$\mathbb{R}^2\backslash\tilde{\mathcal{D}}_\varepsilon$.

To this end, we need to introduce one more smooth cutoff function:
\begin{equation}\label{eqcutoffrho}
\rho_L(x)=\left\{
\begin{array}{ll}
  0, & x\geq -L, \\
    &   \\
  1, & x\leq -2L.
\end{array}
\right.
\end{equation}

We can now define our approximate solution for (\ref{eqEqstretched})
as
\begin{equation}\label{equap}
u_{ap}= \left\{
\begin{array}{ll}
  \tilde{u}_{out} & \textrm{in}\  \tilde{\mathcal{D}}_\varepsilon\backslash\{-2L<x<0\}, \\
    &   \\
  u_{in}+\rho_L(x)(\tilde{u}_{out}-u_{in}) &\textrm{in}\  \{-2L\leq x \leq \delta \varepsilon^{-\frac{2}{3}}\}, \\
    &   \\
   \chi_{10\delta} (x)u_{in} & \textrm{everywhere\ else},
\end{array}
\right.
\end{equation}
where $u_{in},\ \chi_\delta,\ \tilde{u}_{out}$ were defined in
(\ref{equin}), (\ref{eqcutoffchi}), and (\ref{equouttilda})
respectively.

The following proposition contains the main estimates concerning
$u_{ap}$.

\begin{pro}\label{prouap}
The approximate solution $u_{ap}$ satisfies
\begin{equation}\label{eqremuap}
\Delta
{u}_{ap}-\varepsilon^{-\frac{2}{3}}{u}_{ap}\left({u}_{ap}^2-a(\varepsilon^\frac{2}{3}y)\right)=\mathcal{O}\left(\varepsilon
(|s|+1)^{-\frac{1}{2}}\right)
\end{equation}
uniformly in $\{-\delta\varepsilon^{-\frac{2}{3}} \leq \beta s\leq
0, \ z\in [0,\varepsilon^{-\frac{2}{3}}\ell_\varepsilon)\}$,
\begin{equation}\label{eqremuap2}
\Delta
{u}_{ap}-\varepsilon^{-\frac{2}{3}}{u}_{ap}\left({u}_{ap}^2-a(\varepsilon^\frac{2}{3}y)\right)=\mathcal{O}(\varepsilon^{\frac{4}{3}})
\end{equation}
uniformly in $\tilde{\mathcal{D}}_\varepsilon\backslash
\{-\delta\varepsilon^{-\frac{2}{3}} \leq \beta s<0, \ z\in
[0,\varepsilon^{-\frac{2}{3}}\ell_\varepsilon)\}$,
\begin{equation}\label{eqremuap3}
\Delta
{u}_{ap}-\varepsilon^{-\frac{2}{3}}{u}_{ap}\left({u}_{ap}^2-a(\varepsilon^\frac{2}{3}y)\right)=\mathcal{O}(\varepsilon
e^{-c|s|^\frac{3}{2}})
\end{equation}
uniformly in $\{0 \leq \beta s\leq
2\delta\varepsilon^{-\frac{2}{3}}, \ z\in
[0,\varepsilon^{-\frac{2}{3}}\ell_\varepsilon)\}$, as
$\varepsilon\to 0$, and
\begin{equation}\label{eqremuap4}
{u}_{ap}=0 \ \ \textrm{everywhere else}.
\end{equation}

If $\varepsilon>0$ is sufficiently small, we have
\begin{equation}\label{equappotential}
3u_{ap}^2-a(\varepsilon^\frac{2}{3}y)\geq \left\{
\begin{array}{ll}
  c \varepsilon^\frac{2}{3}(1+|x|), &  \textrm{if} \ \ |x|\leq \delta \varepsilon^{-\frac{2}{3}}, \\
    &   \\
  c+c|\varepsilon^\frac{2}{3}y|^p,  & \textrm{otherwise}.
\end{array}\right.
\end{equation}
\end{pro}
\begin{proof}
We will first consider relations
(\ref{eqremuap})--(\ref{eqremuap4}). In view of estimates
(\ref{eqremuin}), (\ref{equinremexplicit}), (\ref{eqremuouttilda}),
(\ref{eqremuouttilda2}), and recalling the super-exponential decay
of $V$ as $x\to \infty$, it just remains to show the validity of
(\ref{eqremuap}) in the interpolating region described by $\{-2L\leq
\beta s\leq-L, \ z\in
[0,\varepsilon^{-\frac{2}{3}}\ell_\varepsilon)\}$. There, we have
\[
\begin{array}{lll}
  \Delta
u_{ap}-\varepsilon^{-\frac{2}{3}}u_{ap}\left(u_{ap}^2-a(\varepsilon^\frac{2}{3}y)
\right) & = & \Delta
u_{in}-\varepsilon^{-\frac{2}{3}}u_{in}\left(u_{in}^2-a(\varepsilon^\frac{2}{3}y)
\right)+(\Delta \rho_L)(\tilde{u}_{out}-u_{in})
 \\
    &   &   \\
    &   & +2\nabla \rho_L
\nabla(\tilde{u}_{out}-u_{in})+\rho_L\Delta (\tilde{u}_{out}-u_{in})
 \\
    &   &   \\
    &   & -\varepsilon^{-\frac{2}{3}}u_{in}\left[
3\rho_L^2(\tilde{u}_{out}-u_{in})^2+2u_{in}\rho_L(\tilde{u}_{out}-u_{in})\right]
  \\
    &   &  \\
& &
-\varepsilon^{-\frac{2}{3}}\rho_L(\tilde{u}_{out}-u_{in})\left[u_{in}^2-a(\varepsilon^\frac{2}{3}y)
+\rho_L^2(\tilde{u}_{out}-u_{in})^2\right],
\end{array}
\]
and the desired estimate follows via (\ref{eqremuin}),
(\ref{equout-uingradients}), noting that $u_{in}^2,\
a(\varepsilon^\frac{2}{3}y)$ are of order $\varepsilon^\frac{2}{3}$
in this region.

 The proof of lower bound (\ref{equappotential})
proceeds as follows: In the neighborhood described by $\{|x|\leq
2L\}$ of the curve $\tilde{\Gamma}_\varepsilon$  (with the obvious
notation), by virtue of (\ref{eqa>b}), (\ref{eqbita}),
(\ref{equout-uingradients}), (\ref{equap}), we find that
\begin{equation}\label{eq3uap1}
\begin{array}{lll}
  3u_{ap}^2-a(\varepsilon^\frac{2}{3}y) & = & 3\varepsilon^\frac{2}{3}\beta^2V^2(x)+\mathcal{O}(\varepsilon^\frac{4}{3})+
  \varepsilon^\frac{2}{3}\beta^2x+\mathcal{O}(\varepsilon^\frac{4}{3})x^2 \\
    &   &   \\
    & = & \varepsilon^\frac{2}{3}\beta^2\left(3V^2(x)+x
    \right)+\mathcal{O}(\varepsilon^\frac{4}{3})
\end{array}
\end{equation}
uniformly as $\varepsilon\to 0$. In $\tilde{\mathcal{D}}_\varepsilon
\backslash\{-2L<x<0\}$, we have $u_{ap}=\tilde{u}_{out}$ and, by
(\ref{eqvestims}), (\ref{equouttilda}), we infer that
\begin{equation}\label{eq3uap2}
\begin{array}{lll}
  3{u}_{ap}^2-a(\varepsilon^\frac{2}{3}y) & = & 2a(\varepsilon^\frac{2}{3}y)+3\varepsilon^\frac{2}{3}\beta^2
  \chi_\delta(\beta s)
  \left(V^2(\beta s)+\beta s \right) \\
    &   &   \\
    & =  &
    2a(\varepsilon^\frac{2}{3}y)+\varepsilon^\frac{2}{3}\chi_\delta(\beta s)\mathcal{O}(|s|^{-2})
\end{array}
\end{equation}
uniformly as $\varepsilon \to 0$. In points outside of the domain
$\tilde{\mathcal{D}}_\varepsilon \cup \{0\leq x < 2L\}$, we plainly
note that
\begin{equation}\label{eq3uap3}
3u_{ap}^2-a(\varepsilon^\frac{2}{3}y)\geq
-a(\varepsilon^\frac{2}{3}y).
\end{equation}
The desired lower bound (\ref{equappotential}) now follows readily
from the above three relations, via (\ref{eqainfinity}),
(\ref{eqa>b}), and (\ref{eqvpeli}), increasing $L>0$ if necessary.

The proof of the proposition is complete.
\end{proof}
\begin{rem}\label{remcornerlemma}
From the geometric singular perturbation viewpoint, recall Remark
\ref{remgeometric}, matching is accomplished by employing a useful
lemma on the flow past a ``corner equilibrium'' (see
\cite{schecterdafermos}). Manifolds of corner equilibria arise in
blown--up geometric singular perturbation problems precisely where
the inner and outer solutions must be matched. When such equilibria
are normally hyperbolic, as in the one-dimensional case of the
problem at hand (see \cite{schecter-sourdis}), this lemma plays the
same role in tracking the flow past them that the Exchange Lemma
\cite{jones} plays at certain other manifolds of equilibria.
\end{rem}
\begin{rem}\label{remmargetis}
Our construction of $u_{ap}$ should also be applicable to the
homogenized Gross-Pitaevskii equations considered in \cite[Sec.
7]{margetis2012}.
\end{rem}
\subsection{Mapping properties of the linearized operator}\label{seclinear}
In this subsection we will invert the linearized operator
\begin{equation}\label{eqLoper}
\mathcal{L}(\varphi)=\Delta\varphi-\varepsilon^{-\frac{2}{3}}\left(3u_{ap}^2-a(\varepsilon^\frac{2}{3}y)\right)\varphi
\end{equation}
in carefully chosen weighted spaces. The use of weighted spaces is a
powerful technique in elliptic singular perturbation problems, and
in many problems arising from geometry, see \cite{pacard-riviere}.
To the best of our knowledge, they are used here for the first time
in singular perturbation problems involving corner layers. Actually,
the weighted spaces that we will use are a variant of those
considered in \cite{rebai}, and are motivated from relations
(\ref{eqremuouttilda}), (\ref{equappotential}), keeping in mind that
we ultimately wish to find a true solution of (\ref{eqEqstretched})
near $u_{ap}$ via a perturbation argument.

Consider a smooth non-increasing function $g$ such that
\begin{equation}\label{eqg}
g(s)=\left\{
\begin{array}{ll}
 1, & s\geq 0, \\
    &   \\
  \left(\frac{\max \beta}{L} \right)^\frac{3}{2}\left(\frac{L}{\max \beta}-s \right)^\frac{3}{2},  & -\frac{2\delta}{\min \beta}
  \varepsilon^{-\frac{2}{3}}\leq s\leq -\frac{L}{\max \beta}, \\
   &  \\
  \left(\frac{\max \beta}{\min \beta} \right)^\frac{3}{2}L^{-\frac{3}{2}}\varepsilon^{-1}, & s\leq -\frac{3\delta}{\min \beta}
  \varepsilon^{-\frac{2}{3}},
\end{array}
\right.
\end{equation}
and
\begin{equation}\label{eqg1}
0\leq -g'\leq CL^{-1},\ \ |g''|\leq CL^{-2}, \ \
s\in\left[-\frac{L}{\max \beta},0\right],
\end{equation}
\begin{equation}\label{eqg2}
0\leq -g'\leq CL^{-\frac{3}{2}}\varepsilon^{-\frac{1}{3}},\ \
|g''|\leq CL^{-\frac{3}{2}}\varepsilon^\frac{1}{3},\ \
s\in\left[-\frac{3\delta}{\min \beta}
  \varepsilon^{-\frac{2}{3}},-\frac{2\delta}{\min \beta}
  \varepsilon^{-\frac{2}{3}}\right],
\end{equation}
where the constant $C$ is independent of small $\varepsilon$ and
large $L$. Recalling (\ref{equappotential}), it is easy to check
that we can fix an $L_0>0$ such that
\begin{equation}\label{eqgmain1}
\left|\frac{g''}{g} \right|+2\left|\frac{g'}{g} \right|^2\leq
C\varepsilon^\frac{4}{3}+CL^{-2}
\leq\frac{\varepsilon^{-\frac{2}{3}}}{2}\min_{\mathbb{R}^2}\left(3u_{ap}^2-a(\varepsilon^\frac{2}{3}y)\right),
\ \ s\in \mathbb{R},\ \textrm{if}\ L\geq L_0,
\end{equation}
provided $\varepsilon$ is sufficiently small ($C$ in the above
relation is independent of $\varepsilon,L$). Relations
(\ref{eq3uap1})--(\ref{eq3uap3}) imply that, for small
$\varepsilon$, we have
\begin{equation}\label{eqgmain2}
g^\frac{2}{3}(y)\leq
C\varepsilon^{-\frac{2}{3}}\left(3u_{ap}^2-a(\varepsilon^\frac{2}{3}y)\right)\
\ \textrm{in}\ \mathbb{R}^2,
\end{equation}
(here $g$ is viewed as a smooth function of $y$, which close to
$\tilde{\Gamma}_\varepsilon$, in coordinates $(s,z)$, is given by
(\ref{eqg}), and otherwise equals the constants in (\ref{eqg})).

For $\varphi\in L^\infty(\mathbb{R}^2)$, we define the following
weighted norms:
\begin{equation}\label{eqgnorms}
\|\varphi\|_\frac{3}{2}\equiv\|g\varphi\|_{L^\infty(\mathbb{R}^2)}\
\ \textrm{and} \ \
\|\varphi\|_\frac{1}{2}\equiv\|g^\frac{1}{3}\varphi\|_{L^\infty(\mathbb{R}^2)}.
\end{equation}
(We utilized this notation because $g$ behaves qualitatively like
$(-s)^\frac{3}{2}$ for $s<0$).

We also consider the Banach space
\begin{equation}\label{Xspace}
\mathcal{X}\equiv \{\varphi \ : \ \|\varphi\|_\mathcal{X}\equiv
\|e^{|\varepsilon^\frac{2}{3}y|}\varphi\|_{L^\infty(\mathbb{R}^2)}
\}<\infty.
\end{equation}

The following proposition will be used essentially in the sequel.
\begin{pro}\label{proL}
If $\varepsilon$ is sufficiently small, given $f\in \mathcal{X}\cap
C^\alpha(\mathbb{R}^2)$, $0<\alpha<1$, there exists a unique
$\varphi\in \mathcal{X}\cap C^{2+\alpha}(\mathbb{R}^2)$ such that
\begin{equation}\label{eqLphi=f}
\mathcal{L}(\varphi)=f,
\end{equation}
where the linear operator $\mathcal{L}$ was defined in
(\ref{eqLoper}). Furthermore, we have
\begin{equation}\label{eqapriori-}
\|\varphi\|_\mathcal{X}\leq C \|f\|_\mathcal{X},
\end{equation}
and
\begin{equation}\label{eqapriori}
\|\varphi\|_\frac{3}{2}\leq C\|f\|_\frac{1}{2},
\end{equation}
for  some constant $C$ independent of $f,\varepsilon$.
\end{pro}
\begin{proof}
The first assertion of the proposition, including estimate
(\ref{eqapriori-}), follows in a standard way: It follows from
(\ref{equappotential}), the maximum principle, and elliptic
regularity theory \cite{Gilbarg-Trudinger},  that there exists a
solution $\varphi \in C^{2+\alpha}(\mathbb{R}^2)$ of
(\ref{eqLphi=f}) such that
\begin{equation}\label{eqphinfinity}
\|\varphi\|_{L^\infty(\mathbb{R}^2)}\leq
C\|f\|_{L^\infty(\mathbb{R}^2)},
\end{equation}
 for  some constant $C$ independent
of $f,\varepsilon$. This is easy to prove, though it is difficult to
find a good reference. (For example, one can first solve equation
(\ref{eqLphi=f}) in a ball $B_R$ with Dirichlet boundary conditions
to obtain a solution $\varphi_R$ such that
$\|\varphi_R\|_{L^\infty(B_R)}\leq C \|f\|_{L^\infty(B_R)}$, for
some $C$ independent of $f,\varepsilon,R$, extend $\varphi_R$ in
$L^\infty(\mathbb{R}^2)$ to be zero outside of $B_R$, and prove that
$\varphi_{R_i}$, for some $R_i\to \infty$, converge uniformly in
compact sets of $\mathbb{R}^2$ to a solution $\varphi$ of
(\ref{eqLphi=f}) that satisfies (\ref{eqphinfinity})). Then, a
standard barrier argument, using as barrier the function
\begin{equation}\label{eqbarrier}\bar{\varphi}_\tau(y)=\tau
e^{|\varepsilon^{2/3}y|}+\|\varphi\|_{L^\infty(\mathbb{R}^2)} e^{
(R-|\varepsilon^{2/3}y|)},\ \ |y|\geq
\frac{R}{\varepsilon^{2/3}},\end{equation} where $R>0$ is chosen
large, $\tau>0$ arbitrary, yields that $|\varphi(y)|\leq
\bar{\varphi}_\tau(y),\ |y|\geq R\varepsilon^{-2/3}$, provided
$\varepsilon$ is small enough so that (\ref{equappotential}) holds.
Lastly, letting $\tau\to 0$, and recalling (\ref{eqphinfinity}), we
conclude that estimate (\ref{eqapriori-}) holds true (see also
\cite[Lemma 7.3]{delpinotoda}).

Let
\begin{equation}\label{eqpsi}
\psi=g\varphi,
\end{equation}
then, thanks to (\ref{eqLaplace}), (\ref{eqB0}), it is easy to see
that, with the obvious notation, we have
\begin{equation}\label{eqpsiinner}
\Delta_{y}\psi-2\frac{g'}{g}\psi_s-\varepsilon^{-\frac{2}{3}}\left(3u_{ap}^2-a(\varepsilon^\frac{2}{3}y)\right)\psi-\frac{g''}{g}\psi+2
\left(\frac{g'}{g}\right)^2\psi-\varepsilon^\frac{2}{3}a_0\frac{g'}{g}\psi=gf,
\end{equation}
$y\in \mathbb{R}^2$. Since $\psi\to 0$ as $|y|\to \infty$, recalling
(\ref{equappotential}), (\ref{eqgmain1}), (\ref{eqgmain2}), we can
apply the maximum principle to show that, for small $\varepsilon$,
\[ |\psi(y)|\leq
C\|g^\frac{1}{3}f\|_{L^\infty(\mathbb{R}^2)},\ \ y\in \mathbb{R}^2,
\]
for some constant $C$ independent of $\varepsilon,f$. We also used
the fact that $\psi_s=0$ whenever $\nabla_{y}\psi=0$, which follows
immediately from relation (\ref{eqsz}). The desired bound
(\ref{eqapriori}) now follows at once from (\ref{eqgnorms}),
(\ref{eqpsi}), and the above relation.

The proof of the proposition is complete.
\end{proof}
\begin{rem}\label{rempotential}
In the above proof, we made essential use of lower bound
(\ref{equappotential}) whose proof, we recall, relied crucially on
lower bound (\ref{eqvpeli}) which was established recently in
\cite{pelinovsky}. However, as we have remarked in the proof of
Proposition \ref{prohastings}, its proof is rather involved and
technical. In Appendix \ref{ap2}, we will provide a more natural and
flexible proof of Proposition \ref{proL} \emph{without} assuming
knowledge of (\ref{eqvpeli}). Instead, we will make use of the
asymptotic stability of the Hastings-McLeod solution $V$, namely the
fact that the principal eigenvalue of the operator $-\mathcal{M}$,
defined in (\ref{eqMcal}), is strictly positive.
This follows immediately by testing the corresponding eigenvalue
problem by $V_x<0$, see \cite{sourdis-fife} and Proposition
\ref{proentire} herein. In other words, we will rely on the lower
bound:
\begin{equation}\label{eqsofter}
-\int_{-\infty}^{\infty}\phi \mathcal{M}(\phi) dx\geq c
\int_{-\infty}^{\infty}\phi^2dx\ \ \ \forall \phi \in
C_0^\infty(\mathbb{R}),
\end{equation}
for some constant $c>0$, which clearly is much ``softer'' than
(\ref{eqvpeli}).
 We point out that the validity of (\ref{eqvpeli})
was not needed (nor known) in references \cite{schecter-sourdis},
\cite{sourdis-fife} which dealt with related one-dimensional
problems. Actually, in one--dimensional or radially symmetric cases,
for Proposition \ref{proL} to hold, it suffices to know that zero is
not in the kernel of $\mathcal{M}$. However, this may \emph{not} be
true in general higher dimensional problems due to a possible
resonance phenomenon (see \cite{del pino cpam},
\cite{karalisourdisresonance}).
\end{rem}
%
%
\subsection{Existence of a solution}\label{secexiste}
Here we will use the contraction mapping principle in order to
capture a genuine solution $u_\varepsilon$ of (\ref{eqEqstretched})
close to the approximate solution $u_{ap}$, provided $\varepsilon$
is sufficiently small.

\begin{pro}\label{thmexistence1}
If $\varepsilon$ is sufficiently small, then there exists a solution
$u_\varepsilon$ of (\ref{eqEqstretched}) such that
\begin{equation}\label{equepsilon-X}
\|u_\varepsilon-u_{ap}\|_\mathcal{X}\leq C\varepsilon,
\end{equation}
\begin{equation}\label{equepsilon-g}
\|u_\varepsilon-u_{ap}\|_\frac{3}{2}\leq C\varepsilon,
\end{equation}
and
\begin{equation}\label{equepsilondecay}
|u_\varepsilon(y)-u_{ap}(y)|\leq C\varepsilon e^{-cs},\ \
y\in\{0\leq \beta s\leq 2\delta\varepsilon^{-\frac{2}{3}},\ z\in
[0,\varepsilon^{-\frac{2}{3}}\ell_\varepsilon)\},
\end{equation}
 where the norms involved were defined in (\ref{eqgnorms}) and
(\ref{Xspace}).
\end{pro}
\begin{proof}
We seek a true solution of problem (\ref{eqEqstretched}) in the form
\[
u=u_{ap}+\varphi.
\]
In order for $u$ to satisfy the equation in (\ref{eqEqstretched}),
we readily find that the correction $\varphi$ has to solve
\begin{equation}\label{eqLphi=E+N} \mathcal{L}(\varphi)=E+N(\varphi),
\end{equation}
where the linear operator $\mathcal{L}$ was defined in
(\ref{eqLoper}),
\begin{equation}\label{eqE}
E=-\Delta
u_{ap}+\varepsilon^{-\frac{2}{3}}u_{ap}\left(u_{ap}^2-a(\varepsilon^\frac{2}{3}y)
\right),
\end{equation}
and
\begin{equation}\label{eqN}
N(\varphi)=3\varepsilon^{-\frac{2}{3}}u_{ap}\varphi^2+\varepsilon^{-\frac{2}{3}}\varphi^3.
\end{equation}

Given $M>0$ to be determined (independently of $\varepsilon$), we
consider the closed bounded subset of $\mathcal{X}$ defined by
\begin{equation}\label{eqY}
\mathcal{Y}=\{\varphi\in \mathcal{X}\ :\ \|\varphi\|_\mathcal{X}\leq
M\varepsilon,\ \|\varphi\|_\frac{3}{2}\leq M\varepsilon\}.
\end{equation}
We will prove that, if $M$ is chosen sufficiently large, the
operator $\mathbb{P}$, defined from $\mathcal{X}\cap
C^{2+\alpha}(\mathbb{R}^2)$ into $\mathcal{X}\cap
C^{2+\alpha}(\mathbb{R}^2)$ by
\[
\mathbb{P}(\varphi)=\mathcal{L}^{-1}\left(E+N(\varphi) \right),
\]
maps $\mathcal{Y}$ into itself, and is a contraction with respect to
the $\mathcal{X}$-norm, provided $\varepsilon$ is sufficiently
small. Note that $\mathcal{L}^{-1}$ is well defined by virtue of
Proposition \ref{proL}. In view of the estimates of Proposition
\ref{prouap}, (\ref{eqg}), (\ref{eqgnorms}), and (\ref{Xspace}), for
small $\varepsilon>0$, we have
\begin{equation}\label{eqEestims}
\|E\|_\mathcal{X}\leq C \varepsilon  \ \ \textrm{and}\ \
\|E\|_\frac{1}{2}\leq C \varepsilon.
\end{equation}
Furthermore, there exists a constant $C>0$ such that, for all
$\varphi_1,\varphi_2, \varphi\in \mathcal{X}$, the following
relations hold pointwise:
\begin{equation}\label{eqNlipsitz}
\begin{array}{l}
  |N(\varphi_1)-N(\varphi_2)|\leq
C\varepsilon^{-\frac{2}{3}}(\varphi_1^2+\varphi_2^2)|\varphi_1-\varphi_2|+C\varepsilon^{-\frac{2}{3}}
|u_{ap}|(|\varphi_1|+|\varphi_2|)|\varphi_1-\varphi_2|,
 \\
   \\
  |g^\frac{1}{3}N(\varphi)|\leq
C\varepsilon^{-\frac{2}{3}}|g\varphi|^3+
C\varepsilon^{-\frac{2}{3}}|u_{ap}||g\varphi|^2,
\end{array}
\end{equation}
for every $y\in \mathbb{R}^2$ (recall that $g\geq 1$). If
$\varphi\in \mathcal{Y}$, by Proposition \ref{proL},
(\ref{eqEestims}), and (\ref{eqNlipsitz}), we obtain that
\[\begin{array}{lll}
    \|\mathbb{P}(\varphi)\|_\mathcal{X} & \leq & C \|E\|_\mathcal{X}+C\|N(\varphi)\|_\mathcal{X} \\
     &  &  \\
      &  \leq & C\varepsilon+C\varepsilon^{-\frac{2}{3}}\|\varphi\|_\mathcal{X}^3+C\varepsilon^{-\frac{2}{3}}\|\varphi\|_\mathcal{X}^2 \\
     &  &  \\
      & \leq  &
      C\varepsilon+CM^3\varepsilon^\frac{7}{3}+CM^2\varepsilon^\frac{4}{3},
  \end{array}
\]
and
\[\begin{array}{lll}
    \|\mathbb{P}(\varphi)\|_{\frac{3}{2}} & \leq & C \|E\|_\frac{1}{2}+C\|N(\varphi)\|_\frac{1}{2} \\
     &  &  \\
      &  \leq & C\varepsilon+C\varepsilon^{-\frac{2}{3}}\|\varphi\|_\frac{3}{2}^3+C\varepsilon^{-\frac{2}{3}}\|\varphi\|_\frac{3}{2}^2 \\
     &  &  \\
      & \leq  &
      C\varepsilon+CM^3\varepsilon^\frac{7}{3}+CM^2\varepsilon^\frac{4}{3},
  \end{array}
\]
 where $C$ is independent of $\varphi, M$ and small
$\varepsilon$. We conclude that, if $M$ is chosen sufficiently
large, the operator $\mathbb{P}$ maps $\mathcal{Y}$ into itself,
provided $\varepsilon$ is sufficiently small. We have to prove that
$\mathbb{P}$ is a contraction from $\mathcal{Y}$ into itself with
respect to the $\mathcal{X}$-norm. Let $\varphi_1,\varphi_2\in
\mathcal{Y}$. As before, we have
\[\begin{array}{lll}
    \|\mathbb{P}(\varphi_1)-\mathbb{P}(\varphi_2)\|_\mathcal{X} & \leq & C \|N(\varphi_1)-N(\varphi_2)\|_\mathcal{X} \\
      &   &   \\
      &  \leq &
      C\varepsilon^\frac{1}{3}\|\varphi_1-\varphi_2\|_\mathcal{X}.
  \end{array}
\]
Hence, for small $\varepsilon>0$, the operator
$\mathbb{P}:\mathcal{Y}\to \mathcal{Y}$ becomes a contraction with
respect to the $\mathcal{X}$-norm. So, recalling that $\mathcal{Y}$
is closed in $\mathcal{X}$, it has a unique fixed point
$\varphi_*\in \mathcal{Y}$, thanks to the contraction mapping
theorem (see for instance \cite{Gilbarg-Trudinger}). It is clear
that the function $u_\varepsilon\equiv u_{ap}+\varphi_*$ satisfies
the elliptic equation in (\ref{eqEqstretched}), and estimates
(\ref{equepsilon-X})--(\ref{equepsilon-g}).

Next, we show that $u_\varepsilon$ is positive, and consequently
solves problem (\ref{eqEqstretched}). In the neighborhood described
by $\{|x|\leq 2L\}$ of the curve $\tilde{\Gamma}_\varepsilon$,
recalling (\ref{equin}), (\ref{equout-uingradients}), (\ref{equap}),
and that $\varphi_*\in \mathcal{Y}$, we have
\[
u_\varepsilon=u_{in}+\mathcal{O}(\varepsilon)=\varepsilon^\frac{1}{3}\beta
V(x)+\mathcal{O}(\varepsilon)\geq c \varepsilon^\frac{1}{3}
\]
uniformly as $\varepsilon\to 0$ . In the domain
$\tilde{\mathcal{D}}_\varepsilon \backslash \{-2L<x<0\}$, thanks to
(\ref{eqvestims}), (\ref{equouttilda}), we have
\[
u_\varepsilon=\tilde{u}_{out}+\mathcal{O}(\varepsilon)=
\left[a(\varepsilon^\frac{2}{3}y)+\mathcal{O}(\varepsilon^\frac{2}{3}L^{-2})
\right]^\frac{1}{2}+\mathcal{O}(\varepsilon)\stackrel{(\ref{eqa>b})}{\geq}
c\varepsilon^\frac{1}{3}
\]
uniformly as $\varepsilon\to 0$ (having increased $L$ if necessary).
It remains to consider points outside of the domain
$\tilde{\mathcal{D}}_\varepsilon \cup \{0\leq x <2L\}$, where $u$
solves an equation of the form
\begin{equation}\label{eqlinearp}
\Delta u-p(y)u=0,\ \ \textrm{where}\ \ p\geq c,
\end{equation}
(recall (\ref{eqainfinity}), (\ref{eqa>b})). The positivity of $u$
in this region follows directly from the maximum principle, and the
fact that we have already shown that $u\geq
c\varepsilon^\frac{1}{3}$ on the boundary described by the closed
curve $\{x=2L\}$.

It remains to establish the validity of
(\ref{equepsilon-X})--(\ref{equepsilondecay}). Since
$u_\varepsilon-u_{ap}=\varphi_*\in \mathcal{Y}$, we see that
(\ref{equepsilon-X}), (\ref{equepsilon-g}) hold. Finally, we will
show (\ref{equepsilondecay}) by suitably modifying the proof of
Lemma 2 in \cite{bethuelCVPDE} (see also \cite[pg. 230]{fife-arma}
and \cite[Lem. 2.2]{ignat}). From (\ref{eqremuap3}),
(\ref{eqremuap4}), (\ref{equappotential}), (\ref{equepsilon-X}), and
(\ref{eqLphi=E+N}), we find that $\varphi_*$ satisfies
\begin{equation}\label{eqexpCrude}
\Delta \varphi_*-P(y)\varphi_*=\mathcal{O}\left(\varepsilon
e^{-\frac{\sqrt{c}}{2}s}\right), \ \textrm{where}\ P(y)\geq c,
\end{equation}
uniformly in $\{0\leq \beta s\leq
4\delta\varepsilon^{-\frac{2}{3}},\ z\in
[0,\varepsilon^{-\frac{2}{3}}\ell_\varepsilon)\}$, as
$\varepsilon\to 0$. The reason for choosing, in the righthand side,
a decay rate strictly less than $\sqrt{c}$ is to facilitate  our
next argument. Let
\[
\bar{\varphi}(y)=\bar{M}\varepsilon \left\{e^{-\frac{\sqrt{c}}{2}s}
+
e^{\frac{\sqrt{c}}{2}(s-4\delta\varepsilon^{-\frac{2}{3}}\beta^{-1})}\right\},
\]
$y\in \{0\leq \beta s\leq 4\delta\varepsilon^{-\frac{2}{3}},\ z\in
[0,\varepsilon^{-\frac{2}{3}}\ell_\varepsilon)\}$, where the value
of the large constant $\bar{M}>0$ will soon be fixed independently
of small $\varepsilon$. By virtue of (\ref{equepsilon-X}),
(\ref{eqexpCrude}), we can choose a large $\bar{M}>0$ such that
\[\left\{\begin{array}{l}
    -\Delta (\bar{\varphi}-\varphi_*)+P(y)(\bar{\varphi}-\varphi_*)>0,\ \ y\in \{0\leq \beta
s\leq 4\delta\varepsilon^{-\frac{2}{3}},\ z\in
[0,\varepsilon^{-\frac{2}{3}}\ell_\varepsilon)\}, \\
     \\
   \bar{\varphi}-\varphi_*>0\ \ \textrm{on}\ \ \{s=0\}\cup \{\beta
s=4\delta\varepsilon^{-\frac{2}{3}}\},
  \end{array}
  \right.
\]
if $\varepsilon$ is sufficiently small. Now, by the second estimate
in (\ref{eqexpCrude}), and the maximum principle, we deduce that
\[
\bar{\varphi}-\varphi_*>0\ \ \textrm{if}\ \ 0\leq \beta s\leq
4\delta\varepsilon^{-\frac{2}{3}}.
\]
In turn, the above estimate readily implies the validity of
(\ref{equepsilondecay}).

The proof of the proposition is complete.
\end{proof}
The following estimates hold:
\begin{cor}\label{corestims}
The solution $u_\varepsilon$ of (\ref{eqEqstretched}), constructed
in Proposition \ref{thmexistence1}, satisfies
\begin{equation}\label{eqcorestim1}
u_\varepsilon=\varepsilon^\frac{1}{3}\beta(\varepsilon^\frac{2}{3}z)
V\left(\beta(\varepsilon^\frac{2}{3}z)s\right)+\mathcal{O}(\varepsilon|s|^\frac{3}{2}+\varepsilon)
\end{equation}
uniformly in $\{-\delta\varepsilon^{-\frac{2}{3}}\leq\beta s\leq 0,
\ z\in [0,\varepsilon^{-\frac{2}{3}}\ell_\varepsilon)\}$,
\begin{equation}\label{eqcorestim1+}
u_\varepsilon=\varepsilon^\frac{1}{3}\beta(\varepsilon^\frac{2}{3}z)
V\left(\beta(\varepsilon^\frac{2}{3}z)s\right)+\mathcal{O}(\varepsilon
e^{-cs})
\end{equation}
uniformly in $\{0\leq\beta s\leq 2\delta\varepsilon^{-\frac{2}{3}},
\ z\in [0,\varepsilon^{-\frac{2}{3}}\ell_\varepsilon)\}$,
\begin{equation}\label{eqcorestim2}
u_\varepsilon=\sqrt{a(\varepsilon^\frac{2}{3}y)}+\mathcal{O}(\varepsilon^\frac{1}{3}|s|^{-\frac{5}{2}})
\end{equation}
uniformly in $\{-\delta\varepsilon^{-\frac{2}{3}}\leq\beta s\leq-2L,
\ z\in [0,\varepsilon^{-\frac{2}{3}}\ell_\varepsilon)\}$,
\begin{equation}\label{eq372+}
u_\varepsilon=\sqrt{a(\varepsilon^\frac{2}{3}y)}+\mathcal{O}(\varepsilon^2)
\end{equation}
uniformly in
$\tilde{\mathcal{D}}_\varepsilon\backslash\{-\delta\varepsilon^{-\frac{2}{3}}<\beta
s<0, \ z\in [0,\varepsilon^{-\frac{2}{3}}\ell_\varepsilon)\}$, and
\begin{equation}\label{eqcorexpdec1}
0<u_\varepsilon \leq
C\varepsilon^\frac{1}{3}\exp\{-c\textrm{dist}(y,\tilde{\mathcal{D}}_\varepsilon)\}
\end{equation}
in $\mathbb{R}^2\backslash \tilde{\mathcal{D}}_\varepsilon$, as
$\varepsilon\to 0$.
\end{cor}
\begin{proof}
Estimates (\ref{eqcorestim1}), (\ref{eqcorestim1+}) follow readily
from (\ref{equin}), the first relation in
(\ref{equout-uingradients}), (\ref{equap}), (\ref{equepsilon-X}),
and (\ref{equepsilondecay}). From (\ref{equouttilda-a}),
(\ref{equap}), (\ref{eqg}), (\ref{equepsilon-g}), for small
$\varepsilon>0$, we have
\[
\left|u_\varepsilon-\sqrt{a(\varepsilon^\frac{2}{3}y)}\right|\leq
C\varepsilon
|s|^{-\frac{3}{2}}+C\varepsilon^\frac{1}{3}|s|^{-\frac{5}{2}}\leq
C\varepsilon^\frac{1}{3}|s|^{-\frac{5}{2}}
\]
in $\{-\delta\varepsilon^{-\frac{2}{3}}\leq\beta s\leq-2L, \ z\in
[0,\varepsilon^{-\frac{2}{3}}\ell_\varepsilon)\}$, and estimate
(\ref{eqcorestim2}) follows immediately. Estimate (\ref{eq372+})
follows at once from (\ref{equouttilda}), (\ref{equap}),
(\ref{eqg}), and (\ref{equepsilon-g}). Finally, estimate
(\ref{eqcorexpdec1}) follows readily from (\ref{eqlinearp}),
(\ref{eqcorestim1+}), arguing as we did for the proof of
(\ref{equepsilondecay}) (but here we need to cover $\mathbb{R}^2$ by
a finite number of disjoint annular domains, surrounding
$\mathcal{D}_0$, and the exterior of a large ball), see also
\cite[pg. 230]{fife-arma}.

The proof of the corollary is complete.
\end{proof}
\begin{rem}\label{remmatching}
Notice that the nonlinear terms in (\ref{eqLphi=E+N}) are of
cubic-like order. Indeed,  as in the proof of Proposition
\ref{proL}, using the maximum principle, (\ref{equappotential}), and
the easily derived bound
\begin{equation}\label{eqrem}
|u_{ap}|\leq
C\varepsilon^{-\frac{1}{3}}\left(3u_{ap}^2-a(\varepsilon^\frac{2}{3}y)\right),
\ \ y\in \mathbb{R}^2,
\end{equation}
  we can show that, if $\varepsilon$ is small, the unique solution
$\varphi\in \mathcal{X}\cap C^{2+\alpha}(\mathbb{R}^2)$ of
\begin{equation}\label{eqL=uapf}
\mathcal{L}(\varphi)=u_{ap}f,\ \ f \in \mathcal{X}\cap
C^{\alpha}(\mathbb{R}^2),\ 0<\alpha<1,
\end{equation}
satisfies
\begin{equation}\label{eqL=uapf(assertion)}
\|\varphi\|_{\mathcal{X}}\leq
C\varepsilon^\frac{1}{3}\|f\|_{\mathcal{X}}.
\end{equation}
 Therefore, in order to successfully
apply the contraction mapping principle, as we did in the proof of
Proposition \ref{thmexistence1}, it is enough to construct an
approximate solution $v_{ap}$ supported in a ball of radius
$\mathcal{O}(\varepsilon^{-\frac{2}{3}})$ such that
\begin{equation}\label{eqEcal}
\mathcal{E}\equiv-\Delta
v_{ap}+\varepsilon^{-\frac{2}{3}}v_{ap}\left(v_{ap}^2-a(\varepsilon^\frac{2}{3}y)
\right)=\mathcal{O}(|\ln
\varepsilon|^{-\gamma}\varepsilon^\frac{1}{3}),\ \
\textrm{uniformly\ in}\ \mathbb{R}^2, \  \textrm{as}\ \varepsilon\to
0,
\end{equation}
for some constant $\gamma>0$, and (\ref{equappotential}),
(\ref{eqrem}) remain true, with $v_{ap}$ in place of $u_{ap}$, for
sufficiently small $\varepsilon$ (the logarithmic term in
(\ref{eqEcal}) is used for convenience purposes only and has nothing
to do with that appearing in (\ref{eqlambdaepsilon})). This
 was the main strategy  followed in \cite{sourdis-fife} for a
related one-dimensional problem. Actually, one can plainly define an
approximate solution for (\ref{eqEqstretched}) as
\begin{equation}\label{eqvap}
v_{ap}= \left\{
\begin{array}{ll}
 {u}_{out} & \textrm{in}\  \tilde{\mathcal{D}}_\varepsilon\backslash\{-2M_\varepsilon<x<0\}, \\
    &   \\
  u_{in}+\rho_{M_\varepsilon}(x)({u}_{out}-u_{in}) &\textrm{in}\  \{-2M_\varepsilon\leq x \leq \delta \varepsilon^{-\frac{2}{3}}\}, \\
    &   \\
   \chi_\delta (x)u_{in} & \textrm{everywhere\ else},
\end{array}
\right.
\end{equation}
where $M_\varepsilon$ is such that $L\leq M_\varepsilon\leq
\frac{\delta}{10}\varepsilon^{-\frac{2}{3}}$, and $u_{in}$,
$u_{out}$, $\rho_M$ as in (\ref{equin}), (\ref{equout}),
(\ref{eqcutoffrho}) respectively. Working as in Proposition
\ref{prououttilda}, we can verify that
\[
  \left|u_{out}-u_{in}\right|+|s| \left|\nabla(u_{out}-u_{in})\right|+s^2\left|\Delta(u_{out}-u_{in})\right|\leq
C\left(\varepsilon|s|^\frac{3}{2}+\varepsilon^\frac{1}{3}|s|^{-\frac{5}{2}}\right),
\]
in the region described by $\{-2\delta\varepsilon^{-\frac{2}{3}}\leq
\beta s \leq -L\}$. Then, as in Proposition \ref{prouap}, we can
show that $\mathcal{E}$, defined in (\ref{eqEcal}), satisfies
\[|\mathcal{E}|\leq C\left\{
\begin{array}{ll}
\varepsilon^\frac{4}{3}, &  \textrm{in}\
\tilde{\mathcal{D}}_\varepsilon\backslash
\{-\delta\varepsilon^{-\frac{2}{3}}<\beta
  s<0\},\\
  & \\
\varepsilon^\frac{1}{3}|s|^{-\frac{3}{2}},  &
-\delta\varepsilon^{-\frac{2}{3}}\leq
\beta s \leq -2M_\varepsilon, \\
    &   \\
  \varepsilon|s|^\frac{5}{2}+\varepsilon^\frac{1}{3}|s|^{-\frac{3}{2}}, &-2M_\varepsilon \leq \beta s \leq -M_\varepsilon, \\
    &   \\
M_\varepsilon ^\frac{5}{2}\varepsilon+M_\varepsilon^\frac{1}{2}\varepsilon^\frac{5}{3}, &-M_\varepsilon \leq \beta s \leq 0,\\
& \\
\varepsilon,& \textrm{everywhere\ else}.
\end{array}
\right.
\]
Consequently, we can achieve bound (\ref{eqEcal}) by plainly
choosing $M_\varepsilon=|\ln\varepsilon|$. Furthermore, the
approximation $v_{ap}$ is sufficiently close to $u_{ap}$ so that the
estimates (\ref{equappotential}) and (\ref{eqrem}) remain true with
$v_{ap}$ in place of $u_{ap}$. However, the corresponding estimates
for the solution of (\ref{eqEqstretched}), obtained using this
approximation, are far from optimal. One can actually check that the
above argument works because the exponent $5/2$ in
(\ref{eqvasympto}) is strictly larger than one. It is worthwhile to
mention that the geometric singular perturbation approach in
\cite{schecter-sourdis} required merely (\ref{eqvboundary}). In
\cite{weipitaevskii}, \cite{weivortex-2}, for a closely related
problem to $(\ref{eqground})_-$, the authors  made the choice
$M_\varepsilon= \frac{\delta}{10}\varepsilon^{-\frac{2}{3}}$
(according to our notation) while at the same time not using any
convergence rate of $V(x)$ to $\sqrt{-x}$ as $x\to -\infty$,
something which is not yet clear to us.
\end{rem}
\subsection{Proof of the main theorem}\label{secproofofmain}
We are now ready for the

\texttt{PROOF OF THEOREM \ref{thmmain}}: It follows from the
definition of $a_\varepsilon$ from (\ref{eqaepsilon}) that
\begin{equation}\label{eqetaepsilon}
\textbf{u}_\varepsilon(\textbf{y})\equiv
u_\varepsilon\left(\frac{\textbf{y}}{\varepsilon^\frac{2}{3}}
\right),\ \ \textbf{y}\in \mathbb{R}^2,
\end{equation}
where $u_\varepsilon$ is the solution of (\ref{eqEqstretched})
 as in Proposition \ref{thmexistence1}, is also a
solution of problem (\ref{eqlagrange}) besides the minimizer
$\eta_\varepsilon$ of $G_\varepsilon$ in $\mathcal{H}$. On the other
hand, we know from Theorem 2.1 in \cite{ignat} that (given
$\lambda_\varepsilon$) problem (\ref{eqlagrange}) has a unique
solution
(see also Remark \ref{remuniq} below). Therefore, we conclude that
$\textbf{u}_\varepsilon\equiv\eta_\varepsilon$. Estimates
(\ref{eqestim1})--(\ref{eqestim2+}) for $\eta_\varepsilon$ follow
readily from the corresponding estimates
(\ref{eqcorestim1})--(\ref{eqcorexpdec1}) for $u_\varepsilon$.
Relation (\ref{equappotentialthmmain}) follows easily from
(\ref{equappotential}) and (\ref{equepsilon-X}).

Next, we will derive estimate (\ref{eqlambdaepsilonnew}) by building
on estimates (\ref{eqlambdaepsilon}),
(\ref{eqcorestim1})--(\ref{eqcorexpdec1}), and using that
$\|\eta_\varepsilon\|_{L^2(\mathbb{R}^2)}=1$. (For a self-contained
proof of (\ref{eqlagrange}) we refer to Remark \ref{remlagrange}
below). We consider the following annular regions of the plane:
 \[
\begin{array}{l}
  S_-=\{-\delta\varepsilon^{-\frac{2}{3}}\leq \beta s\leq -2L,\ 0\leq
z\leq \varepsilon^{-\frac{2}{3}}\ell_\varepsilon\}, \\
  \\
  S_0=\{-2L\leq \beta s\leq 0,\ 0\leq z\leq
\varepsilon^{-\frac{2}{3}}\ell_\varepsilon\}, \\
    \\
 S_+=\{0\leq \beta s\leq \delta\varepsilon^{-\frac{2}{3}},\ 0\leq
z\leq \varepsilon^{-\frac{2}{3}}\ell_\varepsilon\}.
\end{array}
 \]
It follows from (\ref{equap}), and (\ref{equepsilon-g}), that the
solution of the stretched problem (\ref{eqEqstretched}) satisfies
\[
u_\varepsilon=\tilde{u}_{out}+\mathcal{O}(\varepsilon
|s|^{-\frac{3}{2}}),\ \textrm{uniformly\ in}\ S_-,\ \textrm{as}\
\varepsilon\to 0.
\]
Furthermore, from (\ref{eqa>b}), (\ref{eqvestims}), and
(\ref{equouttilda}), if $\varepsilon$ is small, we have
\[
c\varepsilon^\frac{1}{3}|s|^\frac{1}{2}\leq \tilde{u}_{out}\leq
C\varepsilon^\frac{1}{3}|s|^\frac{1}{2}\ \ \textrm{in}\ S_-.
\]
So, from (\ref{eqA}), (\ref{eqaepsilon}), (\ref{equouttilda}), and
the above two relations, we find that
\[
u_\varepsilon^2-A(\varepsilon^\frac{2}{3}y)=\lambda_\varepsilon-\lambda_0+\varepsilon^\frac{2}{3}\beta^2[\beta
s+V^2(\beta s)]+\mathcal{O}(\varepsilon^\frac{4}{3}|s|^{-1}),
\]
uniformly in $S_-$, as $\varepsilon\to 0$. Thus, via
 the identity
\[
\int_{S_-}^{}f(y)dy=\int_{0}^{\varepsilon^{-\frac{2}{3}}\ell_\varepsilon}
\int_{-\delta\varepsilon^{-\frac{2}{3}}\beta^{-1}}^{-2L\beta^{-1}}f(\varepsilon^\frac{2}{3}s,\varepsilon^\frac{2}{3}z)\left(1+\varepsilon^\frac{2}{3}
k_\varepsilon(\varepsilon^\frac{2}{3}z)s \right)dsdz\ \ \ \forall
f\in C(\bar{S}_-),
\]
we obtain that
\begin{equation}\label{eqpro1}
\begin{array}{lll}
  \int_{S_-}^{}\left(
u_\varepsilon^2-A(\varepsilon^\frac{2}{3}y)\right)dy&= &
(\lambda_\varepsilon-\lambda_0)\int_{S_-}^{}1dy+
\int_{0}^{\ell_\varepsilon}\beta_\varepsilon(\theta) d\theta
\int_{-\delta\varepsilon^{-\frac{2}{3}}}^{-2L}[V^2(x)+x]dx\\
& &+\int_{S_-}^{}\mathcal{O}(\varepsilon^\frac{4}{3}|s|^{-1})dsdz\\
  & &   \\
  &\stackrel{(\ref{eqvestims})}{=} & (\lambda_\varepsilon-\lambda_0)\int_{S_-}^{}1dy+
\int_{0}^{\ell_\varepsilon}\beta_\varepsilon(\theta) d\theta
\int_{-\infty}^{-2L}[V^2(x)+x]dx+\mathcal{O}(|\ln
\varepsilon|\varepsilon^\frac{2}{3}),
\end{array}
\end{equation}
as $\varepsilon\to 0$. Similarly, recalling (\ref{eqbita}),
(\ref{eqvasympto}), (\ref{eqcorestim1}), and (\ref{eqcorestim1+}),
we have
\begin{equation}\label{eqpro2}
\int_{S_0}^{}\left(
u_\varepsilon^2-A(\varepsilon^\frac{2}{3}y)\right)dy=
(\lambda_\varepsilon-\lambda_0)\int_{S_0}^{}1dy+
\int_{0}^{\ell_\varepsilon}\beta_\varepsilon(\theta) d\theta
\int_{-2L}^{0}[V^2(x)+x]dx+\mathcal{O}(\varepsilon^\frac{2}{3}),
\end{equation}
and
\begin{equation}\label{eqpro3}
\int_{S_+}^{} u_\varepsilon^2dy=
\int_{0}^{\ell_\varepsilon}\beta_\varepsilon(\theta) d\theta
\int_{0}^{\infty}V^2(x)dx+\mathcal{O}(\varepsilon^\frac{2}{3}),
\end{equation}
 as $\varepsilon\to 0$. Moreover, thanks to (\ref{eq372+}),
 (\ref{eqcorexpdec1}), we have
\begin{equation}\label{eqpro4}
\int_{\tilde{\mathcal{D}}_\varepsilon\backslash(S_0\cup
S_-)}^{}\left( u_\varepsilon^2-A(\varepsilon^\frac{2}{3}y)\right)dy=
(\lambda_\varepsilon-\lambda_0)\int_{\tilde{\mathcal{D}}_\varepsilon\backslash(S_0\cup
S_-)}^{}1dy+\mathcal{O}(\varepsilon^\frac{2}{3}),
\end{equation}
and
\begin{equation}\label{eqpro5}
\int_{\mathbb{R}^2\setminus(\tilde{\mathcal{D}}_\varepsilon\cup
S_+)}^{} u_\varepsilon^2dy=
\mathcal{O}(e^{-c\varepsilon^{-\frac{2}{3}}}),
\end{equation}
as $\varepsilon \to 0$. Let us keep in mind that
\begin{equation}\label{eqpro6}
\int_{\mathbb{R}^2}^{}u_\varepsilon^2(y)dy=
\varepsilon^{-\frac{4}{3}}\int_{\mathbb{R}^2}^{}\eta_\varepsilon^2(\textbf{y})d\textbf{y}=\varepsilon^{-\frac{4}{3}}.
\end{equation}
Furthermore, recalling  (\ref{eqD0}), (\ref{eqVnormal}),
(\ref{eqA}),
 and (\ref{eqlambdaepsilon}), for small
$\varepsilon$, we can write
\[
\int_{\mathcal{D}_\varepsilon}^{}A^+(\textbf{y})d\textbf{y}=\int_{\mathcal{D}_0}^{}A^+(\textbf{y})d\textbf{y}
-\int_{\mathcal{U}_\varepsilon}^{}A^+(\textbf{y})d\textbf{y}
 +\int_{\mathcal{V}_\varepsilon}^{}A^+(\textbf{y})d\textbf{y},
\]
where 
$\mathcal{U}_\varepsilon\subseteq \mathcal{D}_0$, $
\mathcal{V}_\varepsilon \cap \mathcal{D}_0=\emptyset$,
$|\mathcal{U}_\varepsilon|+|\mathcal{V}_\varepsilon|\leq C|\ln
\varepsilon|^\frac{1}{2}\varepsilon$, and $A^+(\textbf{y})\leq C|\ln
\varepsilon|^\frac{1}{2}\varepsilon$ if
$\textbf{y}\in\mathcal{U}_\varepsilon\cup\mathcal{V}_\varepsilon$.
Thus, via (\ref{eqlambda0}), (\ref{eqDepsilon}), we infer that
\begin{equation}\label{eqpro7}
\int_{\tilde{\mathcal{D}}_\varepsilon}^{}A^+(\varepsilon^\frac{2}{3}y)dy=\varepsilon^{-\frac{4}{3}}+\mathcal{O}(|\ln\varepsilon|\varepsilon^\frac{2}{3})
\ \ \textrm{as}\ \varepsilon\to 0.
\end{equation}
By combining (\ref{eqpro1})--(\ref{eqpro7}), we deduce that
\[
(\lambda_\varepsilon-\lambda_0)|\tilde{\mathcal{D}}_\varepsilon|+
\left(\int_{0}^{\ell_\varepsilon}\beta_\varepsilon(\theta)d\theta
\right)\left(\int_{-\infty}^{0}[V^2(x)+x]dx+\int_{0}^{\infty}V^2(x)dx
\right)=\mathcal{O}(|\ln \varepsilon|\varepsilon^\frac{2}{3})
\]
as $\varepsilon\to 0$. Now, the validity of estimate
(\ref{eqlambdaepsilonnew}) follows readily by noting that the sum of
the above two integral involving $V$ is zero. This can be seen by
multiplying (\ref{eqpainleve}) by $V_x$, integrating the resulting
identity by parts over $(-\infty,0)$ and $(0,\infty)$ respectively,
and recalling (\ref{eqvasympto}).

To finish, utilizing all the above, we will establish the validity
of estimate (\ref{eqenergysharp}) for the energy of
$\eta_\varepsilon$. It is straightforward to see that
\[
G_\varepsilon(\eta_\varepsilon)=\int_{\mathbb{R}^2}^{}\left\{
\frac{1}{2}|\nabla
u|^2+\frac{\varepsilon^{-\frac{2}{3}}}{4}u^4+\frac{\varepsilon^{-\frac{2}{3}}}{2}W(\varepsilon^\frac{2}{3}y)u^2
\right\}dy,
\]
where $u(y)=\eta_\varepsilon(\varepsilon^\frac{2}{3}y)$ is the
solutions of the stretched problem (\ref{eqEqstretched}). Motivated
from (\ref{eqGepsrenorintro}), and recalling (\ref{eqlambda0}), it
is easy to check that we can rewrite the above relation as
\begin{equation}\label{eqGepsexpanded}
\begin{array}{lll}
  G_\varepsilon(\eta_\varepsilon) & = & \frac{1}{2}\int_{\mathbb{R}^2}^{} |\nabla
u|^2dy+\frac{\varepsilon^{-\frac{2}{3}}}{4}\int_{\tilde{\mathcal{D}}_\varepsilon}^{}\left(u^2-a(\varepsilon^{\frac{2}{3}}y)\right)^2dy
-\frac{1}{4}\left(\int_{\mathbb{R}^2}^{}(A^+)^2d\textbf{y}\right)\varepsilon^{-2}
+\frac{\lambda_0}{2}\varepsilon^{-2} \\
   &    &   \\
    &   & -\frac{(\lambda_\varepsilon-\lambda_0)^2}{4}|\mathcal{D}_0|\varepsilon^{-2}
+\frac{\varepsilon^{-\frac{2}{3}}}{4}\int_{\mathbb{R}^2\backslash
\tilde{\mathcal{D}}_\varepsilon}^{}\left(u^2-a(\varepsilon^{\frac{2}{3}}y)\right)^2dy
    -\frac{\varepsilon^{-\frac{2}{3}}}{4}\int_{\mathbb{R}^2\backslash(\varepsilon^{-\frac{2}{3}}
\mathcal{D}_0)}^{}a^2(\varepsilon^\frac{2}{3}y)dy.
\end{array}
\end{equation}
Similarly to the above proof of (\ref{eqlambdaepsilonnew}), keeping
in mind the proof of Proposition \ref{thmexistence1},
(\ref{eqlambdaepsilonnew}), (\ref{eqvasympto}), (\ref{equouts1}),
(\ref{eqsz}), (\ref{eqcorestim1}), and relations (\ref{eqnambla}),
(\ref{eqcorgrad2}) below (whose proofs do not require
(\ref{eqenergysharp})), we get
 \[\begin{array}{lll}
     \int_{\mathbb{R}^2}^{}|\nabla
u|^2dy & = &
\int_{S_-}^{}\left(\varepsilon^\frac{2}{3}\beta^4V_x^2(\beta
s)+\mathcal{O}(\varepsilon^\frac{4}{3})\right)(1+\varepsilon^\frac{2}{3}ks)dsdz+\mathcal{O}(1) \\
      &  &  \\
       &  = & \left(\int_{0}^{\ell_\varepsilon}
\beta^3_\varepsilon(\theta)d\theta\right)\int_{-\delta\varepsilon^{-\frac{2}{3}}}^{-2L}V_x^2(x)dx+\mathcal{O}(1)
 \\
       &   &   \\
      &  =& \frac{1}{6}\left(\int_{0}^{\ell_0}
\beta^3_0(\theta)d\theta\right)|\ln \varepsilon|+\mathcal{O}(1),
   \end{array}
\]
as $\varepsilon\to 0$. Furthermore, as in the above proof of
(\ref{eqlambdaepsilonnew}), we have
\[
\frac{\varepsilon^{-\frac{2}{3}}}{4}\int_{\tilde{\mathcal{D}}_\varepsilon}^{}\left(u^2-a(\varepsilon^{\frac{2}{3}}y)\right)^2dy=\mathcal{O}(1),
\]
and
\[
\frac{\varepsilon^{-\frac{2}{3}}}{4}\int_{\mathbb{R}^2\backslash
\tilde{\mathcal{D}}_\varepsilon}^{}\left(u^2-a(\varepsilon^{\frac{2}{3}}y)\right)^2dy
    -\frac{\varepsilon^{-\frac{2}{3}}}{4}\int_{\mathbb{R}^2\backslash(\varepsilon^{-\frac{2}{3}}
\mathcal{D}_0)}^{}a^2(\varepsilon^\frac{2}{3}y)dy=\mathcal{O}(\varepsilon^\frac{2}{3}),
\]
as $\varepsilon\to 0$. Now, the validity of (\ref{eqenergysharp})
follows at once from (\ref{eqlambdaepsilonnew}),
(\ref{eqGepsexpanded}), and the above three relations.
Alternatively, we could have used the formula
\[
G_\varepsilon(\eta_\varepsilon)=\frac{1}{2\varepsilon^2}\lambda_\varepsilon
-\frac{1}{4\varepsilon^2}\int_{\mathbb{R}^2}^{}\eta_\varepsilon^4d\textbf{y},
\]
which follows easily by testing equation (\ref{eqlagrange}) with
$\eta_\varepsilon$.

 The proof of Theorem \ref{thmmain} is complete.\
\ \ $\Box$

We now outline a few remarks.
\begin{rem}\label{remarkupperbound}
The maximum principle yields the upper bound:
\[
\eta_\varepsilon(\textbf{y})\leq
\max_{\mathbb{R}^2}\sqrt{a_\varepsilon^+},\ \ \textbf{y}\in
\mathbb{R}^2,
\]
(see also (\ref{eqyan}) below).
\end{rem}
\begin{rem}\label{remyan}
As in \cite{ignat}, where the authors refer to an idea of Shafrir,
we can rewrite (\ref{eqlagrange}) in the form
\[
-\varepsilon^2\Delta
(\sqrt{a_\varepsilon}-\eta_\varepsilon)+\eta_\varepsilon
(\eta_\varepsilon+\sqrt{a_\varepsilon})(\eta_\varepsilon-\sqrt{a_\varepsilon})=-\varepsilon^2\Delta(\sqrt{a_\varepsilon})\
\ \textrm{in}\ \mathcal{D}_\varepsilon.
\]
The above relation suggests the following, which can be proven
similar to \cite{yanedinburg}: We have
\begin{equation}\label{eqyan}
\eta_\varepsilon(\textbf{y})=\sqrt{a_\varepsilon}+\varepsilon^2\frac{\Delta(\sqrt{a_\varepsilon})}{2a_\varepsilon}+o(\varepsilon^2),
\end{equation}
where $\varepsilon^{-2}o(\varepsilon^2)\to 0$ uniformly on any
compact subset of $\mathcal{D}_0$ as $\varepsilon\to 0$. Keeping in
mind (\ref{eqVnormal}) which implies that
\[
\varepsilon^2\frac{\Delta(\sqrt{a_\varepsilon})}{2a_\varepsilon}=\mathcal{O}\left(\varepsilon^2|t|^{-\frac{5}{2}}\right)\
\textrm{uniformly \ in}\ \mathcal{D}_0\ \textrm{as}\ \varepsilon\to
0,
\]
and (\ref{eqestim2}), we are tempted to believe that (\ref{eqyan})
can be extended to hold  uniformly in the domain
$\mathcal{D}_\varepsilon\backslash
\{-K\varepsilon^\frac{2}{3}<t<0\}$, with $K$ large, if
$\varepsilon\to 0$. A possible approach could be by seeking a more
refined inner solution with $V+\varepsilon\phi$ in place of $V$ in
(\ref{equin}), where $\phi$ is determined by solving a linear
equation of the form $\mathcal{M}(\phi)=f(x,z),\ x\in \mathbb{R}, \
z \in [0,\varepsilon^{-\frac{2}{3}}\ell_\varepsilon)$ with $f$ known
(in terms of the curvature $k_\varepsilon$, $V$, $a_\varepsilon$ and
their derivatives) and $\mathcal{M}$ as in (\ref{eqMcal}); we refer
to \cite{karalisourdisresonance} for a related problem. Actually, we
have computed that in the radially symmetric case, in $N\geq 1$
dimensions, we have
\[
\mathcal{M}(\phi)=\frac{N-1}{R_\varepsilon}
\beta^{-1}V_x-\frac{1}{2}a_{rr}(R_\varepsilon) \beta^{-4}x^2V,
\]
where $R_\varepsilon$ is the radius of $\mathcal{D}_\varepsilon$ and
$\beta=[-a_r(R_\varepsilon)]^\frac{1}{3}$. Moreover, due to matching
conditions with
$a_\varepsilon(R_\varepsilon+\varepsilon^\frac{2}{3}\beta^{-1}x)$,
we need that
\[
\phi(x)+
\frac{1}{4}a_{rr}(R_\varepsilon)\beta^{-4}(-x)^\frac{3}{2}\to 0,\
x\to -\infty;\ \ \phi \to 0,\ x\to \infty.
\]
In the special case of the model harmonic potential,  this lower
order term in the inner solution has been formally derived in
\cite{fetterPhys} and rigorously in \cite{pelinovsky}. It might also
be useful for the reader to take a look at (\ref{eqpsiupperWtt})
below.
\end{rem}
\begin{rem}\label{remC1}
As in \cite[Prop. 2.1 e)]{ignat} , it follows that
\[
\|\eta_\varepsilon-\sqrt{a_\varepsilon}\|_{C^1(\mathcal{K})}\leq
C_\mathcal{K}\varepsilon^2\ \ \textrm{for\ any\ compact \ subset}\
\mathcal{K}\subset \mathcal{D}_0,
\]
if $\varepsilon$ is small.
\end{rem}

\begin{rem}\label{remuniq}
 Given $\lambda>\min_{\mathbb{R}^N}W$, problem $(\ref{eqground})_-$, in $N=2$ dimensions, has a
unique positive solution, for small $\varepsilon$, as has been
proven recently in \cite{ignat} by combining ideas of Brezis and
Oswald \cite{brezis-oswald} (see also the uniqueness part of
Proposition \ref{prophalf} herein) with those used in the proof of
De Giorgi's conjecture in low dimensions \cite{gui}. This fact
allows us to work exclusively with equation (\ref{eqlagrange}),
since positive solutions of the latter coincide with the unique real
valued minimizer of $G_\varepsilon$ in $\mathcal{H}$. On the other
hand, using (\ref{eqJ1}), which holds for every positive solution of
$(\ref{eqground})_-$, it is easy to see that the method of
\cite{ignat} can be extended to the case of arbitrary $N \geq 1$
dimensions. Hence, our Theorem \ref{thmmain} can be extended
naturally to treat the case where the functional $G_\varepsilon$ is
considered in arbitrary dimensions, with the analogous conditions on
the potential. In the radially symmetric case, uniqueness results
which allow the case where $\lambda=\inf_{\mathbb{R}^N} W$ may be
found in \cite{brownJmaa} and \cite{pelinovsky}.
\end{rem}
\begin{rem}\label{remdirichlet}
One can also prove an analogous result to Theorem \ref{thmmain} for
the real valued minimizer of $G_\varepsilon$ in
\[
\mathcal{J}\equiv\left\{u \in W_0^{1,2}(\mathcal{D}_0; \mathbb{C})\
:\ \int_{\mathcal{D}_0}^{}|u|^2 dy=1\right\}.
\]
This problem has been studied in \cite{jerard}, in a
three-dimensional setting, with potentials of the form
(\ref{eqharmonic}). The special property that $\Delta W>0$ in
$\mathcal{D}_0$ was used in an essential way in the latter reference
for estimating the minimizer near the surface
$\partial\mathcal{D}_0$, along which it has a steep corner layer. We
also refer the interested reader to \cite{alama} and
\cite{alamamontero} for the case where $\mathcal{D}_0$ has annular
shape. For a numerical treatment of the problem we refer to
\cite{numerics}. As in Subsection \ref{secnearcurve} below, it is
not hard to see that in this case the layer profile near $\partial
\mathcal{D}_0$ should be determined by the unique solution of
\begin{equation}\label{eqvdirichlet}
\left\{ \begin{array}{ll}
          v_{xx}-v(v^2+x)=0, & x<0, \\
           &  \\
          v(x)-\sqrt{-x}\to 0\  \textrm{as}\ x\to -\infty; &
          v(0)=0.
        \end{array}
 \right.
\end{equation}
We refer to Appendix \ref{appenpainleve} below for a treatment of
the above problem in relation with (\ref{eqpanintro}).

The minimization of the functional $G_\varepsilon$ in
$W^{1,2}(\mathcal{D}_0)$, subject to the mass constraint, leads to
the equation in $(\ref{eqground})_-$ with Neumann boundary
conditions. The latter singular perturbation problem may be treated
by using in place of $V$, in (\ref{equin}), the (reflection of the)
solution described in Remark \ref{rempainleveNeumann} below.
\end{rem}
\begin{rem}\label{remlagrange}
Identical estimates to (\ref{eqestim1})-(\ref{eqestim2+}) (with $t$
replaced by the signed distance from $\partial \mathcal{D}_0$) hold
for the solution $\tilde{\eta}_\varepsilon$ of $(\ref{eqground})_-$
(with $\lambda=\lambda_0$, $q=3$). One can use the function
$\frac{\tilde{\eta}_\varepsilon}{\|\tilde{\eta}_\varepsilon\|_{L^2(\mathbb{R}^2)}}\in
\mathcal{H}$ as a competitor in order to give a self-contained proof
of (\ref{eqG1intro}), from which (\ref{eqlambdaepsilon}) follows
readily (see \cite{aftalion-jerrard}, \cite{ignat}). (As in the
proof of (\ref{eqenergysharp}), keeping in mind that
$\|\tilde{\eta}_\varepsilon\|^2_{L^2(\mathbb{R}^2)}\to
\|A^+\|_{L^1(\mathbb{R}^2)}$, the main contribution would be from
the gradient term).
\end{rem}
\begin{rem}\label{remignat}
In the  case where the potential $W$ is of harmonic type, as in
(\ref{eqharmonic}), it was observed by the authors of \cite{ignat}
that
\[
\eta_\varepsilon(\textbf{y})=\frac{\sqrt{\lambda_0+\lambda_\varepsilon}}{\sqrt{\lambda_0}}\tilde{\eta}_{\tilde{\varepsilon}}
\left(
\frac{\sqrt{\lambda_0}\textbf{y}}{\sqrt{\lambda_0+\lambda_\varepsilon}}
\right)\ \ \textrm{with}\ \ \tilde{\varepsilon}=\frac{\lambda_0
\varepsilon}{\lambda_0+\lambda_\varepsilon},
\]
where $\tilde{\eta}_{\tilde{\varepsilon}}$ was defined in Remark
\ref{remlagrange}. This identity and a technique of Struwe
\cite[Lemma 2.3]{ignat} were used essentially in their proof of
(\ref{eqlambdaepsilonnew}) for this special class of potentials.
\end{rem}
\begin{rem}\label{remsmoothness}
Relation (\ref{equappotentialthmmain}) implies that, if
$\varepsilon$ is small, the linearized operator
$\mathcal{L}_\varepsilon$ about $u_\varepsilon$ is invertible.
Hence, the implicit function theorem (see for instance
\cite{ambrozetti-proddi}) implies that there exists a small number
$\varepsilon_0>0$ such that, in addition to being isolated (for each
$\varepsilon$), the minimizers
$\eta(\varepsilon)\equiv\eta_\varepsilon$ depend smoothly on
$\varepsilon \in (0,\varepsilon_0)$ (in all the usual function
spaces). In particular, we have $\eta:(0,\varepsilon_0)\to
W^{1,2}(\mathbb{R}^2)$ is $C^1$. This last property yields at once
the first part of Lemma 2.3 in \cite{ignat}, mentioned in Remark
\ref{remignat} above.
\end{rem}
\begin{rem}\label{rempeli}
In the special case of the model harmonic potential
$W(\textbf{y})=|\textbf{y}|^2$, since $\eta_\varepsilon$ is radially
symmetric \cite{gidas}, we can define $\nu_\varepsilon:
(-\infty,\lambda_\varepsilon\varepsilon^{-\frac{2}{3}}]\to
\mathbb{R}$ by
\[
\eta_\varepsilon(\textbf{y})=\varepsilon^\frac{1}{3}\nu_\varepsilon\left(\frac{\lambda_\varepsilon-|\textbf{y}|^2}{\varepsilon^\frac{2}{3}}
\right),\ \ \textbf{y}\in \mathbb{R}^2.
\]
Letting $\xi=\left(\lambda_\varepsilon-|\textbf{y}|^2
\right)/\varepsilon^\frac{2}{3}$, then the equation in
(\ref{eqlagrange}) becomes equivalent to
\[
4(\lambda_\varepsilon-\varepsilon^\frac{2}{3}\xi)\partial_{\xi\xi}\nu_\varepsilon-4\varepsilon^\frac{2}{3}\partial_{\xi}\nu_\varepsilon+\xi
\nu_\varepsilon-\nu_\varepsilon^3=0,\ \ \xi\in
(-\infty,\lambda_\varepsilon\varepsilon^{-\frac{2}{3}}].
\]
At first glance this might look rather counterintuitive but this
strategy, already used in \cite{pelinovsky}, allows one to use
directly
$(2\lambda_\varepsilon)^\frac{1}{3}V\left(-\frac{\xi}{(2\lambda_\varepsilon)^\frac{2}{3}}
\right)$ as a global approximation (recall (\ref{eqpainleve})--
(\ref{eqvestims})).
\end{rem}
\begin{rem}\label{remdalfovo}
In the radially symmetric case, estimate (\ref{eqenergysharp}) was
formally predicted in \cite{dalfovofirst,dalfovokinetic}, and
rigorously proven and extended very recently, for the case of the
harmonic potential, in \cite{gallonew} at the same time that the
current paper was written.
\end{rem}

\section{Further properties of the ground state
$\eta_\varepsilon$}\label{secfurther} As a byproduct of our
construction of the ground state $\eta_\varepsilon$, we can extend
to the non-radial case relation (\ref{eqJ4}), improve relation
(\ref{eqJ3}), show that $\eta_\varepsilon$ has maximal H\"{o}lder
regularity, and improve relation (\ref{eqJ2}). Finally, under an
additional but natural non-degeneracy assumption on the potential,
we can refine bound (\ref{eqestim1+}).

 \begin{cor}\label{cor1}
There exist small constants $c,\ d'>0$ such that, given $D>0$, we
have
\begin{equation}\label{eqetat<0}
(\eta_\varepsilon)_t\leq
-c(|t|+\varepsilon^\frac{2}{3})^{-\frac{1}{2}}
\end{equation}
in $\left\{-d'\leq t\leq D\varepsilon^\frac{2}{3},\ \ \theta\in
[0,\ell_\varepsilon) \right\}$, provided $\varepsilon$ is
sufficiently small.
 \end{cor}
\begin{proof} We will present it for the
 solution $u_\varepsilon$ of the stretched problem (\ref{eqEqstretched}). It follows
 from
(\ref{eqLphi=E+N})
  that, if $\varepsilon$ is small, the
function $\varphi_*=u_\varepsilon-u_{ap}$ satisfies
\begin{equation}\label{eqDelta}
|\Delta_y \varphi_*|\leq \left\{
\begin{array}{ll}
  C\varepsilon (|s|+1)^{-\frac{1}{2}}&\textrm{if} \ \ -10\delta\varepsilon^{-\frac{2}{3}}\leq \beta s\leq 0, \\
   &  \\
C\varepsilon e^{-c|s|} &\textrm{if} \ \ 0\leq \beta s\leq
10\delta\varepsilon^{-\frac{2}{3}}.
\end{array}
\right.
\end{equation}
 In establishing the above estimate, the term that needed extra
care was
\[\varepsilon^{-\frac{2}{3}}\left(3u_{ap}^2-a(\varepsilon^\frac{2}{3}y)
\right)\varphi_*,\] which can be estimated by noting that
\begin{equation}\label{eqDelta2}
u_{ap}^2+|a(\varepsilon^\frac{2}{3}y)|\leq C
\varepsilon^\frac{2}{3}(|s|+1)\ \ \textrm{and}\ \ |\varphi_*|\leq
C\varepsilon (|s|+1)^{-\frac{1}{2}}\ \ \textrm{if}\ \{|\beta s|\leq
10\delta \varepsilon^{-\frac{2}{3}}\},
\end{equation}
(recall (\ref{eqa>b}), (\ref{eq3uap1}), (\ref{eq3uap2}), and
(\ref{equepsilon-g})). Making (mild) use of (\ref{eqDelta}) and the
second estimate in (\ref{eqDelta2}), via standard interior elliptic
regularity estimates \cite{Gilbarg-Trudinger} (applied on balls of
radius one) or the interpolation-type inequality of Lemma A.1 in
\cite{bethuelCVPDE}, the statement of which is included below as
Lemma \ref{lembrezis} for the reader's convenience, we obtain that
\begin{equation}\label{eqnambla}
|\nabla_y \varphi_*|\leq C\varepsilon
\end{equation}
in the neighborhood of $\tilde{\Gamma}_\varepsilon$ described by
$\{|\beta s|\leq 5\delta \varepsilon^{-\frac{2}{3}}\}$, provided
$\varepsilon$ is sufficiently small.

By  (\ref{eqVx<0}), (\ref{equin}), (\ref{eqsz}), (\ref{equap}), and
(\ref{eqnambla}), we deduce that, given $\tilde{D}>0$,
\begin{equation}\label{eqnambla2}
(u_\varepsilon)_s=\varepsilon^\frac{1}{3}\beta^2V_x(\beta
s)+\mathcal{O}(\varepsilon)<-c\varepsilon^\frac{1}{3},
\end{equation}
uniformly in the neighborhood of $\tilde{\Gamma}_\varepsilon$
described by $\{-L \leq x \leq \tilde{D}\}$, as $\varepsilon\to 0$.
In the same manner, recalling (\ref{eqvestims}) and the second
estimate in (\ref{equout-uingradients}), we find that
\begin{equation}\label{eqnambla3}
\begin{array}{lll}
  (u_\varepsilon)_s & = & \varepsilon^\frac{1}{3}\beta^2V_x(\beta s)+\mathcal{O}(\varepsilon
|s|^\frac{1}{2})
 \\
    &   &   \\
    &  \leq &
 -c\varepsilon^\frac{1}{3}
|s|^{-\frac{1}{2}}+C\varepsilon |s|^\frac{1}{2}\\
&&\\
 &\leq& -\frac{c}{2}\varepsilon^\frac{1}{3}|s|^{-\frac{1}{2}},
\end{array}
\end{equation}
as long as $-\tilde{d}\varepsilon^{-\frac{2}{3}}<\beta s<-L$, for
some small constant $\tilde{d}$, provided $\varepsilon$ is
sufficiently small. The corresponding assertion of the corollary,
for the solution of the equivalent stretched problem
(\ref{eqEqstretched}),  follows readily from (\ref{eqnambla2}) and
(\ref{eqnambla3}).

 The proof of the corollary
is complete.\end{proof}


The following is Lemma A.1 in \cite{bethuelCVPDE}:
\begin{lem}\label{lembrezis}
Assume that $u$ satisfies
\[
-\Delta u=f\ \ \textrm{in}\ \ \Omega \subset \mathbb{R}^N.
\]
Then
\[
\left|\nabla u(\textbf{y}) \right|^2\leq
C\left\{\|f\|_{L^\infty(\Omega)}\|u\|_{L^\infty(\Omega)}+\frac{1}{\textrm{dist}^2(\textbf{y},\partial
\Omega)} \|u\|^2_{L^\infty(\Omega)} \right\}\ \ \forall \
\textbf{y}\in \Omega,
\]
where $C$ is some constant depending only on $N$.
\end{lem}

We can also improve the bound (\ref{eqJ3}):
\begin{cor}\label{corgrad}
If $\varepsilon>0$ is sufficiently small, we have
\[
\|\nabla\eta_\varepsilon\|_{L^\infty(\mathbb{R}^2)}\leq
C\varepsilon^{-\frac{1}{3}}.
\]
\end{cor}
\begin{proof}
We will prove the corresponding assertion for the solution
$u_\varepsilon$ of the stretched problem (\ref{eqEqstretched}). From
the proof of Proposition \ref{thmexistence1}, recalling
(\ref{eqvestims}), (\ref{equouts1}), (\ref{equouttildainpro2}),
(\ref{eqsz}), and (\ref{eqnambla}), we find that
\begin{equation}\label{eqcorgrad1}
|\nabla u(y)|\leq C\varepsilon^\frac{1}{3}\ \ \textrm{if}\
y\in\{|x|\leq 3\delta \varepsilon^{-\frac{2}{3}},\ z\in
[0,\varepsilon^{-\frac{2}{3}}\ell_\varepsilon)\},
\end{equation}
if $\varepsilon$ is small. Furthermore, recalling
(\ref{equepsilon-g}), we get
\[
  u^2-a(\varepsilon^\frac{2}{3}y)  =
 2\varphi_*\sqrt{a(\varepsilon^\frac{2}{3}y)}+\varphi_*^2= \mathcal{O}(\varepsilon^2),
\]
uniformly in $\tilde{\mathcal{D}}_\varepsilon \backslash \{-2\delta
\varepsilon^{-\frac{2}{3}}\leq x\leq 0\}$, as $\varepsilon\to 0$.
Hence, by the equation (\ref{eqEqstretched}), we obtain that
\[
|\Delta u|\leq C\varepsilon^\frac{4}{3}\ \ \textrm{in}\
\tilde{\mathcal{D}}_\varepsilon \backslash \{-2\delta
\varepsilon^{-\frac{2}{3}}\leq x\leq 0\}.
\]
Consequently, by  the interpolation-type inequality of Lemma A.1 in
\cite{bethuelCVPDE} (see Lemma \ref{lembrezis} above), we infer that
\begin{equation}\label{eqcorgrad2}
|\nabla u|\leq C\varepsilon^\frac{2}{3}\ \ \textrm{in}\
\tilde{\mathcal{D}}_\varepsilon \backslash \{-3\delta
\varepsilon^{-\frac{2}{3}}\leq x\leq 0\}.
\end{equation}
Similarly, recalling (\ref{eqainfinity}) and (\ref{eqestim3}), we
have
\begin{equation}\label{eqcorgrad3}
|\nabla u|\leq Ce^{-c\varepsilon^{-\frac{2}{3}}}\ \ \textrm{outside\
of}\ \tilde{\mathcal{D}}_\varepsilon \cup \{0\leq x\leq 3\delta
\varepsilon^{-\frac{2}{3}}\}.
\end{equation}
 The corresponding assertion of
the corollary for $u$ follows readily from (\ref{eqcorgrad1})--
(\ref{eqcorgrad3}).

The proof of the corollary is complete.
\end{proof}

In the following corollary, we will show that $\eta_\varepsilon$ has
the maximal H\"{o}lder regularity available (recall Definition
\ref{defmaximal} from Subsection \ref{secmotivation}).
\begin{cor}\label{cormax}
If $\varepsilon$ is sufficiently small, we have
\begin{equation}\label{eqmax}
\|\eta_\varepsilon\|_{C^{{1}/{2}}(\mathbb{R}^2)}\leq C.
\end{equation}
\end{cor}
\begin{proof}
From (\ref{eqestim1}), (\ref{eqestim1+}), (\ref{eqvasympto}),
(\ref{eqvestims}), (\ref{eqsz}), and (\ref{eqnambla}), abusing
notation, it follows that
\begin{equation}\label{eqmaximal1}
\left\{\begin{array}{c}
  c\left(|t|+\varepsilon^\frac{2}{3}\right)^\frac{1}{2}\leq\eta_\varepsilon\leq C\left(|t|+\varepsilon^\frac{2}{3}\right)^\frac{1}{2}, \\
    \\
    \left|(\eta_\varepsilon)_t \right|\leq C\left(|t|+\varepsilon^{\frac{2}{3}} \right)^{-\frac{1}{2}},\\
   \\
\left|(\eta_\varepsilon)_\theta \right|\leq
C\left(|t|+\varepsilon^{\frac{2}{3}} \right)^{\frac{1}{2}},
\end{array}
\right.
\end{equation}
in the region described by $\{|t|\leq 3\delta\}$, if $\varepsilon$
is small, having further decreased the value of $\delta$ if
necessary. Abusing notation once more, let
$\textbf{y}_i=(t_i,\theta_i)$, $i=1,2$, with $ |t_i|\leq 3\delta$
and $\theta_i\in [0,\ell_\varepsilon)$, be any two points in that
region. 
We write
\begin{equation}\label{eqnu12}
\eta_\varepsilon(\textbf{y}_1)-\eta_\varepsilon(\textbf{y}_2)=
\eta_\varepsilon(t_1,\theta_1)-\eta_\varepsilon(t_2,\theta_1)+
\eta_\varepsilon(t_2,\theta_1)-\eta_\varepsilon(t_2,\theta_2).
\end{equation}
 Now, instead of considering the difference
$\eta_\varepsilon(t_1,\theta_1)-\eta_\varepsilon(t_2,\theta_1)$, we
will first consider the difference
$\eta^2_\varepsilon(t_1,\theta_1)-\eta^2_\varepsilon(t_2,\theta_1)$.
We have
\[
\eta^2_\varepsilon(t_1,\theta_1)-\eta^2_\varepsilon(t_2,\theta_1)=(t_1-t_2)\int_{0}^{1}2\eta\eta_t\left(t_1+r(t_2-t_1),\theta_1\right)dr.
\]
So, thanks to (\ref{eqmaximal1}), we find that
\[
\left|\eta^2_\varepsilon(t_1,\theta_1)-\eta^2_\varepsilon(t_2,\theta_1)\right|\leq
C|t_1-t_2|.
\]
In turn, via the lower bound in (\ref{eqmaximal1}), the above
relation yields that
\begin{equation}\label{eqmaximal21}
\left|\eta_\varepsilon(t_1,\theta_1)-\eta_\varepsilon(t_2,\theta_1)\right|\leq
C\frac{|t_1-t_2|^\frac{1}{2}}{|t_1|^\frac{1}{2}+|t_2|^\frac{1}{2}}|t_1-t_2|^\frac{1}{2}\leq
C\left(|t_1-t_2|+|\theta_1-\theta_2|\right)^\frac{1}{2}.
\end{equation}
Similarly, we obtain that
\[
\left|\eta^2_\varepsilon(t_2,\theta_1)-\eta^2_\varepsilon(t_2,\theta_2)\right|\leq
C\left(|t_2|+\varepsilon^\frac{2}{3} \right)|\theta_1-\theta_2|,
\]
which, as before, implies that
\begin{equation}\label{eqmaximal22}
\left|\eta_\varepsilon(t_2,\theta_1)-\eta_\varepsilon(t_2,\theta_2)\right|\leq
C\left(|t_2|+\varepsilon^\frac{2}{3}
\right)^\frac{1}{2}|\theta_1-\theta_2|\leq
C\left(|t_1-t_2|+|\theta_1-\theta_2|\right)^\frac{1}{2}.
\end{equation}
Hence, by (\ref{eqnu12}), (\ref{eqmaximal21}), and
(\ref{eqmaximal22}) (also keeping in mind Remark
(\ref{remarkupperbound})), we deduce that
\[
\|\eta_\varepsilon\|_{C^{1/2}\left(|t|\leq3\delta\right)}\leq C,
\]
if $\varepsilon$ is small. That was the hard part. In the  remaining
regions of the plane, by virtue of (\ref{eqcorgrad2}) and
(\ref{eqcorgrad3}), we see that  $|\nabla\eta_\varepsilon|\leq C$
which implies that, in those regions, the family $\eta_\varepsilon$
is in fact uniformly Lipschitz continuous, as $\varepsilon\to 0$.

The proof of the corollary is complete.
\end{proof}
\begin{rem}\label{remCAffarelli}
Uniform H\"{o}lder $C^{0,\alpha},\ 0<\alpha<1$, bounds for
Gross-Pitaevskii systems, where the singular limit functions have
Lipschitz regularity ($\alpha=1$), have been proven in
\cite{terraciniCPAM} by blow-up techniques and the monotonicity
formulae of Almgren and Alt, Caffarelli, and Friedman (see also
\cite{caffarelli-lin}). \emph{The importance of our result lies in
the fact that the H\"{o}lder exponent $1/2$ in (\ref{eqmax}) equals
to the \emph{exact} maximal H\"{o}lder regularity of the singular
limit profile.} To the best of our knowledge, this property has been
proven, in singular perturbation problems, only in one-dimensional
problems, see \cite{berestycki-wei2012}.
\end{rem}

The following corollary answers a question posed to one of us by A.
Tertikas in relation with \cite{sourdis-fife}.
\begin{cor}\label{cortertikas}
Given $\alpha\in [0,\frac{1}{2})$, we have
\begin{equation}\label{eqtertikas1}
\eta_\varepsilon\to \sqrt{A^+}\ \ \textrm{in}\ \
C^\alpha(\mathbb{R}^2)\ \  \textrm{as}\ \ \varepsilon\to 0,
\end{equation}
but $\eta_\varepsilon$ does \emph{not} converge to $\sqrt{A^+}$ in
$C^\frac{1}{2}(\mathbb{R}^2)$ as $\varepsilon\to 0$.
\end{cor}
\begin{proof}
Given $\alpha\in [0,\frac{1}{2})$, in view of (\ref{eqn2toA+}),
Corollary \ref{cormax}, and the compactness of the embedding
$C^{\frac{1}{2}}(\bar{2\mathcal{D}_0}) \hookrightarrow
C^\alpha(\bar{2\mathcal{D}_0})$ (see \cite{Gilbarg-Trudinger}),  we
find that $\eta_\varepsilon\to \sqrt{A^+}$ in
$C^\alpha(\bar{2\mathcal{D}_0})$ as $\varepsilon\to 0$. Now, the
desired relation (\ref{eqtertikas1}) follows via (\ref{eqestim2+})
and (\ref{eqcorgrad3}).

On the other hand, the following simple argument shows that we do
not have convergence in $C^\frac{1}{2}$. Let $\{\varepsilon_n\}$ be
a decreasing sequence such that $\varepsilon_n\to 0$, and, abusing
notation, consider the points
$\textbf{y}_n=(t_n,\theta_n)=(-\varepsilon_n^\frac{2}{3},0)$. If
$n<m$, thanks to (\ref{eqbita}), (\ref{eqestim1}), and
(\ref{eqnambla}), we get
\[
\|\eta_{\varepsilon_n}-\eta_{\varepsilon_m}\|_{C^\frac{1}{2}(\mathbb{R}^2)}\geq
\frac{\left|\eta_{\varepsilon_n}(\textbf{y}_n)-\eta_{\varepsilon_m}(\textbf{y}_m)\right|}{|\textbf{y}_n-\textbf{y}_m|^\frac{1}{2}}
\geq c
\left(\frac{\varepsilon_n^\frac{1}{3}-\varepsilon_m^\frac{1}{3}}{\varepsilon_n^\frac{1}{3}+\varepsilon_m^\frac{1}{3}}\right)^\frac{1}{2}-C\varepsilon_n^\frac{2}{3},
\]
where $c,C$ are independent of $n,m$. Now, choosing for example
$\varepsilon_n=(\frac{1}{n})^3$ and $m=2n$, we conclude that
$\{\eta_{\varepsilon_n}\}$ is not Cauchy in $C^\frac{1}{2}$.

The proof of the corollary is complete.
\end{proof}
\begin{rem}\label{remholderalpha}
We expect that there exists some constant $C>0$ such that, for small
$\varepsilon>0$, we have
\[
\|\eta_\varepsilon-\sqrt{A^+}\|_{C^{\alpha}(\mathbb{R}^2)}\leq C
\varepsilon^{\frac{2}{3}(\frac{1}{2}-\alpha)},\ \  0\leq\alpha\leq
\frac{1}{2},
\]
see also Remark \ref{remestimglobal} below.
\end{rem}

The following result is motivated from (\ref{eqJ2}):
\begin{cor}\label{cor2}
Given $0<\alpha\leq 1$, if $\varepsilon$ is sufficiently small, and
$A$ as in (\ref{eqA}), there exists a constant $C>0$ such that
\begin{equation}\label{eqJ2improved}
|\eta_\varepsilon-\sqrt{A^+}|\leq
C\left(\varepsilon^{2-\frac{5}{3}\alpha}+\varepsilon^{-\frac{1}{3}\alpha}|\lambda_\varepsilon-\lambda_0|\right)\leq
C(\varepsilon^{2-2\alpha}+\varepsilon^{\frac{4}{3}-\frac{2}{3}\alpha})\sqrt{A^+},
\end{equation}
at points in $\mathcal{D}_0$ whose distance from $\partial
\mathcal{D}_0$ is greater than $\varepsilon^{\frac{2}{3}\alpha}$.
\end{cor}
\begin{proof} From (\ref{eqD0}), (\ref{eqVnormal}), (\ref{eqA}),
(\ref{eqlambdaepsilon}),  and (\ref{eqestim2}), it follows that in
the region $
\mathcal{D}_0\backslash\{-\frac{1}{2}\varepsilon^{\frac{2}{3}\alpha}<t<0\}$
we have:
\begin{equation}\label{eqA2}
\sqrt{A(\textbf{y})}\geq c \varepsilon^\frac{\alpha}{3},
\end{equation}
\begin{equation}\label{eqA-a}
\sqrt{A(\textbf{y})}-\sqrt{\lambda_\varepsilon-W(\textbf{y})}=\mathcal{O}(|\lambda_\varepsilon-\lambda_0|\varepsilon
^{-\frac{\alpha}{3}}),
\end{equation}
and
\begin{equation}\label{eqA3}
\eta_\varepsilon(\textbf{y})-\sqrt{\lambda_\varepsilon-W(\textbf{y})}=\mathcal{O}(\varepsilon^{2-\frac{5}{3}\alpha}),
\end{equation}
uniformly, as $\varepsilon\to 0$. Now, the assertion of the
corollary follows readily by combining (\ref{eqlambdaepsilonnew}),
Remark \ref{remdist}, and the above three relations.

The proof is complete.
\end{proof}

\begin{rem}\label{remimprove13}
Note that estimate (\ref{eqJ2improved}), when $\alpha=\frac{1}{2}$,
considerably improves estimate (\ref{eqJ2}), which was originally
proven in \cite{alama} (for solutions of (\ref{eqlagrange}) with
$\lambda_\varepsilon=\lambda_0$) and in the sequel used in
\cite{aftalion-jerrard}, \cite{alamamontero}, and \cite{ignat}.
\end{rem}
\begin{rem}\label{remestimglobal}
By (\ref{eqestim1})--(\ref{eqestim1+}), and (\ref{eqJ2improved})
with $\alpha=1$, for small $\varepsilon>0$, we obtain that
\[
\|\eta_\varepsilon-\sqrt{A^+}\|_{L^\infty(\mathbb{R}^2)}\leq
C\varepsilon^\frac{1}{3}.
\]
\end{rem}

Under an additional but natural non-degeneracy condition, satisfied
by most potentials used in physical applications (recall the
discussion in Subsection \ref{secmotivation}), we can improve
estimate (\ref{eqestim1+}).
\begin{pro}\label{prorefine}
If we assume that $W\in C^2$ and
\begin{equation}\label{eqsecondW}
W_{tt}(0,\theta)\geq c>0,\ \ \theta\in [0,\ell),
\end{equation}
then
\begin{equation}\label{eqestim1+improved}
\eta_{\varepsilon}(\textbf{y})=\varepsilon^\frac{1}{3}\beta_\varepsilon(\theta)V\left(\beta_\varepsilon(\theta)\frac{t}{\varepsilon^\frac{2}{3}}
\right)\left[1+\mathcal{O}(\varepsilon^\frac{2}{3})\left(\frac{t}{\varepsilon^\frac{2}{3}}\right)^\frac{5}{2}
\right],
\end{equation}
uniformly in $\left\{0\leq t \leq d,\ \ \theta\in
[0,\ell_\varepsilon) \right\}$, as $\varepsilon\to 0$, where $d>0$
is some small constant.
\end{pro}
\begin{proof}
Once again, we will work with the equivalent problem in stretched
variables. Our aim is to estimate $\varphi_*=u_\varepsilon-u_{ap}$
using equation (\ref{eqLphi=E+N}), as we did for estimate
(\ref{equepsilondecay}).
By virtue of (\ref{eqpainleve}), (\ref{eqVsimAiry}),
(\ref{eqairyasymptotic}), (\ref{equinremexplicit}), and
(\ref{equap}), for small $\varepsilon$, we get that the remainder in
(\ref{eqE}) satisfies
\begin{equation}\label{eqEbound}
|E|\leq C\varepsilon x^2\textrm{Ai}\ \ \textrm{if}\ x\in
[1,2\delta\varepsilon^{-\frac{2}{3}}).
\end{equation}
In view of (\ref{eqaepsilon}), (\ref{eq3uap1}), and
(\ref{eqsecondW}), it is easy to see that, decreasing $\delta$ if
necessary, we have
\begin{equation}\label{eqpotentialNEW}
\varepsilon^{-\frac{2}{3}}\left(3u_{ap}^2-a(\varepsilon^\frac{2}{3}y)\right)\geq
\beta^2 x+c\varepsilon^\frac{2}{3}x^2\ \ \textrm{if}\  x\in
[1,2\delta\varepsilon^{-\frac{2}{3}}),
\end{equation}
provided $\varepsilon$ is sufficiently small. We point out that in
the above relation the constant $c$ is independent of small
$\delta,\varepsilon$. Let $\Psi>0$ be determined from
\begin{equation}\label{eqpsiupperWtt}
-\Psi''+x\Psi=x^2\textrm{Ai};\ \ \Psi(1)=1,\ \Psi(\infty)=0.
\end{equation}
In order to proceed, we need some estimates for $\Psi$. A short
calculation shows that
\[
\Psi=(\textrm{Ai})h\ \ \textrm{with}\ \
h'=\frac{\int_{x}^{\infty}t^2(\textrm{Ai})^2dt}{(\textrm{Ai})^2}.
\]
Note that, from (\ref{eqairy}), (\ref{eqairyasymptotic}), we have
\begin{equation}\label{eqairy'}
\textrm{Ai}'\sim -x^\frac{1}{2}\textrm{Ai}\ \  \textrm{as}\ x\to
\infty.
\end{equation}
By the way, a neat way to show the above relation is to use
L'hospital's rule to find that
\[
\lim_{x\to \infty}\frac{x^{-1}(\textrm{Ai}')^2}{(\textrm{Ai})^2}=1.
\]
 Now, refereing  to L'hospital's rule once more, we get \[
  h'\sim
\frac{1}{2}x^\frac{3}{2}\ \ \textrm{and}\ \
h\sim\frac{1}{5}x^\frac{5}{2}\ \ \textrm{as} \ x\to \infty.
\]
 Hence, we find that
\begin{equation}\label{eqh}
\Psi\sim \frac{1}{5}x^\frac{5}{2}\textrm{Ai}\ \ \textrm{and}\ \
\Psi'\sim -\frac{1}{5}x^3\textrm{Ai}\ \ \textrm{as}\ \ x\to \infty.
\end{equation}
Keeping in mind that $x=\beta(\varepsilon^\frac{2}{3}z) s$, via
formulas (\ref{eqLaplace})--(\ref{eqB1}), we get that
\begin{equation}\label{eqLaplacePsi}
\Delta_y \Psi=\beta^2
\Psi''+\mathcal{O}(\varepsilon^\frac{4}{3})x\Psi'+\mathcal{O}(\varepsilon^\frac{4}{3})x^2\Psi''+\mathcal{O}(\varepsilon^\frac{2}{3})\Psi',
\end{equation}
uniformly in $\left\{1\leq x \leq
2\delta\varepsilon^{-\frac{2}{3}},\ \ z\in
[0,\ell_\varepsilon\varepsilon^{-\frac{2}{3}}) \right\}$, as
$\varepsilon\to 0$ (note that here $\mathcal{O}(\cdot)$ is bounded
uniformly in small $\delta$). Using (\ref{eqVnormal}),
(\ref{eqLoper}), (\ref{eqpotentialNEW})--
(\ref{eqLaplacePsi}), and further decreasing $\delta$, we readily
find that
\begin{equation}\label{eqL1}
\begin{array}{lll}
  -\mathcal{L}_\varepsilon(\Psi)
& \geq & \beta^2x^2\textrm{Ai}+c\varepsilon^\frac{2}{3}x^2\Psi
+\mathcal{O}(\varepsilon^\frac{4}{3})x^3\Psi+\mathcal{O}(\varepsilon^\frac{4}{3})x^4\textrm{Ai}+\mathcal{O}(\varepsilon^\frac{2}{3})x^3\textrm{Ai} \\
    &   &   \\
    & \geq & \frac{\beta^2}{2}x^2\textrm{Ai}+c\varepsilon^\frac{2}{3}x^2\Psi
\end{array}
\end{equation}
if $x\in [1,2\delta\varepsilon^{-\frac{2}{3}})$, provided
$\varepsilon$ is sufficiently small ($c,\ \mathcal{O}(\cdot)$
independent of small $\delta$). It is nice to note that, in the
above calculation, the cubic power  in the second asymptotic
relation of (\ref{eqh}) was ``the most appropriate'' one in order to
absorb the last term of (\ref{eqLaplacePsi}) into the term
$\beta^2x^2\textrm{Ai}$, by decreasing $\delta$. Similarly, keeping
in mind (\ref{eqairy})--(\ref{eqairyasymptotic}), we find that the
function
\begin{equation}\label{eqBfunction}
B(y)\equiv\textrm{Bi}(x-2\delta\varepsilon^{-\frac{2}{3}}),\ \ \
y\in \{x\in [1,2\delta\varepsilon^{-\frac{2}{3}}),\ z\in
[0,\varepsilon^{-\frac{2}{3}}\ell_\varepsilon)\},
\end{equation}
satisfies
\begin{equation}\label{eqL2}
\begin{array}{lll}
  -\mathcal{L}_\varepsilon(B)&\geq& -\beta^2B''+\beta^2xB+c\varepsilon^\frac{2}{3}x^2B+\mathcal{O}(\varepsilon^\frac{4}{3})x^2(x-2\delta\varepsilon^{-\frac{2}{3}})B
  +\mathcal{O}(\varepsilon^\frac{2}{3})B'
   \\
  &&\\
  & \geq & 2\delta\beta^2\varepsilon^{-\frac{2}{3}}B+c\varepsilon^\frac{2}{3}x^2B+\delta^3\mathcal{O}(\varepsilon^{-\frac{2}{3}})B
  +\mathcal{O}(\varepsilon^\frac{2}{3})\left(|x|^\frac{1}{2}+\delta^\frac{1}{2}\varepsilon^{-\frac{1}{3}} \right)B
   \\
   &  &  \\
    & \geq &
    \delta\beta^2\varepsilon^{-\frac{2}{3}}B+c\varepsilon^\frac{2}{3}x^2B,
\end{array}
\end{equation}
if $\varepsilon$ is sufficiently small, having further decreased
$\delta$ if necessary. From now on we will fix $\delta$. In view of
(\ref{eqN}), (\ref{eqEbound}),
 (\ref{eqL1}), and (\ref{eqL2}), given $M>1$, the
function
\[
\Phi(\textbf{y})\equiv M\varepsilon\left\{
\Psi(x)+\textrm{Bi}(x-2\delta\varepsilon^{-\frac{2}{3}})\right\}
\]
  satisfies
  \[\begin{array}{lll}
      -\mathcal{L}_\varepsilon(\Phi)+N(\Phi)+E & \geq & M\frac{\beta^2}{2}\varepsilon x^2\textrm{Ai}+cM\varepsilon^\frac{5}{3}x^2\Psi
      +M\delta\beta^2 \varepsilon^\frac{1}{3}B+cM\varepsilon^\frac{5}{3}x^2B \\
        &   & -C\varepsilon x^2\textrm{Ai}
        -CM^2\varepsilon^\frac{5}{3}x^5V(\textrm{Ai})^2-CM^2\varepsilon^\frac{5}{3}x^\frac{5}{2}V\textrm{Ai}B\\
        &   &  -CM^2\varepsilon^\frac{5}{3}VB^2
        -CM^3\varepsilon^\frac{7}{3}x^\frac{5}{2}\textrm{Ai}-CM^3\varepsilon^\frac{7}{3}B
\\
        &   &   \\
        &  \geq & M\frac{\beta^2}{4}\varepsilon x^2\textrm{Ai}+M\delta \beta^2\varepsilon^\frac{1}{3}B
        +c_1M\varepsilon^\frac{5}{3}x^\frac{9}{2}\textrm{Ai}-C_2M^2\varepsilon^\frac{5}{3}x^5V(\textrm{Ai})^2,
    \end{array}
  \]
for some constants $c_1,C_2>0$ (independent of $\varepsilon$),  if
$\varepsilon<\varepsilon(M)$ is sufficiently small. Note that the
seventh, eighth, and tenth term in the first inequality's righthand
side were absorbed into the corresponding third term, whereas the
ninth into the first. What we want to do next is to somehow ``get
rid'' of the last term in the above relation, and end up with a
positive righthand side. We will achieve this by absorbing that term
into the one that proceeds it. By virtue of (\ref{eqVsimAiry}),
(\ref{eqairyasymptotic}), and (\ref{eqairy'}), if $M$ is
sufficiently large, there exists $x_M>0$ such that
\begin{equation}\label{eqxm}
C_2 M x^\frac{1}{2}V\textrm{Ai}<c_1 \ \ \textrm{if}\ x>x_M.
\end{equation}
(Note that $x_M\to \infty$ as $M\to \infty$). Hence, it follows
that
\begin{equation}\label{eqcomparisonEq}
 -\mathcal{L}_\varepsilon(\Phi)+N(\Phi)+E\geq M\delta \beta^2\varepsilon^\frac{1}{3}B+M\frac{\beta^2}{4}\varepsilon
 x^2\textrm{Ai}>0
\end{equation}
in the strip-like domain described by
\[\mathcal{S}_\varepsilon=\left\{x_M\leq x\leq
2\delta\varepsilon^{-\frac{2}{3}},\ \ z\in
[0,\ell_\varepsilon\varepsilon^{-\frac{2}{3}}) \right\}.\]
 Now, in view of
(\ref{equepsilondecay}), (\ref{eqh}), and (\ref{eqxm}), we can fix a
large
  $M>0$ such that
\begin{equation}\label{eqcomparisonbdry}
\varphi_*<\Phi\ \ \textrm{on}\ \ \{x=x_M\};\ \ \varphi_*<\Phi\ \
\textrm{on}\ \ \{x=2\delta\varepsilon^{-\frac{2}{3}}\},
\end{equation}
if $\varepsilon$ is small. By (\ref{eqcomparisonEq}),
(\ref{eqcomparisonbdry}), and a standard maximum principle argument,
making use of (\ref{equappotential}) and the property that
\[
\left|N(\Phi)-N(\varphi_*) \right|\leq
C\varepsilon^\frac{2}{3}|\Phi-\varphi_*|\ \ \textrm{in}\
\mathcal{S}_\varepsilon,
\]
 we deduce
that $\varphi_*\leq \Phi$ in $\mathcal{S}_\varepsilon$, if
$\varepsilon$ is small.
 Similarly we can show that
$\varphi_*\leq -\Phi$ in $\mathcal{S}_\varepsilon$, if $\varepsilon$
is small. The desired assertion of the proposition (for the
equivalent stretched problem) now follows via (\ref{eqestim1+}),
(\ref{eqVsimAiry}), (\ref{eqh}), and noting that
\[
\textrm{Bi}(x-2\delta\varepsilon^{-\frac{2}{3}})\leq \textrm{Bi}(-
\delta\varepsilon^{-\frac{2}{3}})\stackrel{(\ref{eqairyasymptotic})}{\leq}
2 \textrm{Ai}( \delta\varepsilon^{-\frac{2}{3}})\leq 2
\textrm{Ai}(x)
\]
if $x\in [x_M,\delta\varepsilon^{-\frac{2}{3}})$ with $\varepsilon$
sufficiently small.

The proof of the proposition is complete.
\end{proof}
\begin{rem}\label{remexpimproved}
Note that the first term in the righthand side of
(\ref{eqestim1+improved}) dominates for $0\leq t \ll
\varepsilon^\frac{2}{5}$.
\end{rem}

Assuming additionally that $W$ is radial and convex (outside of
$\mathcal{D}_\varepsilon$), we can derive an explicit global upper
bound on the minimizer.

\begin{pro}\label{proRadial}
Assume that the potential trap $W$ is radially symmetric with
\[W_{rr}(R_\varepsilon)>0, \ \ \textrm{and}\ \  W_{rr}(r)\geq 0 \ \ \textrm{if}\ \  r>
R_\varepsilon,\] where $R_\varepsilon$ denotes the radius of
$\mathcal{D}_\varepsilon$. Then, we have
\begin{equation}\label{eqcorupperboundrs}
\eta_\varepsilon(s)\leq \frac{\textrm{Ai}
\left(\beta_\varepsilon\frac{s-R_\varepsilon}{\varepsilon^\frac{2}{3}}\right)}
{\textrm{Ai}\left(\beta_\varepsilon\frac{r-R_\varepsilon}{\varepsilon^\frac{2}{3}}\right)}
\eta_\varepsilon(r)\ \ \ \forall\ s\geq r\geq R_\varepsilon.
\end{equation}
In particular, it holds that
\begin{equation}\label{eqcorupperbounds}
\eta_\varepsilon(s)\leq
\varepsilon^\frac{1}{3}\left(\beta_\varepsilon+o(1)
\right)\textrm{Ai}
\left(\beta_\varepsilon\frac{s-R_\varepsilon}{\varepsilon^\frac{2}{3}}\right)
\ \ \textrm{if}\ s-R_\varepsilon \gg \varepsilon^\frac{2}{3},
\end{equation}
as $\varepsilon\to 0$.
\end{pro}
\begin{proof}
Since
\[
-a_\varepsilon(r)=W(r)-W(R_\varepsilon)\geq
W_r(R_\varepsilon)(r-R_\varepsilon),\ \ \ r\geq R_\varepsilon,
\]
we see that the minimizer $\eta_\varepsilon$ is a positive
lower-solution of the linear equation
\begin{equation}\label{eqairyremark}
-\varepsilon^2\left(\eta_{rr}+\frac{1}{r}\eta_r\right)
+\beta_\varepsilon^3(r-R_\varepsilon)\eta=0,
\end{equation}
if $r\geq R_\varepsilon$ (recall that
$\beta_\varepsilon=\left[W_r(R_\varepsilon) \right]^\frac{1}{3}>0$).
On the other side, making use  of (\ref{eqairy}) and the fact that
$(\textrm{Ai})'<0$, we readily find that
$\textrm{Ai}\left(\beta_\varepsilon\frac{r-R_\varepsilon}{\varepsilon^\frac{2}{3}}\right)$
is a positive upper-solution of (\ref{eqairyremark}) if $r\geq
R_\varepsilon$. Hence, by the maximum principle, we deduce that
relation (\ref{eqcorupperboundrs}) holds true. In turn, via
(\ref{eqestim1+}), (\ref{eqVsimAiry}), (\ref{eqgamma1}), and
Proposition \ref{prorefine}, relation (\ref{eqcorupperboundrs})
implies the validity of relation (\ref{eqcorupperbounds}).

The proof of the proposition is complete.
\end{proof}

\begin{rem}\label{remXinfu}
It seems plausible that the techniques of the very recent paper
\cite{chenXinfuJDE12} can be extended to derive a WKB
(Wentzel-Kramers-Brillouin) type estimate, in the region
$(R,\infty)$ (with the obvious notation), for the ground state
solution of (\ref{eqground}) with $q=2$ and $N=1$.
\end{rem}

\section{Refined estimates for the auxiliary functions $\xi_\varepsilon,\
f_\varepsilon$ in the case of radial symmetry}\label{secfepsilon} In
this section, restricting ourselves to radial potentials with
$\mathcal{D}_0$ a ball, building on our previous results for the
ground state $\eta_\varepsilon$, we will improve upon the estimates
obtained recently in \cite{aftalion-jerrard} for the auxiliary
function $f_\varepsilon$ in (\ref{eqfepsilon}). As we have already
discussed in Subsection \ref{secknown}, the latter estimates were
essential for the analysis of \cite{aftalion-jerrard} regarding the
functional $E_\varepsilon$, defined in (\ref{eqEfunctional}). We
believe that the improved estimates herein
may provide important intuition for the treatment of the general
case, which may ultimately lead to the resolution of the open
problem raised in \cite{aftalion-jerrard} (recall the discussion in
Subsection \ref{secknown}).

In the general case, for potentials as described in Subsection
\ref{secproblem}, we define $\xi_\varepsilon$ to be the solution of
\begin{equation}\label{eqxidiv}
\textrm{div}\left( \frac{1}{\eta_\varepsilon^2}\nabla \xi\right)=-2,
\ \textbf{y}\in\mathbb{R}^2,\ \ \xi_\varepsilon(\textbf{y})\to 0,\
|\textbf{y}|\to \infty,
\end{equation}
so that $\nabla^\bot \xi_\varepsilon=x^\bot \eta_\varepsilon^2$. An
integration by parts in (\ref{eqFepsilonfuncitonal}) yields
\begin{equation}\label{eqEepsintegrbparts}
F_\varepsilon(w)=\int_{\mathbb{R}^2}^{}\left\{
\frac{\eta_\varepsilon^2}{2}\left(|\nabla w|^2-\frac{4\Omega
\xi_\varepsilon}{\eta_\varepsilon^2}Jw\right)+\frac{\eta_\varepsilon^4}{4\varepsilon^2}\left(|w|^2-1
\right)^2\right\}d\textbf{y},
\end{equation}
where $Jw=\frac{1}{2}\nabla \times (iw,\nabla w)=(iw_{\textbf{y}_1},
w_{\textbf{y}_2})$ is the Jacobian.

We recall that the function $f_\varepsilon:=
\xi_\varepsilon/\eta_\varepsilon^2$, appearing in the functional
$F_\varepsilon$, is important since it is well known that vortices
in the interior of $\mathcal{D}_0$ first appear near where this
function attains a local maximum
\cite{alama,aftalionbook,ignat,ignat2}; its importance is also clear
from (\ref{eqEepsintegrbparts}), since it controls the relative
strength of the positive and negative contributions to
$F_\varepsilon$.

In the case where the potential $W$ is radially symmetric,  one can
solve problem (\ref{eqxidiv}) explicitly to find that the functions
$\xi_\varepsilon, \ f_\varepsilon$ are given by relation
(\ref{eqfepsilon}). In particular, if the domain $\mathcal{D}_0$ is
a ball, it has been shown in \cite{aftalion-jerrard} that, for small
$\varepsilon$, we have
\begin{equation}\label{eqaftalion}
f_\varepsilon(|\textbf{y}|)\leq \left\{
\begin{array}{ll}
  C\textrm{dist}(\textbf{y},\partial \mathcal{D}_0)+C\varepsilon^\frac{2}{3} & \textrm{if}\ \textbf{y}\in \mathcal{D}_0, \\
   &   \\
  C\varepsilon^\frac{2}{3} & \textrm{if\ not},\end{array}
\right. \ \ \textrm{and}\ \ \
\|f_\varepsilon-f_0\|_{L^\infty(\mathbb{R})}\leq
C\varepsilon^\frac{1}{3},
\end{equation}
where $f_0$ is the function in (\ref{eqf0}) below, which solves the
``limiting" problem corresponding to (\ref{eqxidiv}):
\begin{equation}\label{eqxi0limitequation}
\textrm{div}\left( \frac{1}{A}\nabla \xi\right)=-2 \ \textrm{in}\
\mathcal{D}_0,\ \ \xi_\varepsilon=0 \ \textrm{on}\
\partial\mathcal{D}_0,
\end{equation}
with $A$ as in (\ref{eqA}). Existence and properties of a positive
solution $\xi_0$ of (\ref{eqxi0limitequation}) have been established
in \cite{alamaPinning}.


The following proposition refines and improves relation
(\ref{eqaftalion}).
\begin{pro}\label{profeps}
If the potential $W$ is radially symmetric with
$\mathcal{D}_0=\{\textbf{y}\in \mathbb{R}^2\ :\ r=|\textbf{y}|<R
\}$, then the function $f_\varepsilon$, defined in
(\ref{eqfepsilon}), satisfies
\begin{equation}\label{eqfepsprostatement}
f_\varepsilon(r)=R_\varepsilon\beta_\varepsilon^{-1}\varepsilon^\frac{2}{3}V^{-2}\left(\beta_\varepsilon\frac{r-R_\varepsilon}{\varepsilon^\frac{2}{3}}
\right)\int_{\beta_\varepsilon\frac{r-R_\varepsilon}{\varepsilon^\frac{2}{3}}}^{\infty}V^2(\sigma)d\sigma+o(\varepsilon^\frac{2}{3}),
\end{equation}
uniformly in $[R-o(\varepsilon^\frac{1}{3}),\infty)$, as
$\varepsilon\to 0$, where $R_\varepsilon$ denotes  the radius of the
ball $\mathcal{D}_\varepsilon$, satisfying (\ref{eqRepsilon}) below,
and $\beta_\varepsilon=[W'(R_\varepsilon)]^\frac{1}{3}$ (recall
(\ref{eqbita})).

Moreover, if $\varepsilon$ is small, it holds that
\begin{equation}\label{eqfeps-f0state}
\|f_\varepsilon-f_0\|_{L^\infty(\mathbb{R})}\leq
C\varepsilon^\frac{1}{2},
\end{equation}
where
\begin{equation}\label{eqf0}
f_0(r)=\left\{\begin{array}{ll}
             \frac{1}{A(r)}\int_{r}^{R}sA(s)ds, & 0\leq r <R, \\
               &   \\
             0, & r\geq R.
           \end{array}
 \right.
\end{equation}
(In view of (\ref{eqD0}), and (\ref{eqVnormal}), an application of
L'hospital's rule shows that $f_0'(R^-)=-\frac{R}{2}$).
\end{pro}
\begin{proof}
First of all note that, thanks to (\ref{eqlambdaepsilonnew}), we
have
\begin{equation}\label{eqRepsilon}
R_\varepsilon=R+\mathcal{O}(\varepsilon^\frac{4}{3})\ \ \textrm{as}\
\varepsilon\to 0.
\end{equation}

 By virtue of (\ref{eqainfinity}), (\ref{eqD0}),
(\ref{eqVnormal}),
 and (\ref{eqlambdaepsilonnew}), there exists a constant $c>0$ such that, given $K\gg 1$,
 we have
\[
\eta_\varepsilon^2+W(r)-\lambda_\varepsilon\geq
W(r)-\lambda_\varepsilon \geq c (r-R)^2+cK\varepsilon^\frac{2}{3}, \
\ r\geq R+K\varepsilon^\frac{2}{3},
\]
provided that $\varepsilon$ is sufficiently small. (We note that $K$
is considered fixed in the corresponding relation in
\cite{aftalion-jerrard}). Then, by a standard barrier argument in
equation (\ref{eqlagrange}), we deduce that
\[
\eta_\varepsilon(s)\leq
\eta_\varepsilon(r)e^{-K^\frac{1}{3}\varepsilon^{-\frac{2}{3}}(s^2-r^2)},\
\ s\geq r\geq R+K\varepsilon^\frac{2}{3},
\]
if $K$ is sufficiently large and $\varepsilon$ sufficiently small.
As a result, we get
\begin{equation}\label{eqappend0}
f_\varepsilon(r)\leq \int_{r}^{\infty}
se^{-2K^\frac{1}{3}\varepsilon^{-\frac{2}{3}}(s^2-r^2)}ds=
\frac{1}{4}K^{-\frac{1}{3}}\varepsilon^\frac{2}{3},\ \ r\geq
R+K\varepsilon^\frac{2}{3},
\end{equation}
if $\varepsilon$ is small.

If $r\in [R-d', R+K\varepsilon^\frac{2}{3}]$, in view of Corollary
\ref{cor1} and (\ref{eqappend0}), for small $\varepsilon$, we have
\begin{equation}\label{eq1}
\begin{array}{lll}
  f_\varepsilon(r) & = & \frac{1}{\eta^2_\varepsilon(r)}\int_{r}^{R+K\varepsilon^\frac{2}{3}}s
\eta_\varepsilon^2(s)ds+\frac{\eta_\varepsilon^2(R+K\varepsilon^\frac{2}{3})}{\eta_\varepsilon^2(r)}f_\varepsilon(R+K\varepsilon^\frac{2}{3}) \\
    &   &   \\
    & = & \frac{1}{\eta^2_\varepsilon(r)}\int_{r}^{R+K\varepsilon^\frac{2}{3}}s
\eta_\varepsilon^2(s)ds+\mathcal{O}(K^{-\frac{1}{3}}\varepsilon^\frac{2}{3}),
\end{array}
\end{equation}
uniformly as $\varepsilon\to 0$. If $r\in [R,
R+K\varepsilon^\frac{2}{3}]$, from (\ref{eqestim1+}),
(\ref{eqvasympto}), and (\ref{eqRepsilon}), it follows readily that
\begin{equation}\label{eq1+}
\begin{array}{lll}
  \frac{1}{\eta^2_\varepsilon(r)}\int_{r}^{R+K\varepsilon^\frac{2}{3}}s
\eta_\varepsilon^2(s)ds & = & R_\varepsilon
\beta_\varepsilon^{-1}\varepsilon^\frac{2}{3}
V^{-2}\left(\beta_\varepsilon\frac{r-R_\varepsilon}{\varepsilon^\frac{2}{3}}
\right)\int_{\beta_\varepsilon\frac{r-R_\varepsilon}{\varepsilon^\frac{2}{3}}}^{\infty}V^2(\sigma)d\sigma \\
    &   &   \\
    &   & +\mathcal{O}\left(\varepsilon^\frac{2}{3}V^{-2}(K)\int_{K}^{\infty}V^2(\sigma)d\sigma\right)
+\mathcal{O}_K(\varepsilon^\frac{4}{3}),
\end{array}
\end{equation}
uniformly, as $\varepsilon\to 0$ (the constant $\mathcal{O}_K(1)$
may diverge as $K\to \infty$). The second term in the righthand side
of the above relation can be estimated as before, by noting that,
thanks to (\ref{eqpainleve}), we find that
\[ V(\sigma)\leq
V(K)e^{-K^\frac{1}{3}(\sigma-K)},\ \ \sigma\geq K\ \ \textrm{if}\ K
\ \textrm{is\ large}.
\]
Thus, by (\ref{eq1}), (\ref{eq1+}), we infer that relation
(\ref{eqfepsprostatement}) holds true in
$[R,R+K\varepsilon^\frac{2}{3}]$. In fact, by (\ref{eqappend0}) and
the above relation, we deduce that (\ref{eqfepsprostatement}) holds
true in $[R+K\varepsilon^\frac{2}{3}, \infty)$ as well. If $r\in
[R-d,R]$, similarly as before, but this time using (\ref{eqestim1})
instead of (\ref{eqestim1+}), we arrive at
\begin{equation}\label{eq1++}
\begin{array}{lll}
  f_\varepsilon(r) & = & R_\varepsilon
\beta_\varepsilon^{-1}\varepsilon^\frac{2}{3}
V^{-2}\left(\beta_\varepsilon\frac{r-R_\varepsilon}{\varepsilon^\frac{2}{3}}
\right)\int_{\beta_\varepsilon\frac{r-R_\varepsilon}{\varepsilon^\frac{2}{3}}}^{\infty}V^2(\sigma)d\sigma \\
    &   &   \\
    &   &
    +\mathcal{O}\left(|r-R_\varepsilon|^2+\varepsilon|r-R_\varepsilon|^\frac{1}{2}+\varepsilon^\frac{1}{3}|r-R_\varepsilon|^\frac{3}{2}\right)
    +o(\varepsilon^\frac{2}{3}),
\end{array}
\end{equation}
uniformly, as $\varepsilon\to 0$. The above relation implies at once
the validity of (\ref{eqfepsprostatement}) in
$[R-o(\varepsilon^\frac{1}{3}), R]$ as $\varepsilon\to 0$.
Consequently, we have established the validity of
(\ref{eqfepsprostatement}).

Next, we will show the validity of estimate (\ref{eqfeps-f0state}).
If $r\in [R-\varepsilon^{\frac{2\alpha}{3}},R]$, with
$\frac{1}{2}<\alpha\leq 1$, recalling (\ref{eqVnormal}), we have
$c(R-r)\leq A(r)\leq C(R-r)$. So, as in \cite{aftalion-jerrard}, we
obtain that
\begin{equation}\label{eqf0in}
f_0(r)\leq \frac{C}{R-r}\int_{r}^{R}s(R-s)ds\leq C(R-r)\leq
C\varepsilon^\frac{2\alpha}{3}.
\end{equation}
Furthermore, thanks to (\ref{eqvasympto}),
(\ref{eqfepsprostatement}), if $\varepsilon$ is small, we find that
\begin{equation}\label{eq512+}
f_\varepsilon(r)\leq C\varepsilon^\frac{2\alpha}{3},\ \ r\in
[R-\varepsilon^{\frac{2\alpha}{3}},R].
\end{equation}
If $r\in[0,R-\varepsilon^\frac{2\alpha}{3}]$, following
\cite{aftalion-jerrard}, we write
\begin{equation}\label{eqfeps-f0inproof}
\begin{array}{lll}
  f_\varepsilon(r)-f_0(r) & = & \left\{\frac{1}{\eta_\varepsilon^2(r)}\int_{r}^{R-\varepsilon^\frac{2\alpha}{3}}s\eta_\varepsilon^2(s)ds-
\frac{1}{A(r)}\int_{r}^{R-\varepsilon^\frac{2\alpha}{3}}sA(s)ds
\right\}
 \\
    &   &   \\
   &  &
   +\frac{\eta_\varepsilon^2(R-\varepsilon^\frac{2\alpha}{3})}{\eta_\varepsilon^2(r)}f_\varepsilon(R-\varepsilon^\frac{2\alpha}{3})-
   \frac{A(R-\varepsilon^\frac{2\alpha}{3})}{A(r)}f_0(R-\varepsilon^\frac{2\alpha}{3})\\
   & & \\
   &=&I+II-III.
\end{array}
\end{equation}
Using (\ref{eqD0}), (\ref{eqVnormal}), (\ref{eqestim2}), Corollary
\ref{cor1}, and our earlier estimates on $f_\varepsilon,\ f_0$ for
$r\geq R-\varepsilon^\frac{2\alpha}{3}$, we see that
\[
|II|\leq Cf_\varepsilon(R-\varepsilon^\frac{2\alpha}{3})\leq
C\varepsilon^\frac{2\alpha}{3}\ \ \textrm{and}\ \ |III|\leq
Cf_0(R-\varepsilon^\frac{2\alpha}{3})\leq
C\varepsilon^\frac{2\alpha}{3}.
\]
We further decompose the remaining term as
\[
I=\left(\frac{1}{\eta_\varepsilon^2(r)}-\frac{1}{A(r)}
\right)\int_{r}^{R-\varepsilon^\frac{2\alpha}{3}}s\eta_\varepsilon^2(s)ds
+\frac{1}{A(r)}
\int_{r}^{R-\varepsilon^\frac{2\alpha}{3}}s\left(\eta_\varepsilon^2(s)-A(s)\right)ds.
\]
Using Corollary \ref{cor2}, if $\frac{1}{2}<\alpha<1$, for small
$\varepsilon$, it follows that
\[
|I|\leq
C\varepsilon^{2-2\alpha}\int_{r}^{R-\varepsilon^\frac{2\alpha}{3}}s\frac{\eta_\varepsilon^2(s)}{\eta_\varepsilon^2(r)}ds+
C\varepsilon^{2-2\alpha}\int_{r}^{R-\varepsilon^\frac{2\alpha}{3}}s\frac{A(s)}{A(r)}ds.
\]
Due to Corollary \ref{cor1}, we have
$\frac{\eta_\varepsilon^2(s)}{\eta_\varepsilon^2(r)}\leq 1$ if
$R-d'\leq r \leq s \leq R-\varepsilon^\frac{2\alpha}{3}$. If $0\leq
r \leq R-d'$, then $\eta_\varepsilon^2(r)\geq c$, and so
$\frac{\eta_\varepsilon^2(s)}{\eta_\varepsilon^2(r)}\leq C$. Thus,
the first integral in the above relation is bounded by
$C\varepsilon^{2-2\alpha}$. The second integral is estimated
similarly, using (\ref{eqVnormal}) instead of Corollary \ref{cor1}.
Therefore, relation (\ref{eqfeps-f0inproof}) implies that
\[|f_\varepsilon(r)-f_0(r)|\leq C\varepsilon^{2-2\alpha},\ \
r\in[0,R-\varepsilon^\frac{2\alpha}{3}],
\]
provided that $\varepsilon$ is sufficiently small. The validity of
estimate (\ref{eqfeps-f0state}) follows at once by combining
(\ref{eqf0in}), (\ref{eq512+}), the above relation, and choosing
$\alpha=\frac{3}{4}$.

The proof of the proposition is complete.
\end{proof}

\section{Open problems and future directions}\label{secopen}
What follows is a list of questions which are currently unresolved.
These are presented as an illustration of where our interests lie.
No attempt is being made to be precise in their formulation.

 A  question that comes naturally to mind is to examine whether the estimates of Theorem
 \ref{thmmain},
 for the minimizer $\eta_\varepsilon$ of $G_\varepsilon$ in $\mathcal{H}$,
can be used to answer the interesting open problem posed recently in
\cite{aftalion-jerrard}. As we have already mentioned in Subsection
\ref{secknown}, the latter is to see to what extend the analysis of
\cite{aftalion-jerrard}, for the functional $E_\varepsilon$ in
(\ref{eqEfunctional}), continues to hold if one drops the assumption
of radial symmetry on the potential $W$. Hopefully, our estimates
for $\eta_\varepsilon$ can be used in estimating the corresponding
auxiliary functions $\xi_\varepsilon,\ f_\varepsilon$, arising in
the functional $E_\varepsilon$ as in (\ref{eqEepsintegrbparts}),
which for the radial case were given by (\ref{eqfepsilon}). It seems
that the elliptic problem (\ref{eqxidiv}), which determines
$\xi_\varepsilon$, seems to be a singular perturbation problem of
its own independent interest.

In the special case of the model harmonic potential, an approximate
solution for $(\ref{eqground})_-$, ``close'' to
$\sqrt{(\lambda-W)^+}$, of \emph{arbitrary order accuracy} was
constructed in \cite{pelinovsky} (keep in  mind Remark
\ref{rempeli}). We feel that it would be very interesting if one can
do the same thing for the case of general potential. A major
difficulty (or problem) is that each term of the inner expansion
diverges polynomially in a complicated manner as the distance from
$\Gamma$ increases (recall Remark \ref{remyan}), see also the
appendix in \cite{calnigap}. The construction of arbitrary order
approximations is especially important in the treatment of
singularly perturbed elliptic problems involving \emph{resonance},
where the order of accuracy of the approximation is dictated by the
space dimension, see for instance \cite{malchiodi-fife}. Problems of
these type which feature the presence of a corner layer (similar to
the problem at hand) have been studied recently in
\cite{karalisourdisresonance} (in two dimensions), see also Remark
\cite{hastingsTroy} herein. It should be noticed that in Allen--Cahn
or  (focusing) Schr\"{o}dinger type equations, where its possible to
construct arbitrary order approximations (see \cite{malchiodi-fife}
and the references therein), the phenomenon is exponentially
localized, i.e, the corresponding terms approach certain constants
exponentially fast.

Relation (\ref{equappotentialthmmain}) implies that the spectrum, in
$L^2(\mathbb{R}^2)$, of the operator $\textbf{L}_\varepsilon$,
defined in (\ref{eqLreal}), is bounded above by
$-c\varepsilon^\frac{2}{3}$, for some constant $c>0$, as
$\varepsilon\to 0$. We expect that, making further use of the
estimates of Theorem \ref{thmmain}, one can
rigorously ``link'' the spectrum of $\textbf{L}_\varepsilon$ to that
of the one--dimensional ``limit'' operator $\mathcal{M}$, defined in
(\ref{eqMcal}), as $\varepsilon\to 0$ (see also relations
(\ref{eqopenev1})--(\ref{eqopenev2}) below), and thus provide a
valid asymptotic approximation for the eigenvalues of
$\textbf{L}_\varepsilon$. In particular, the difference between the
first two eigenvalues, called the fundamental gap, is of importance
since it determines the rate at which positive solutions of the
nonlinear heat equation, corresponding to $(\ref{eqground})_-$,
approach the first eigenspace of $\textbf{L}_\varepsilon$ (see
\cite{batesjones}, \cite{henry}, and especially
\cite{brunovskyfiedler}). In the case where $W$ is the harmonic
potential, a rigorous connection between the spectrum of
$\textbf{L}_\varepsilon$ and that of $\mathcal{M}$ (see the
discussion following relation (\ref{eqmu1}) below), as
$\varepsilon\to 0$, has been made recently in \cite{pelinovsky}
 (see also \cite{karalisourdisradial} for a related
radially symmetric problem).
 A possible approach for the general case, where the potential is
as in the present paper, could be by mixing techniques found in the
aforementioned references with those developed in
\cite{stefanopoulos, chenxinfu} for the study of the spectrum of
multi--dimensional Allen-Cahn and related phase-field operators for
generic interfaces.
One could  even carry out an analogous program for the spectrum of
the linearization of the defocusing nonlinear Schr\"{o}dinger
equation $(\ref{eqNLS})_-$ at the corresponding ground state, recall
(\ref{eqgrillakis}). The latter problem is often referred to as the
Bogolyubov-de Gennes problem in the context of Bose--Einstein
condensates, see \cite{gallo,kevrekidis-pelinovsky} for recent
studies specializing on the model harmonic potential. In the latter
references, for reducing the complexity of the problem, the authors
linearized at $\eta_0$ (recall (\ref{eqeta0})) instead of the ground
state $\eta_\varepsilon$. In this case, the linear operator defined
by the lefthand side of the first equation in (\ref{eqgrillakis}),
with $\eta_0$ in place of $\eta_\varepsilon$, has also been studied
in \cite{fusco}, in relation with \cite{sourdis-fife}.

Excited states are solutions of $(\ref{eqground})_-$ with zero set
inside the domain $\mathcal{D}_\lambda\equiv\{\textbf{y}\ :\
W(\textbf{y})<\lambda\}$.
In the Thomas--Fermi limit, $\varepsilon\to 0$, the Bose--Einstein
condensate is a nearly compact cloud, which may contain localized
dips of the atomic density. The nearly compact cloud is modeled by
the ground state of the defocusing nonlinear Schr\"{o}dinger
equation $(\ref{eqNLS})_-$, whereas the localized dips are modeled
by the excited states.
In the one--dimensional case, with $W$ the harmonic potential,
excited states of $(\ref{eqground})_-$ which are approximated, as
$\varepsilon\to 0$, by a product of  the ground state and $m$ dark
solitons (localized waves of the defocusing NLS equation with
nonzero boundary conditions at respective infinities, which after a
re-scaling solve the one-dimensional (\ref{eqAllenCahn})) were
constructed in \cite{pelinovskynonlinear} by a finite--dimensional
Lyapunov--Schmidt reduction (for the latter see for instance the
book \cite{malchiodibook}).
 Loosely speaking, these solutions have a  corner
layer at the points corresponding to $\partial \mathcal{D}_\lambda$,
and $m$ (clustering)
 transition layers in $\left(-C|\ln \varepsilon|\varepsilon,C|\ln \varepsilon|\varepsilon\right)$, as $\varepsilon\to
 0$ (see also \cite{coles}).
Studies in the case of radial symmetry have been conducted in
\cite{kevrekidisexcitedradial}.
 We believe that, at least in two space dimensions, analogous excited states can still be
constructed without any symmetry assumptions on the potential, by
employing the estimates of Theorem \ref{thmmain} (in particular
(\ref{equappotentialthmmain})) and the infinite dimensional
Lyapunov-Schmidt reduction of \cite{delpinoarma} (see also
\cite{weicluster}). In this context, the dimension $N=2$ plays an
important role  for the solvability of a \emph{Toda system},
periodic orbits of which
 determine, up to principal order, the location of $m$ closed curves in $\mathcal{D}_\lambda$ where the excited state
changes sign.
 These  curves should collapse, as $\varepsilon\to 0$, to a
closed curve in $\mathcal{D}_\lambda$ that may be determined by the
arguments in \cite{guilayer}, \cite{nakashimaNonradial},
\cite{sakamoto} (if $N=1$, the interfaces collapse at critical
points of $W$). We expect that, in the case at hand, the reduction
procedure is more delicate than \cite{delpinoarma} because the
corresponding linear operator $\textbf{L}_\varepsilon$ has small
eigenvalues (see also \cite{dancer-lazer} for a related
finite-dimensional reduction).  If $N=1$ or $W$ is radial, one could
also try to construct ``high energy'' excited states of
$(\ref{eqground})_-$, having an increasing number of layers of order
 $1/\varepsilon$, as $\varepsilon\to 0$, in the spirit of
 \cite{felmermartinezJDE}, \cite{felmertanaka} and the references therein. We remark that the  result of \cite{felmertanaka} relied
on ODE techniques, but it is expectable that a similar result could
be proven for higher--dimensional problems. On the other hand, in
the one-dimensional case, solutions of $(\ref{eqground})_-$
bifurcating from the trivial branch have been studied in
\cite{kurth}, \cite{selemNodal}, and \cite{tsai} (see also
\cite{kirrCommPhysics} and the references therein). Let us make a
formal connection between these two different types of solutions
(layered and small amplitude respectively).
 Consider the one-dimensional case with
potential $W$ having a global minimum which is attained at a unique
point, say at $\textbf{y}=0$, and satisfying
\[
W(\textbf{y})=W(0)+c|\textbf{y}|^\alpha+o(|\textbf{y}|^\alpha)\ \
\textrm{as}\ \ \textbf{y}\to 0,
\]
for some constants $\alpha,c>0$. Arguing as in \cite[Prop.
3.25]{karalisourdisradial}, it is not hard to establish that, given
$m\in \mathbb{N}$, the first $m$ eigenvalues of the linear operator
\begin{equation}\label{eqopenev1}
-\varepsilon^2\partial_{\textbf{yy}}+\left(W(\textbf{y})-\lambda
\right)I,
\end{equation}
which corresponds to the linearization of $(\ref{eqground})$ about
the trivial solution, are of the form
\begin{equation}\label{eqopenev2}
W(0)-\lambda+\mu_i\varepsilon^\frac{2\alpha}{\alpha+1}+o\left(\varepsilon^\frac{2\alpha}{\alpha+1}\right)\
\ \textrm{as}\ \ \varepsilon\to 0,\ i=1,\cdots,m,
\end{equation}
where $\{\mu_i\}$ are the eigenvalues of the ``limit'' operator
\[
-\partial_{\textbf{yy}}+c|\textbf{y}|^\alpha I,
\]
(these exist by \cite[Thm. 10.7]{hislop}, and $\mu_i\to \infty$ as
$i\to \infty$). In passing, we not that formulas (\ref{eqopenev2})
improve the corresponding lower bounds found in Theorem 1.4 of
\cite{fusco}.
 Hence, we see that the number of negative
eigenvalues (counting multiplicities), namely the Morse index of the
trivial solution of $(\ref{eqground})_-$ (see
\cite{karalisourdisradial} for the precise definition), diverges as
$\varepsilon\to 0$ (recall that $\lambda>W(0)$). From a variant of
Weyl's asymptotic formula, see for example \cite[pg. 521]{BenderO},
it turns out that one has
\[
\mu_i\sim ci^{\frac{2\alpha}{\alpha+2}}\ \ \textrm{as}\  i\to
\infty,
\]
for some  constant $c>0$. We expect that, by refining the above
argument, one can prove that the Morse index of the trivial solution
is of order greater than or equal to $1/\varepsilon$, as
$\varepsilon\to 0$ (keep in mind that Landau's symbol in
(\ref{eqopenev2}) may depend on $m\gg 1$). On the other side, it
seems plausible that the operator in (\ref{eqopenev1}) does not have
any negative eigenvalues if $\varepsilon$ is sufficiently large (by
Poincar\'{e}'s inequality, this is certainly true when considered in
a fixed interval with Dirichlet boundary conditions). Consequently,
since the eigenvalues are smooth functions of $\varepsilon$ (by
virtue of their simplicity \cite[Th. 3.1, p. 482]{chowHale}, see
also \cite{malchiodi-fife}), we expect that there exists a sequence
$\{\varepsilon_i\}$ with
$\varepsilon_1>\varepsilon_2>\cdots>\varepsilon_i\to 0$ as $i\to
\infty$ such that, for each $\varepsilon=\varepsilon_i$, zero is an
eigenvalue of the linearized operator described in
(\ref{eqopenev1}). This suggests that the aforementioned local
bifurcation of solutions of $(\ref{eqground})_-$, from the trivial
branch, takes place at each $\varepsilon=\varepsilon_i$. We further
expect that, using global bifurcation techniques \cite{rabinowitz}
(see also \cite{karachalios}), one can show that these solution
branches reach, as $\varepsilon\to 0$, the layered solutions of
$(\ref{eqground})_-$  that we discussed previously. (We point out
that solutions belonging to the $i$-th branch have exactly $i-1$
zeros). We note that analogous eigenvalues of the form
(\ref{eqopenev2}), with the obvious modifications, also exist in the
multi-dimensional case, and existence of many solutions for the
nonlinear problem may follow by adapting Theorem 10.22 in
\cite{malchiodicambridge}. Moreover, in the ``flat'' case (motivated
from a definition in \cite{byeon}, see also \cite{kurata}), where
the potential $W$ attains its minimum value over a domain
$\Omega_0$, we expect that the multi-dimensional operator,
corresponding to (\ref{eqopenev1}), has eigenvalues of the form
\[
W(0)-\lambda+\mu_i(\Omega_0)\varepsilon^2+o(\varepsilon^2)\ \
\textrm{as}\ \varepsilon\to 0,
\]
where $\{\mu_i(\Omega_0)\}$ are the Dirichlet eigenvalues of
$-\Delta$ in $\Omega_0$. We believe that the eigenfunctions
associated to the above eigenvalues and the corresponding
eigenfunctions of the Laplacian share the topology of their level
sets, as in \cite{felmermayorgaNonl}. In any case, motivated from
results in \cite{alamaTarantelo}, we believe that the existence and
multiplicity of solutions to $(\ref{eqground})_-$, is strongly
associated to the number of negative eigenvalues (counting
multiplicities) of  the corresponding linearized operator about the
trivial solution.
%
 In the radially symmetric case, the topological approaches of
\cite{grillakisnodal}, \cite{selemNodal}, for constructing nodal
standing wave solutions of the focusing NLS, should also be
applicable to the defocusing case with a trapping potential, see
also a related remark in \cite{margetis2012}.

It would also be interesting if one can find an asymptotic
expansion, as $\varepsilon\to 0$, of the energy $E_1$ of the first
excited state (with least energy), as we did for the energy $E_g$ of
the ground state in Theorem \ref{thmmain}. The difference $E_1-E_g$
is of importance since it represents the ``excitation energy''
required to reach the first excited state from the ground state; it
thus determines in some sense the stability of the ground state. (In
the case of a convex bounded domain with Dirichlet boundary
conditions, with the obvious modifications, this would provide
evidence  on the validity of a ``nonlinear fundamental gap
conjecture'', see \cite{andrews}, \cite{baoACMAC}).

If in addition the potential $W$ is assumed to be even with respect
to the coordinate axis, we have observed that one can construct a
sign changing solution of $(\ref{eqground})_-$, whose nodal set is
the union of the coordinate axis, using the following strategy:
Firstly, by minimizing the functional $\mathcal{G}_-$, described in
(\ref{eqmountain}), over $\eta\in W_0^{1,2}(\mathbb{R}^2_+)$ such
that $W\eta^2 \in L^1(\mathbb{R}^2_+)$, where $\mathbb{R}^2_+\equiv
\{\textbf{y}=(\textbf{y}_1,\textbf{y}_2)\ :\ \textbf{y}_1>0,\
\textbf{y}_2>0 \}$, for small $\varepsilon>0$, we obtain a positive
solution in $\mathbb{R}^2_+$ of the equation in $(\ref{eqground})_-$
which is zero on the coordinate axis and approaches zero as
$|\textbf{y}|\to \infty$ (we can see that the minimizer is
nontrivial, if $\varepsilon$ is small, by adapting Example 5.11 in
\cite{malchiodicambridge} or Lemma 2.1 in \cite{delPinoscrew}). A
solution $\textbf{u}_2$ defined in the entire space is then obtained
using odd reflections through the lines $\textbf{y}_1=0$ and
$\textbf{y}_2=0$. The function $\textbf{u}_2$ is a solution of
$(\ref{eqground})_-$, whose 0-level set is the union of the two
axis. Our construction parallels that of the well known \emph{saddle
solution} of the Allen-Cahn equation (\ref{eqAllenCahn}), see
\cite{saddleFife}. The problem of existence and qualitative
properties of saddle type solutions for the Allen-Cahn equation (not
necessarily in two dimensions) has received a considerable amount of
attention in recent years, see \cite{saddleCabreJEMS1},
\cite{saddleCabreCPDE2}, \cite{saddlecabre3solo},
\cite{saddleKowal}. We wonder if an analogous study can be conducted
for the saddle type solutions of $(\ref{eqground})_-$ that we just
described. Can one rigorously verify the formal prediction that
\[\textbf{u}_2\to
\textrm{sign}\{\textbf{y}_1\textbf{y}_2\}\sqrt{(\lambda-W)^+},\] say
in $L^2(\mathbb{R}^2)$, as $\varepsilon\to 0$? The finer structure
at the  \emph{junction points} on the axis, where $W=\lambda$, may
be demonstrated by a solution $v$ of the following elliptic problem:
\begin{equation}\label{eqjunction}
\left\{\begin{array}{l}
  v_{xx}+v_{zz}-(x+v^2)v=0,\ \ x\in\mathbb{R},\ z>0, \\
    \\
  v-\sqrt{-x}\to 0\ \ \textrm{as}\ \ x\to -\infty;\ \ v\to 0\ \
\textrm{as}\ \ x\to \infty, \\
   \\
  v=0\ \ \textrm{if}\ \ z=0;\ v-V(x)\to 0\ \ \textrm{as}\ \ z\to \infty,
\end{array}\right.
\end{equation}
where $V$ denotes the Hastings-McLeod solution as usual, which seems
to be of independent interest. The above can be generalized to the
case of arbitrary even space dimensions.
 Let us also note that, if the potential trap $W$ is radial and two-dimensional, our construction can
easily be generalized to obtain solutions $\textbf{u}_k$ of
$(\ref{eqground})_-$ with $N=2$, for small $\varepsilon>0$, whose
zero level set has the symmetry of a regular $2k$-polygon and
consists of $k$ straight lines passing through the origin (see
\cite{saddlealessio} for the corresponding solutions of
(\ref{eqAllenCahn})).

In the case where the potential $W$ is, say,  two-dimensional and
symmetric with respect to the coordinate axis (as in the above
paragraph) but the equation of $(\ref{eqground})_-$ is posed in
$\mathbb{R}^3$, motivated from a definition in \cite{fuscoTrans}, we
can also consider ``tick'' saddle solutions: As before, minimizing
the functional $\mathcal{G}_-$
 over $\eta \in W^{1,2}_0(\Omega)$
such that $W\eta^2 \in L^1(\Omega)$, where $\Omega\equiv
\{\textbf{y}_1>0,\ \textbf{y}_2> 0,\ -D<\textbf{y}_3<D) \}$, $D>0$,
for $\varepsilon<\varepsilon(D)$ sufficiently small, yields a
positive solution of the equation which vanishes on $\partial
\Omega=\{\textbf{y}_1=0,\ \textbf{y}_2=0,\ \textbf{y}_3=\pm D\}$. By
odd reflection with respect to $\{\textbf{y}_1=0,\ \textbf{y}_2=0,\
\textbf{y}_3 \in(-D,D)\}$, and then with respect to the planes
$\textbf{y}_3 = (2k + 1)D,\ k \in \mathbb{Z}$, that solution can be
extended to the whole of $\mathbb{R}^3$, yielding an entire solution
of the equation in $(\ref{eqground})_-$ which has a saddle structure
on each plane $\textbf{y}_3= \textrm{constant}$ and is periodic of
period $4D$ in the $\textbf{y}_3$ variable.

Very recently, del Pino, Musso, and Pacard \cite{delPinoscrew}
studied entire solutions of the Allen-Cahn equation
(\ref{eqAllenCahn}) which are defined in 3--dimensional Euclidean
space and which are invariant under screw-motion. In particular,
their nodal set is a helico\"{i}d of $\mathbb{R}^3$. We believe
that, for sufficiently small $\varepsilon>0$, similar solutions
exist for the 3-dimensional defocusing Gross-Pitaevskii equation in
$(\ref{eqground})_-$ with a 2-dimensional radial potential $W$
($\lambda, \ W$ satisfying our usual assumptions). What is the
asymptotic behavior of these solutions as $\varepsilon\to 0$? Can
some results of \cite{delPinoscrew} be generalized in our context?


We believe that, if the potential $W$ is restricted to the radial
class, the approach of the current paper can also be applied to the
study of the $\varepsilon\to 0$ limiting behavior of \emph{vortex
solutions} of the NLS equation $(\ref{eqNLS})_-$, see \cite{pego} or
\cite{seiringer}, namely solutions of the form
\begin{equation}\label{eqvortex}\left\{\begin{array}{l}
    u_n(\textbf{y},t)=U_n(r)e^{i(n\theta-\lambda t/\varepsilon)},\ \
n=\pm1,\pm2,\cdots,
 \\
      \\
    U_n(0)=0, \ \ \ U_n(\infty)=0,
  \end{array}
  \right.
\end{equation}
where $(r,\theta)$ denote the polar coordinates in $\mathbb{R}^2$.
Hopefully, the obtained estimates could be used to prove the,
indicated by numerical evidence \cite{kevrekidis-pelinovsky} (for
the case of the model harmonic potential), orbital stability of
$u_1$ in the time evolution of the Gross-Pitaevskii equation, and
thus answering the question raised in the end of the recent paper
\cite{pelinovskyBifvortex}.

Non-degeneracy conditions of the form (\ref{eqVnormal}) are common
in the study of transition layered solutions of elliptic equations
with bistable nonlinearity, see \cite{fifegreenlee},
\cite{weicluster}. In that context, the surface
$\partial\mathcal{D}_0$ represents the interface of the layer. It
turns out that, in some cases, the aforementioned conditions can be
removed completely (see \cite{danceryanCVPDE}, \cite{delpinoCPDE}).
In particular, the interface may be non-smooth or intersect the
boundary of the domain. Motivated from this, we believe that one can
show that $\eta_\varepsilon\to \sqrt{A^+}$ uniformly in
$\mathbb{R}^2$, or at least in compact subsets of
$\mathbb{R}^2\backslash
\partial\mathcal{D}_0$, as $\varepsilon\to 0$, without assuming
condition (\ref{eqVnormal}). (Here $\eta_\varepsilon$ denotes the
minimizer of $G_\varepsilon$  or the ground state of
$(\ref{eqNLS})_-$). In this regard, we refer to \cite[Prop.
3.16]{cantell} for a related result (for $(\ref{eqground})_-$ with
$q=2$).

Is there a ``$\Gamma$-Convergence'' theory \cite{kohnstrenberg} for
(\ref{eqGminimizer}), relating local minimizers of the limit
functional (\ref{eqGminimizer000})  to local minimizers of
(\ref{eqGminimizer}), as $\varepsilon\to 0$?

We wonder if, besides  the one-dimensional  profile $V(x)$, there is
a (genuine) two-dimensional one $v(x,z)$ that could be used in
(\ref{eqvsz}). In view of (\ref{equs-uss-uz-uzz}), (\ref{eqS(v)=0}),
(\ref{eqB2}), and the matching conditions with $\sqrt{A^+}$, the
profile $v$ should satisfy
\[
\left\{\begin{array}{lll}
  \beta^{-2}v_{zz}+v_{xx}-v(v^2+x)=0,  & (x,z)\in \mathbb{R}^2, \\
   &   \\
 v-\sqrt{-x}\to 0 \ \textrm{as}\ x\to
  -\infty,  &  v\to 0\ \textrm{as}\  x\to \infty,
\end{array}
\right.
\]
with $v$ being $\ell_\varepsilon/\varepsilon^\frac{2}{3}$-periodic
in $z$.
As in \cite{del pino cpam} (see also \cite{karalisourdisresonance}),
after a simple transformation of the $z$ independent variable (the
$x$ variable remains unchanged), abusing notation, we are led to the
problem:
\begin{equation}\label{eqdegiorgi1}
  v_{zz}+v_{xx}-v(v^2+x)=0,  \ \  (x,z)\in \mathbb{R}^2,
  \end{equation}
  \begin{equation}\label{eqdegiorgi1+}
 v-\sqrt{-x}\to 0 \ \textrm{as}\ x\to
  -\infty,  \ \  v\to 0\ \textrm{as}\  x\to \infty,
\end{equation}
with $v$ being
$\hat{\ell}_\varepsilon/\varepsilon^\frac{2}{3}$-periodic in $z$,
where
$\hat{\ell}_\varepsilon=\int_{0}^{\ell_\varepsilon}\beta_\varepsilon(\theta)d\theta$.
Uniqueness of positive solutions to the above problem does not seem
to follow from the approach of Brezis and Oswald, as in
\cite{ignat}, since solutions are unbounded (compare with
(\ref{eqjunction})).  In fact, we believe that entire solutions of
the equation in (\ref{eqdegiorgi1}) should satisfy the growth
estimate $v(x,z)= O(|x|^{\frac{1}{2}})$. Moreover, it is not clear
how to adapt the uniqueness result of Brezis \cite{brezis}. On the
other hand, motivated from (\ref{eqJ4}), it is natural to seek
solutions such that
\begin{equation}\label{eqdegiorgi2}
v_x<0,\ \ (x,z)\in \mathbb{R}^2.
\end{equation}
It is irresistible to compare problem (\ref{eqdegiorgi1}),
(\ref{eqdegiorgi2}) with the famous De Giorgi conjecture for the
Allen-Cahn equation (\ref{eqAllenCahn}), see for instance
\cite{farina}, \cite{gui}. In this regard, it is interesting to
investigate whether there are genuine two-dimensional solutions $v$
of problem (\ref{eqdegiorgi1}), (\ref{eqdegiorgi2}) or not. Note
that solutions of the latter problem could be unbounded and, in
particular, so could be $v_z$ (see \cite{berestycki-wei2012,
berestycki-wei2012-2} where a similar difficulty arises). We point
out that the space dimension usually plays a very important role in
these type of problems. Another direction could be to investigate
the same question for stable solutions of (\ref{eqdegiorgi1}), in
the sense of (\ref{eqsofter}),
 see
also a remark in pg. 79 of the review article \cite{dancer-fields}.
Actually, using the method of \cite{farina}, one can show that any
solution of (\ref{eqdegiorgi1}), (\ref{eqdegiorgi2}) is stable. A
variation of these questions could be to consider problem
(\ref{eqdegiorgi1})-(\ref{eqdegiorgi1+}), with the asymptotic
behavior  in (\ref{eqdegiorgi1+}) being uniform  in $z\in
\mathbb{R}$, along the lines of the so called Gibbons conjecture
(see also \cite{berestycki-wei2012}). We remark that in this case,
as in \cite{gui}, the method of moving planes \cite{gidas} can be
applied to show that (\ref{eqdegiorgi2}) holds.

Many recent papers deal with the study of semiclassical ground
states for the focusing $(\ref{eqground})_+$, where the potential
$W(\textbf{y})-\lambda$ is positive but decays to zero, as
$|\textbf{y}|\to \infty$, at most like $|\textbf{y}|^{-2}$ (see for
instance \cite{decaypot} and the references therein). Can one study
the defocusing case under analogous conditions? For $\varepsilon$
fixed, a related existence result may be found in \cite{brownJmaa}.

Suppose, for simplicity purposes, that $W$ is an even, double-well,
one-dimensional potential (for example as in \cite{doublewell} or
\cite{kirrCommPhysics}), say $W(\textbf{y})=(\textbf{y}^2-1)^2$.
What happens in the ``degenerate case'' when $\lambda$ equals the
local maximum of $W$? Assuming that $W''(0)<0$, suitably blowing up
at the origin, we expect that the fine behavior of solutions, as
$\varepsilon\to 0$, near the origin should be determined by a
solution of the problem:
\[
\left\{
\begin{array}{l}
  v''=v(v^2+\frac{W''(0)}{2}x^2)=0,\ \ x\in\mathbb{R}, \\
    \\
  v-\sqrt{\frac{-W''(0)}{2}}|x|\to 0\ \ \textrm{as}\ \ |x|\to \infty.
\end{array}
\right.
\]
Notice the similarities of the above problem with
(\ref{eqkaralisoudisPain}) below. Note also that in the case of a
\emph{symmetric} double-well potential (for any $\lambda$) formulas
(\ref{eqopenev2}) do not hold due to \emph{tunneling} effects, see
for instance \cite{pelinovsky}, \cite{hislop}.

Finally, we believe that similar studies can be conducted in the
case of the ``exterior'' problem, where $\lambda>W$ outside of a
bounded domain and $\lambda<W$ in its interior. It is natural to
assume that  $\lambda-W\to c>0$ as $|y|\to \infty$, and consider the
Gross-Pitaevskii equation $(\ref{eqNLS})_-$ with boundary conditions
$|u(\textbf{y},t)|\to c^\frac{1}{q-1}$ as $|\textbf{y}|\to \infty$.
The approach of \cite{brezis-oswald} does not yield uniqueness of
bounded ground states, namely solutions of the equation in
$(\ref{eqground})_-$ coupled with the aforementioned boundary
conditions, if $N\ge 3$ (compare with Remark \ref{remuniq}), and one
has to apply a sophisticated ``squeezing'' argument (see
\cite{duma}). Let us mention that the stability of standing wave
solutions of Gross-Pitaevskii equations, considered with nonzero
boundary conditions at infinity, is a very active field of current
research, see for instance \cite{bethuelsaut}.

\begin{rem}\label{remkening}
It is worthwhile to mention that if $0\leq W(\textbf{y})-\lambda\leq
 c(1+|\textbf{y}|)^{2+d},\ \textbf{y}\in \mathbb{R}^N$, $N\geq 3$, for some positive constants $c,d$,
  then $(\ref{eqground})_-$ has infinitely many bounded solutions  with positive lower bounds (see \cite{kenig}).
\end{rem}

\appendix\section{A-priori estimates for the linearized operator based on the non-degeneracy of the inner profile}\label{ap2}
Here we will provide an alternative, more natural, proof of the
important Proposition \ref{proL} that does not require knowledge of
lower bound (\ref{eqvpeli}), whose proof is rather technical (recall
Remark \ref{rempotential}), but instead relies merely on the
non-degeneracy of the Hastings-McLeod solution $V$. This proof has
the flexibility to deal with problems where the corresponding inner
profile $V$ is non-degenerate but the corresponding lower bound
(\ref{eqvpeli}) may be hard to establish or fails (see Remark
\ref{rempotentialhalfapenc} below for an example where the latter
case occurs). The latter situation certainly occurs when trying to
construct unstable solutions (with respect to the parabolic
dynamics) in related problems, see \cite{karalisourdisradial,
karalisourdisresonance}. Let us also point out that it is not clear
to us how to generalize the last part of the proof of
(\ref{eqvpeli}) in \cite{pelinovsky} to the case of arbitrary power
nonlinearity, as in Proposition \ref{proentire} below (see also
Remark \ref{remexponent} below).

\texttt{PROOF OF PROPOSITION \ref{proL}}: Observe that it suffices
to show the following a-priori estimate: There exists a constant $C$
such that if $\varepsilon$ is sufficiently small, $\varphi\in
\mathcal{X}\cap C^{2+\alpha}(\mathbb{R}^2)$, and $f\in
\mathcal{X}\cap C^{\alpha}(\mathbb{R}^2),\ 0<\alpha<1,$ satisfy
\begin{equation}\label{eqaprAp--}
\mathcal{L}(\varphi)=f,
\end{equation}
then
\begin{equation}\label{eqaprAp}
\|\varphi\|_{L^\infty(\mathbb{R}^2)}\leq
C\|f\|_{L^\infty(\mathbb{R}^2)}.
\end{equation}
To this end, as in the one-dimensional related problem treated in
\cite[Prop. 5.2]{sourdis-fife}, we will argue by contradiction. We
remark that this indirect method has been used extensively in the
study of elliptic singular perturbation problems involving
transition and spike layers, see \cite{malchiodibook}.

Firstly, note that without knowledge of the validity of
(\ref{eqvpeli}), relation (\ref{equappotential}) would be
\begin{equation}\label{equappotentialappendix}
3u_{ap}^2-a(\varepsilon^\frac{2}{3}y)\geq \left\{
\begin{array}{ll}
  c \varepsilon^\frac{2}{3}|x|, &  \textrm{if} \ \ L\leq|x|\leq \delta \varepsilon^{-\frac{2}{3}}, \\
    &   \\
  c+c|\varepsilon^\frac{2}{3}y|^p,  & \textrm{otherwise},
\end{array}\right.
\end{equation}
for small $\varepsilon>0$, having increased the value of $L$ if
necessary.

 Suppose now that there exist sequences $\varepsilon_n\to
0,\ \varphi_n\in \mathcal{X}\cap C^{2+\alpha}(\mathbb{R}^2)$,
$f_n\in \mathcal{X}\cap C^{\alpha}(\mathbb{R}^2)$ such that
\begin{equation}\label{eqapBEQ}
\left\{
\begin{array}{c}
  \mathcal{L}(\varphi_n)=\Delta \varphi_n-\varepsilon_n^{-\frac{2}{3}}\left( 3u_{ap}^2-a(\varepsilon_n^\frac{2}{3}y)\right)\varphi_n=f_n \\
   \\
  \|\varphi_n\|_{L^\infty(\mathbb{R}^2)}=1\ \ \textrm{and}\ \
  \|f_n\|_{L^\infty(\mathbb{R}^2)}\to 0.
\end{array}
\right.
\end{equation}
Keeping in mind that $\varphi_n\to 0$ as $|y|\to \infty$, we may
assume that there exist $y_n\in \mathbb{R}^2$ such that, without
loss of generality, we have
\begin{equation}\label{eqapBEQ1}
\varphi_n(y_n)=1,\ \ \nabla\varphi_n(y_n)=0,\ \ \Delta
\varphi_n(y_n)\leq 0,\ \ n\geq 1.
\end{equation}
From (\ref{equappotentialappendix})--(\ref{eqapBEQ}), we obtain that
\begin{equation}\label{eqapdif}\textrm{dist}(y_n,\tilde{\Gamma}_{\varepsilon_n})\leq
C,\ \ n\geq 1,
\end{equation}for some  (generic) constant $C$  independent of
$n\geq 1$.
 Thus, abusing notation, we can write
$y_n=\left(\beta^{-1}(\varepsilon_n^\frac{2}{3}z_n)x_n,z_n \right)$
with $|x_n|\leq C$, $z_n\in
[0,\varepsilon_n^{-\frac{2}{3}}\ell_{\varepsilon_n})$. Therefore,
passing to a subsequence, we may assume that
\begin{equation}\label{eqapBEQ2}
x_n\to x_*\ \ \textrm{and}\ \ \varepsilon_n^\frac{2}{3}z_n\to z_*\in
[0,{\ell}_0].
\end{equation}
 Recalling (\ref{eqLaplace}), (\ref{eqvsz}), in terms of coordinates
$(x,z)$, the equation in (\ref{eqapBEQ}) takes the form
\begin{equation}\label{eqapregion}
(\varphi_n)_{zz}+\beta^2(\varepsilon_n^\frac{2}{3}z)(\varphi_n)_{xx}+\tilde{B}_1(\varphi_n)-\varepsilon_n^{-\frac{2}{3}}\left(
3u_{ap}^2-a(\varepsilon_n^\frac{2}{3}y)\right)\varphi_n=f_n,
\end{equation}
in the neighborhood of the curve $\tilde{\Gamma}_{\varepsilon_n}$
described by $\left\{|x|\leq \delta_0\varepsilon_n^{-\frac{2}{3}},\
\ z\in
\left[0,\varepsilon_n^{-\frac{2}{3}}\ell_{\varepsilon_n}\right)
\right\}$, where $\tilde{B}_1$ is the differential operator:
\begin{equation}\label{eqB6}
\tilde{B}_1(\varphi)=\varepsilon^\frac{4}{3}\beta''\beta^{-1}x\varphi_x+\varepsilon^\frac{4}{3}
(\beta')^2\beta^{-2}x^2\varphi_{xx}+2\varepsilon^\frac{2}{3}\beta'\beta^{-1}x\varphi_{xz}+B_1(\varphi),
\end{equation}
and $B_1$ is the differential operator in (\ref{eqB1}) where
derivatives are expressed in terms of formulas
(\ref{equs-uss-uz-uzz}) and $s$ replaced by $\beta^{-1}x$. By
(\ref{eqvasympto}), the first relation in
(\ref{equout-uingradients}), and working as in (\ref{eq3uap1}), we
obtain that
\begin{equation}\label{eqB7}
3u_{ap}^2-a(\varepsilon_n^\frac{2}{3}y)=\varepsilon_n^\frac{2}{3}\beta^2(\varepsilon_n^\frac{2}{3}z)\left(
3V^2(x)+x\right)+\mathcal{O}(\varepsilon_n^\frac{4}{3})(x^2+1)
\end{equation}
uniformly in the region described below (\ref{eqapregion}), as $n\to
\infty$. Making use of (\ref{eqapBEQ})--(\ref{eqB7}), and a standard
compactness argument, passing to a subsequence, we find that
\[
\varphi_n\to \phi\ \ \textrm{in} \ \ C^2_{loc}(\mathbb{R}^2),
\]
where $\phi$ satisfies
\begin{equation}\label{eqphilimit}
\phi_{zz}+\beta_0^2(z_*)\phi_{xx}-\beta_0^2(z_*)\left(
3V^2(x)+x\right)\phi=0,\ \ (x,z)\in \mathbb{R}^2,
\end{equation}
\begin{equation}\label{eqcontra}
\|\phi\|_{L^\infty(\mathbb{R}^2)}= 1 \ \ \
\left(\phi(x_*,z_*)=1\right).
\end{equation}
Since
\begin{equation}\label{eqB7+-}
3V^2(x)+x\to \infty\  \textrm{as}\  x\to \pm \infty,
\end{equation}
a standard barrier argument, as in (\ref{eqbarrier}), and elliptic
estimates \cite{Gilbarg-Trudinger}, yield that there exists a
constant $C$ such that
\begin{equation}\label{eqB7+}
\left|\nabla\phi(x,z)\right|+|\phi(x,z)|\leq Ce^{-|x|},\ \ (x,z)\in
\mathbb{R}^2,
\end{equation}
(see also Lemma 7.3 in \cite{delpinotoda}). Let $(\mu_1,\psi_1)$
denote the principal eigenvalue-eigenfunction pair of
\begin{equation}\label{eqtoda}
-\mathcal{M}(\psi)=-\psi''+\left(3V^2(x)+x\right)\psi=\mu \psi,\ \
\psi(\pm \infty)=0.
\end{equation}
Without loss of generality, we may assume that $\psi_1$ is positive.
 Furthermore, we have
 \begin{equation}\label{eqmu1}\mu_1>0,
 \end{equation} as testing against $V_x<0$ readily
shows (see \cite{sourdis-fife}). (By virtue of (\ref{eqB7+-}) and
Theorem 10.7 in \cite{hislop}, the spectrum of $-\mathcal{M}$, in
$L^2(\mathbb{R})$, consists of simple eigenvalues
$\mu_1<\mu_2<\cdots$ with $\mu_i\to\infty$). Now let
\[
\Phi(z)=\int_{-\infty}^{\infty}\phi(x,z)\psi_1(x)dx,
\]
where $\phi$ is as in (\ref{eqphilimit}). From (\ref{eqphilimit}),
(\ref{eqB7+}), and (\ref{eqtoda}) with $\psi=\psi_1,\ \mu=\mu_1$, we
calculate that
\[
\begin{array}{lll}
  \Phi'' & = & \int_{-\infty}^{\infty}\phi_{zz}(x,z)\psi_1(x)dx \\
    &   &   \\
    &  = & \beta^2_0(z_*)\int_{-\infty}^{\infty}\left[-\phi_{xx}+\left(3V^2(x)+x\right)\phi\right]\psi_1(x)dx \\
    &   &   \\
    &  = & \beta^2_0(z_*)\int_{-\infty}^{\infty}\phi\left[-\psi_1''+\left(3V^2(x)+x\right)\psi_1\right]dx \\
    &   &   \\
    &  = & \mu_1 \beta_0^2(z_*)\Phi,
\end{array}
\]
and
\[
|\Phi|\leq C,\ \ z\in \mathbb{R}.
\]
 From (\ref{eqmu1}), and the above two relations, it follows at once that $\Phi$ is identically zero, which contradicts the previous
relation (\ref{eqcontra}). Consequently, we have established the
validity of the desired a-priori estimate (\ref{eqaprAp}).

The proof of the proposition is complete. \ \ \ \ $\Box$

\begin{rem}\label{remfuscoextension}
By adapting  Lemma 5.3 of \cite{sourdis-fife}, we can show that
relation (\ref{eqaprAp--}) implies that
$\|\varphi\|_{L^2(\mathbb{R}^2)}\leq C\|f\|_{L^2(\mathbb{R}^2)}$ for
some constant that is independent of $\varphi,\ f$ and small
$\varepsilon>0.$ In fact, as in \cite[Thm. 1.2]{fusco}, we expect
that more general a-priori estimates of the form
$\|\varphi\|_{L^p(\mathbb{R}^2)}\leq C_{p,q}
\varepsilon^{\alpha_{p,q}}\|f\|_{L^q(\mathbb{R}^2)}$ hold
true.\end{rem}
\section{Around the Hastings-McLeod solution of the Painlev\'{e}-II
equation}\label{appenpainleve} In this appendix we will provide a
new proof of the existence  of the Hastings-McLeod solution $V$ of
the Painlev\'{e}-II equation (\ref{eqpainleve}). Moreover, we will
establish various qualitative properties of the solution that are
required for the singular perturbation analysis. In contrast to the
original proof of Hastings and Mcleod \cite{hastings} (see also
\cite{hastingsbook}), where a shooting argument was employed, here
we will use an upper and lower solution argument, which in principle
is not restricted to ODE problems. Even though such an approach was
successfully applied to this problem recently in \cite{alikakos},
and very recently in \cite{weipitaevskii}, in our opinion our
construction is more flexible and intuitive. The main advantage of
our proof, compared to those of the latter references, is that, in
the process, we also establish existence and uniqueness of a
solution of problem (\ref{eqvdirichlet}), which seems to be a new
and useful result (recall Remark \ref{remdirichlet}). Although a
sizable literature has been devoted to the study of the Painlev\'{e}
equation (see \cite{fokas,itsLNM,painleve-levi}), we understand that
the solution of this problem was not previously known.

Our choice of lower--solution is motivated from \cite{dancer-lazer}
where, in particular, the authors treat the problem
\begin{equation}\label{eqdancerPainle}
u''=|u|^p-x,\ x>0,\ u(0)=0,\ x^{-\frac{1}{p}}u(x)\to 1\ \textrm{as}\
x\to \infty,
\end{equation}
where $p>1$ (recall the discussion in the third part of Subsection
\ref{secmotivation}), see also \cite[Sec. 3.2]{hastingsbook} and
\cite{holmespainleve} for the case $p=2$ which is the Painlev\'{e}-I
equation. On the other hand, our choice of upper--solution is
motivated from \cite{schecter-sourdis} where, in particular, the
authors treat the problem
\begin{equation}\label{eqkaralisoudisPain}u''=u^2-x^2,\ x\in
\mathbb{R},\ u(x)-|x|\to 0\ \textrm{as}\ |x|\to \infty,
\end{equation} see also Remark \ref{remlastprofile}
below.

The notation in this appendix is independent of the rest of the
paper.

As in \cite{hastings}, see also \cite[pg. 200]{BenderO}, we will
prove the following more general result:

\begin{pro}\label{proentire}
Given $p>1$, there exists a unique nonnegative solution $U$ of
\begin{equation}\label{eqEqappenC}
-u''-xu+|u|^{p}u=0,\ x\in \mathbb{R},
\end{equation}
such that
\begin{equation}\label{eqEqappenCbdry}
u(x)\to 0\ \textrm{as}\ x\to -\infty;\ u(x)-x^\frac{1}{p}\to 0\
\textrm{as}\ x\to \infty. \end{equation} Furthermore, we have that
$U'>0$ in $\mathbb{R}$, and
\begin{equation}\label{eqsuperexpappenC}
U(x)=\mathcal{O}\left(|x|^{-\frac{1}{4}}e^{-\frac{2}{3}|x|^\frac{3}{2}}
\right) \textrm{as}\ x\to -\infty; \ \
U(x)-x^\frac{1}{p}=\mathcal{O}\left(x^{\frac{1-3p}{p}} \right) \
\textrm{as}\ x\to \infty.
\end{equation}

The solution $U$ is non-degenerate in the sense that there are no
nontrivial  bounded solutions of the problem
\[
\phi''-\left[(p+1)U^p-x\right]\phi=0,\  \ x\in \mathbb{R}.
\]
\end{pro}

Note that (\ref{eqpanintro}) falls in the above class of problems by
means of the transformation $x\to -x$.

\begin{rem}\label{remexponent}
The results of this appendix can be used in extending the results of
the current paper, and treat  the defocusing $(\ref{eqground})_-$
with arbitrary nonlinearity exponent $q>2$. In particular, the
considered model of the latter equation,
 with
nonlinearity exponents $7/3 < q < 3$ (in one space dimension and the
model harmonic potential), is particularly relevant to the physics
of BEC-BCS (Bardeen-Cooper-Schrieffer) transition in ultracold
Fermi gases (see \cite{yankevrekidis}).
\end{rem}

As a stepping stone towards the proof of the above proposition, we
will first prove the following result which, as we have already
mentioned, is of interest in its own right. In particular, the
solution $U_+$ below will form the basis for our construction of a
lower-solution to problem
(\ref{eqEqappenC})--(\ref{eqEqappenCbdry}).

\begin{pro}\label{prophalf}
Given $p>1$, there exists a unique solution $U_+$ of the problem
\begin{equation}\label{eqhalf}
\left\{\begin{array}{ll}
  -u''=xu-u^{p+1}, & x>0, \\
   &  \\
  u(0)=0, &  \\
   &  \\
  0\leq u(x)\leq x^\frac{1}{p}, &x\geq0.
\end{array}
\right.
\end{equation}
Furthermore, we have that $U_+'(x)>0,$ $x\geq 0$, and
\begin{equation}\label{eqhalfestim}
U_+(x)-x^\frac{1}{p}=\mathcal{O}\left(x^{\frac{1-3p}{p}} \right)\
\textrm{as}\ x\to \infty.
\end{equation}
\end{pro}
\begin{proof}
It is easy to check that $x^\frac{1}{p}$ is an upper-solution of
(\ref{eqhalf}), while $\delta \chi_{[K,K+\pi]}\sin(x-K)$ is a (weak)
lower-solution provided that $K\geq 2$  and $0<\delta\leq 1$ (here
$\chi$ denotes the characteristic function). (We refer the reader to
\cite{Berestyckilion} for more information on piecewise smooth weak
upper/lower-solutions). From now on we fix such a $K$, say $K=\pi$.
Then, by a well known theorem \cite{Berestyckilion}, for every
$0<\delta\leq 1$, there exist solutions $u_1$, $u_2$ of
(\ref{eqhalf}) such that $\delta \chi_{[\pi,2\pi]}\sin(x-\pi)\leq
u_1\leq u_2\leq x^\frac{1}{p}$, with the property that any solution
 of (\ref{eqhalf}) such that $\delta
\chi_{[\pi,2\pi]}\sin(x-\pi)\leq u \leq x^\frac{1}{p}$, satisfies  $
u_1\leq u \leq u_2$. (Note that $u_1,\ u_2$ depend on $\delta$).

For any nontrivial solution of (\ref{eqhalf}) we have
$u''=u(u^p-x)\leq 0,\ x>0$. Thus $u'$ is non-increasing. As a
result, $u'\to a$ as $x \to \infty$. Here $a$ may be minus infinity.
We claim that $a=0$. In fact, if $a<0$, then $u(x)<0$ for $x$ large,
which is a contradiction to $u\geq 0$. If $a>0$, then $u(x)\geq
\frac{a}{2}x$ for $x>0$ large, which is a contradiction to $u(x)\leq
x^\frac{1}{p}$. Thus $a=0$. Consequently $u'\geq 0$. Actually, since
$u'$ cannot be constant over a nontrivial interval, we find that
$u'$ is decreasing and thus $u'>0$.

In order to proceed further, we will show that problem
(\ref{eqhalf}) has a unique solution. Inspired by the work by H.
Brezis and L. Oswald in \cite{brezis-oswald} (see also \cite[Sec.
8.5.2]{pacard-riviere}), we will study the quotient of two solutions
(for a different approach see Remark \ref{remSerrin} below). Suppose
that (\ref{eqhalf}) has two solutions $\tilde{u}_1, \tilde{u}_2$.
Then, we can find $0<\delta\leq 1$ small such that $\delta
\chi_{[\pi,2\pi]}\sin(x-\pi)\leq \tilde{u}_i\leq x^\frac{1}{p}$,
$i=1,2$. Thus, from the previous discussion, we infer that $u_1\leq
\tilde{u}_i \leq u_2,\ i=1,2$. We only need to prove that $u_1=u_2$.
From (\ref{eqhalf}), we find that
\[
\frac{u_2''}{u_2}\geq \frac{u_1''}{u_1}.
\]
The above inequality implies that the function $u_2'u_1-u_2u_1'$ is
non-decreasing, and so
\[
(u_2'u_1-u_2u_1')(x)\geq (u_2'u_1-u_2u_1')(0)=0, \ x>0,
\]
which in turn implies that the function $\frac{u_2}{u_1}$ is
non-decreasing.
Therefore, we get
\[
\frac{u_2(x)}{u_1(x)}\geq \lim_{x\to
0^+}\frac{u_2(x)}{u_1(x)}=\frac{u_2'(0)}{u_1'(0)},\ x>0,
\]
by L'Hospital's rule (recall that $u_1'(0)>0$). From $u_2\geq u_1$
and $u_1(0)=u_2(0)=0$, we know that $u_2'(0)\geq u_1'(0)$. Suppose
that $u_2'(0)> u_1'(0)$. We have
\[
u_1\leq \frac{u_1'(0)}{u_2'(0)}u_2\leq
\frac{u_1'(0)}{u_2'(0)}x^\frac{1}{p},\ x>0.
\]
So
\[
u_1''=u_1(u_1^p-x)\leq u_1\left[\left(\frac{u_1'(0)}{u_2'(0)}
\right)^p-1 \right]x\leq -cx
\]
for $x>0$ large, and some constant $c>0$, since $u_1'(0)< u_2'(0)$
and $u_1'> 0$. It follows that $u_1'<0$ for $x>0$ large. This is a
contradiction to $u_1'(x)>0,\ x\geq 0$. We conclude that
$u_1'(0)=u_2'(0)$, which gives $u_1=u_2$.

Let $u$ denote the unique solution of (\ref{eqhalf}). Adapting an
argument from \cite{dancer-lazer}, we will show that
\begin{equation}\label{eqlim}
x^{-\frac{1}{p}}u(x)\to 1 \ \textrm{as} \ x\to \infty.
\end{equation}
From $u'(\infty)=0$, we obtain
\[
\int_{0}^{\infty}u(x)\left(x-u^p(x)\right)dx=-\int_{0}^{\infty}u''(x)dx=u'(0).
\]
Hence, we can choose $x_i\to \infty$ such that $
u(x_i)\left(x_i-u^p(x_i)\right)\to 0\ \textrm{as}\ i\to \infty$, and
recalling that $u'> 0$, we find that
\begin{equation}\label{eqxi}
x_i-u^p(x_i)\to 0\ \textrm{as}\ i\to \infty.
\end{equation}
Now, for any small $\theta>0$, we claim that
\[
u(x)\geq (1-\theta)x^\frac{1}{p},
\]
for $x>0$ large. Suppose that the claim is not true. Then, there are
$0<\theta<1$ and $\tilde{x}_i\to \infty$ such that
\begin{equation}\label{eqxitilda}
u(\tilde{x}_i)< (1-\theta)\tilde{x}_i^\frac{1}{p}.
\end{equation}
It is easy to check that there is a $T>0$ large such that, for
$x>T$, we have
\begin{equation}\label{eqsubdancer}
-\left((1-\theta)x^\frac{1}{p}
\right)''<x(1-\theta)x^\frac{1}{p}-\left((1-\theta)x^\frac{1}{p}
\right)^{p+1}.
\end{equation}
By (\ref{eqxi}), we can choose a constant $\bar{T}>T$, such that
$u(\bar{T})>(1-\theta)\bar{T}^\frac{1}{p}$. Define $v(x)=u(x)$ if $x
\in [0,\bar{T}]$;
$v(x)=\max\left(u(x),(1-\theta)x^\frac{1}{p}\right)$ if $x\in
[\bar{T},\infty)$. Then $v$ is continuous, $0\leq v\leq
x^\frac{1}{p}$, and is a (weak) lower-solution of (\ref{eqhalf}), in
view of (\ref{eqsubdancer}) and \cite{Berestyckilion}. As a result,
from \cite[Thm. 2.10]{niIndiana}, problem (\ref{eqhalf}) has a
solution $u^*$ with $v\leq u^*\leq x^\frac{1}{p}$. On the other
hand, since $v\geq u$, and $u\neq v$ (by (\ref{eqxitilda})), we find
that $u^*\neq u$. This contradicts the uniqueness of the solutions
of (\ref{eqhalf}).

It remains to show the validity of estimate (\ref{eqhalfestim}). We
have
\begin{equation}\label{equq}
u''=u\frac{\left(u^p-x
\right)}{u-x^\frac{1}{p}}(u-x^\frac{1}{p})=q(x)(u-x^\frac{1}{p}),\
x>0,
\end{equation}
with
\begin{equation}\label{eqqlim}
\frac{q(x)}{x} \to p \ \textrm{as}\ x\to \infty,
\end{equation}
by (\ref{eqlim}). (Note that, by the maximum principle, we get
$u(x)<x^\frac{1}{p},\ x>0$). Let
\[w=u-x^\frac{1}{p},\ x>0. \]
Then, from (\ref{equq}), we obtain that
\begin{equation}\label{eqqw}
-w''+q(x)w-\frac{1}{p}\left(\frac{1}{p}-1
\right)x^{\frac{1}{p}-2}=0,\ x>0.
\end{equation}
We claim that there exist $L,\ M>0$ sufficiently large such that,
for every $\delta \in (0,1)$, the function
\[
\underline{w}(x)=-Mx^{\frac{1}{p}-3}-\delta x^\frac{1}{p},
\]
is a lower-solution of (\ref{eqqw}) in $[L,\infty)$. Indeed, thanks
to (\ref{eqqlim}), for every $\delta \in (0,1)$, we find that
\[
\begin{array}{l}
  -\underline{w}''+q(x)\underline{w}-\frac{1}{p}\left(\frac{1}{p}-1
\right)x^{\frac{1}{p}-2}\leq \\
   \\
  M\left(\frac{1}{p}-3 \right)\left(\frac{1}{p}-4 \right)x^{\frac{1}{p}-5}+\delta \frac{1}{p}\left(\frac{1}{p}-1 \right)x^{\frac{1}{p}-2}
  -\frac{p}{2}Mx^{\frac{1}{p}-2}-\frac{p}{2}\delta x^{\frac{1}{p}+1}-\frac{1}{p}\left(\frac{1}{p}-1 \right)x^{\frac{1}{p}-2}\leq \\
   \\
  Mx^{\frac{1}{p}-5}\left[\left(\frac{1}{p}-3\right)\left(\frac{1}{p}-4\right)+\left(\frac{1}{M}-\frac{p}{2}\right)x^3
  \right]<0,
\end{array}
\]
for $x>L$,  provided $M>\frac{4}{p}$ and $L$ is sufficiently large
(independently of $M$). We fix such an $L>0$, and
choose $M>\frac{4}{p}$ large, depending only on $L$, such that
$\underline{w}(L)<w(L)$ for every $\delta \in (0,1)$. In view of
(\ref{eqlim}), for every $\delta \in (0,1)$, we have
\[
\underline{w}(x)-w(x)=-Mx^{\frac{1}{p}-3}-\delta
x^\frac{1}{p}-(u-x^\frac{1}{p})\leq x^\frac{1}{p}
\left(-\delta-(x^{-\frac{1}{p}}u-1)\right)\to -\infty
\]
as $x\to \infty$. Thus, by the maximum principle, we deduce that
$\underline{w}(x)\leq w(x),\ x\geq L$, i.e.,
\[
u(x)-x^\frac{1}{p}\geq -Mx^{\frac{1}{p}-3}-\delta x^\frac{1}{p},\
x\geq L,\ 0<\delta<1.
\]
By letting $\delta \to 0$, we deduce that $u(x)-x^\frac{1}{p}\geq
-Mx^{\frac{1}{p}-3},\ x\geq L$. The validity of (\ref{eqhalfestim})
follows at once from this lower bound and the upper bound $u\leq
x^\frac{1}{p}$.

The proof of the proposition is complete.
\end{proof}
\begin{rem}\label{remSerrin}
An alternative way to establish uniqueness for (\ref{eqhalf}) is the
following: Suppose that $V_+$ solves (\ref{eqhalf}). Then, it is
easy to see that $\lambda V_+$, $\lambda\geq 1$,  is a family of
upper solutions of (\ref{eqhalf}) such that $\lambda
V_+(0)=U_+(0)=0$, and $\lambda V_+-U_+\to \infty$ as $x\to \infty$
if $\lambda>1$. Moreover, since $V_+'(0)>0$, we have that $\lambda
V_+>U_+$ in $(0,\infty)$ if $\lambda\gg 1$. Therefore, by
\emph{Serrin's sweeping technique} (see \cite[pg. 40]{sattinger}),
we get $V_+\geq U_+$ in $[0,\infty)$. Similarly, we can show that
$V_+\leq U_+$ in $[0,\infty)$. Consequently, we get that $V_+\equiv
U_+$.
\end{rem}
\begin{rem}\label{remvoritcesmine}
In relation with the problems mentioned in the third part of
Subsection \ref{secmotivation}, it would be of interest to
generalize Proposition \ref{prophalf} in the following direction:
Study solutions $u:\mathbb{R}^2\to \mathbb{C}$ (if they exist) of
the problem
\[\left\{
\begin{array}{lll}
  \Delta u+\left(|\textbf{y}|-|u|^p \right)u=0, &\textbf{y}\in \mathbb{R}^2,  \\
    &    \\
  u(\textbf{0})=0;  & |\textbf{y}|^{-\frac{1}{p}}|u|\to 1 \
  \textrm{as}\ |\textbf{y}|\to \infty.
\end{array}
\right.
\]

We cannot resist to compare the above problem with the well known
\[\left\{
\begin{array}{lll}
  \Delta u+\left(1-|u|^2 \right)u=0, &\textbf{y}\in \mathbb{R}^2,  \\
    &    \\
  u(\textbf{0})=0;  & |u|\to 1 \
  \textrm{as}\ |\textbf{y}|\to \infty,
\end{array}
\right.
\]
see \cite{herve}, \cite{sigalOvch1}, and the references in the
research monographs \cite{bethuel}, \cite{pacard-riviere},
\cite{serfatybook}.
\end{rem}

The following proposition plays an important role in relation with
Remark \ref{remdirichlet}.
\begin{pro}\label{prohalfnondeg}
The solution $U_+$ of (\ref{eqhalf}) is non-degenerate in the sense
that there are no nontrivial  bounded solutions of the problem
\begin{equation}\label{eqlinearappC}
\phi''-\left[(p+1)U_+^p-x\right]\phi=0,\  x>0,\ \ \phi(0)=0.
\end{equation}
\end{pro}
\begin{proof}
Suppose that there exists a nontrivial bounded solution $\phi$ of
(\ref{eqlinearappC}). The fact that $(p+1)U_+^p-x\to \infty$ as
$x\to \infty$, easily implies that $\phi$ has a finite number of
zeros in $[0,\infty)$, all of them simple, and that $|\phi|,\
|\phi'|$ decay to zero super--exponentially as $x\to \infty$. Let
$r_0\geq 0$ be the largest zero of $\phi$. Without loss of
generality, we may assume that $\phi'(r_0)>0$. Differentiating
(\ref{eqhalf}),  multiplying the resulting identity by $\phi$, then
multiplying (\ref{eqlinearappC}) by $U_+'$, subtracting, and
integrating by parts over $(r_0,\infty)$, we readily arrive at
\[
U_+'(r_0)\phi'(r_0)=-\int_{r_0}^{\infty}U_+\phi dx.
\]
However, this is a contradiction to $U_+'>0,\ \phi'(r_0)>0$, and
$\phi(x)>0,\ x>r_0$.

The proof of the proposition is complete.
\end{proof}
\begin{rem}\label{rempotentialhalfapenc}
Note that, since $U_+(0)=0$ and $p>1$, the potential $(p+1)U_+^p-x$
of the linear  operator in the righthand side of
(\ref{eqlinearappC}) is negative for small $x>0$.
\end{rem}

We can now give the

\texttt{PROOF OF PROPOSITION \ref{proentire}}: Let $U_+$ be as in
Proposition \ref{prophalf}, it is clear that
\begin{equation}\label{equlower}
\underline{u}(x)=\left\{
\begin{array}{ll}
  U_+(x), & x\geq 0,\\
   &  \\
  0, & x\leq 0,
\end{array}
\right.
\end{equation}
is a (weak) lower-solution of
(\ref{eqEqappenC})-(\ref{eqEqappenCbdry}).

Next we will construct an upper-solution of
(\ref{eqEqappenC})--(\ref{eqEqappenCbdry}), for a different
construction we refer the interested reader to Remark
\ref{remuppersolNew} below. Let
\begin{equation}\label{equ0appendix}
u_0(x)=\left\{
\begin{array}{ll}
  x^\frac{1}{p}, & x\geq 0,\\
   &  \\
  0, & x\leq 0.
\end{array}
\right.
\end{equation}
Then, fix a continuous function $\phi \geq 0$ such that $u_0+\phi
\in C^2(\mathbb{R})$ and $\phi(x)=0$ if $|x|\geq 1$. Let $\mu_1>0,\
\psi_1>0$ denote the principal eigenvalue and the corresponding
$L^\infty$-normalized eigenfunction of
\[
-\psi''+\left((p+1)u_0^p-x\right)\psi=\mu\psi,\ \ \psi\in
L^2(\mathbb{R}).
\]
Such $\mu_1,\ \psi_1$ exist, since the potential
\begin{equation}\label{eqpotentialAppendix}
Q(x)\equiv(p+1)u_0^p-x=\left\{
\begin{array}{ll}
  px, & x\geq 0,\\
   &  \\
  -x, & x\leq 0,
\end{array}
\right.
\end{equation}
clearly satisfies $\inf_{x\in \mathbb{R}}Q(x)<\liminf_{x\to
\pm\infty}Q(x)$, see \cite{rabinowitz} (see also \cite[Thm.
10.7]{hislop}). Furthermore, $\psi_1,\ |\psi'_1|,\ |\psi_1''|$ decay
to zero super-exponentially as $|x|\to +\infty$. More precisely,
there exist constants $c_\pm$ and $x_-<0,\ x_+>0$ such that
\begin{equation}\label{eqBOa}
\psi_1(x)\sim c_\pm
\frac{\exp\left\{-\left|\int_{x_\pm}^{x}\sqrt{Q(t)-\mu_1}dt\right|\right\}}{\left[Q(x)-\mu_1\right]^\frac{1}{4}}\
\ \ \textrm{as}\ x\to \pm \infty,
\end{equation}
see \cite[Chap. IV, Thm. 14]{coppel} and \cite[Appx. A]{alfimov}. We
can now define our upper solution for
(\ref{eqEqappenC})--(\ref{eqEqappenCbdry}) as
\[
\bar{u}=u_0+\phi+M\zeta_D\psi_1
\]
with $M,\ D>1$  large constants to be chosen, and
$\zeta_D(x)=z(x-D)$, where $z\in C^\infty(\mathbb{R})$ is such that
$z=1,\ t\leq 0;\ z=0,\ t\geq 1;\ z'<0,\ t\in (0,1)$ . We proceed in
verifying that $\bar{u}$ is indeed an upper solution. In $[0,1]$,
$\bar{u}=(u_0+\phi)+M\psi_1$ and
\[
-\bar{u}''-x\bar{u}+\bar{u}^{p+1}\geq -CM+cM^{p+1}>0,
\]
provided $M>0$ is sufficiently large ($C,c>0$ are independent of
large $M$). We fix such an $M>0$. In $[1,D]$, we have
$\bar{u}=u_0+M\psi_1$ and
\begin{eqnarray}
  -\bar{u}''-x\bar{u}+\bar{u}^{p+1} &=& -u_0''-M\psi_1''-xu_0-Mx\psi_1+(u_0+M\psi_1)^{p+1} \nonumber\\
   &&  \nonumber\\
   &\geq& -M\psi_1''-xu_0-Mx\psi_1+u_0^{p+1}+(p+1)u_0^pM\psi_1=M\mu_1\psi_1>0.\nonumber
\end{eqnarray}
In $[D,D+1]$, we have $\bar{u}=u_0+M\zeta_D\psi_1$ and
\begin{eqnarray}
  -\bar{u}''-x\bar{u}+\bar{u}^{p+1} &=& -u_0''-M\left(\zeta_D\psi_1\right)''-xu_0-M x \zeta_D\psi_1+\left(u_0+M \zeta_D\psi_1\right)^{p+1} \nonumber\\
   &&  \nonumber\\
   &\geq& -\frac{1}{p}\left(\frac{1}{p}-1
   \right)x^{\frac{1}{p}-2}-Ce^{-x}-xu_0+u_0^{p+1}\nonumber\\
&&\nonumber\\
&=&-\frac{1}{p}\left(\frac{1}{p}-1
   \right)x^{\frac{1}{p}-2}-Ce^{-x}>0,
\end{eqnarray}
provided $D$ is chosen large ($C>0$ is independent of $D$). In
$[D+1,+\infty)$, we plainly have $\bar{u}=x^\frac{1}{p}$.
 Analogous calculations also hold in $(-\infty,0]$.
Consequently, the function $\bar{u}$ is an upper-solution of
(\ref{eqEqappenC})-(\ref{eqEqappenCbdry}).

It follows from \cite[Thm. 2.10]{niIndiana} as before that there
exists a solution of (\ref{eqEqappenC}) such that $\underline{u}\leq
u\leq \bar{u}$. The second estimate in (\ref{eqsuperexpappenC})
follows at once from (\ref{eqhalfestim}) and (\ref{equlower}); the
first one follows from the fact that, for every solution of
(\ref{eqEqappenC}) that tends to zero as $x\to -\infty$, there
exists some constant $c>0$ such that
\begin{equation}\label{eqasymptoticappend}
u(x)\sim c\textrm{Ai}(-x)\ \ \textrm{as}\ x\to -\infty,
\end{equation}
(recall the discussion leading to (\ref{eqvariationofcons})), and
(\ref{eqairyasymptotic}).

 We will show that $u'>0$. We follow \cite{alikakos}. Since
$0\leq u \leq x^\frac{1}{p}$ for $x \geq D+1$, as in the proof of
Proposition \ref{prophalf}, we obtain that $u'\to 0$ as $x\to
\infty$. Moreover, it is easy to show that $u'\to 0$
super-exponentially as $x\to -\infty$. Since $u$ is strictly
positive (by the maximum principle), in view of (\ref{eqEqappenC}),
we can write
\[
\left(\frac{u''}{u}-u^p \right)'=-1,
\]
i.e.,
\begin{equation}\label{eqvu}
v''-\frac{1}{u}vv'-pu^pv=-u, \ \ \textrm{where}\ \ v=u'.
\end{equation}
Since $v\to 0 $ as $|x|\to \infty$, it follows that if $v$ is not
strictly positive, then there exists $x_0$ such that
\[
v(x_0)\leq0,\ v'(x_0)=0,\ v''(x_0)\geq 0,
\]
but this is impossible because (\ref{eqvu}) implies
\[
v''(x_0)-pu^p(x_0)v(x_0)=-u(x_0)<0,
\]
and therefore we conclude that $u'$ is strictly positive. The same
conclusion can also be derived by adapting an argument from
\cite{weipitaevskii}, i.e., applying the maximum principle to the
function $\frac{u'}{u}$ (keep in mind the second identity in
(\ref{eqberestycki})). A more PDE approach is to apply the moving
plane method \cite{gidas}, starting from $-\infty$, as in
\cite{gui}.

Uniqueness (of nonnegative solutions) for the problem
(\ref{eqEqappenC})--(\ref{eqEqappenCbdry}) can be established in a
similar manner as we did in Proposition \ref{prophalf} for the
problem (\ref{eqhalf}): Again we suppose that there exist two
distinct non-negative solutions $u_1,u_2$ of
(\ref{eqEqappenC})--(\ref{eqEqappenCbdry}). By the strong maximum
principle, we deduce that both are strictly positive. Hence, there
is some small $\delta>0$ such that ${u}_i(x)\geq \delta
\chi_{[\pi,2\pi]}\sin(x-\pi)$, $i=1,2$, for every $x\in \mathbb{R}$.
Recall that the function in the righthand side is a lower-solution
of (\ref{eqEqappenC})--(\ref{eqEqappenCbdry}). Moreover, both $u_i,\
i=1,2,$  will eventually lie below the graph of $x^\frac{1}{p}$
(note that $u-x^\frac{1}{p}$ is strictly convex as long as it is
nonnegative). Thus, by virtue of (\ref{eqairyasymptotic}),
(\ref{eqpotentialAppendix}), (\ref{eqBOa}) and
(\ref{eqasymptoticappend}), we can choose  sufficiently large
numbers $D, M$ such that $u_i(x)\leq \bar{u}(x)$, $i=1,2$, for every
$x\in \mathbb{R}$. (This is a fine point that was not present in the
uniqueness proof for (\ref{eqhalf})). Recall that the function in
the righthand side is an upper-solution of
(\ref{eqEqappenC})--(\ref{eqEqappenCbdry}). Consequently, we may
assume that $u_1(x)\leq u_2(x)$, $x\in \mathbb{R}$, and
it is easy to see
that $\frac{u_2}{u_1}$ is non-decreasing in $\mathbb{R}$. On the
other hand, it follows from the second relation in
(\ref{eqEqappenCbdry}) that $\frac{u_2}{u_1}\to 1$ as $x\to \infty$.
So, we get that $u_2\leq u_1$ in $\mathbb{R}$ which is a
contradiction.

 Finally, the
non-degeneracy of $U$ can be derived as in Proposition
\ref{prohalfnondeg} for the non-degeneracy of $U_+$.

The proof of the proposition is complete.\ \ \ $\Box$
\begin{rem}\label{remhastingsUniq}
The fact that problem (\ref{eqEqappenC})--(\ref{eqEqappenCbdry}) has
a unique solution, which is a stronger result, has been proven in
\cite{hastings}.

In contrast, problem (\ref{eqdancerPainle}) with $p=2$, has exactly
two solutions. Existence of two solutions has been established by
Holmes and Spence \cite{holmespainleve} by a shooting argument (and
in \cite{dancer-lazer} for any $p>1$, via the method of upper/ lower
solutions and variational arguments, perhaps unaware of
\cite{holmespainleve}), where the authors also conjectured that
these solutions were indeed the only ones. Their conjecture was
settled, to the affirmative, by Hastings and Troy
\cite{hastingsTroy}. However, their proof was, as we discover now
(\emph{almost 25 years later!}), much more complicated than
necessary, and relied on some four decimal point numerical
calculations. Motivated by an idea of ours from
\cite{karalisourdisresonance}, where problem
(\ref{eqkaralisoudisPain}) was shown to have exactly two solutions,
we can give a truly simple proof of the uniqueness result of
\cite{hastingsTroy} as follows. We know from \cite{dancer-lazer},
\cite{holmespainleve} that problem (\ref{eqdancerPainle}) with $p=2$
has a unique increasing solution $\mathcal{U}_+$. Let
$\tilde{\mathcal{U}}$ be any other solution, and let
$\eta=\mathcal{U}_+-\tilde{\mathcal{U}}$. By an easy calculation,
and the maximum principle, we find that $\eta$ has to be a positive
solution of
\begin{equation}\label{eqeta}
\eta''-2\mathcal{U}_+(x)\eta+\eta^2=0,\ x>0,\ \ \eta(0)=0,\
\eta(x)\to 0\ \textrm{as}\ x\to \infty.
\end{equation}
The key observation now is that the solution $\eta$ furnishes an odd
standing wave solution of a focusing NLS equation of the form
$(\ref{eqNLS})_+$, with $N=1,\ q=2$. Thankfully, in the last years a
lot of research and efforts of many authors have been devoted to the
uniqueness of ground states of radially symmetric focusing nonlinear
Schr\"{o}dinger equations with non-decreasing potential (in $r>0$)
and power nonlinearity, considered in the whole space, in a ball, or
an annulus (see \cite{byeonOshita} for the state of the art). The
problem of uniqueness of $\eta$ resembles more the case of the
annulus and,  having all those tools at our disposal which were not
available at the time that \cite{hastingsTroy} was written, we can
infer that uniqueness as well as \emph{non-degeneracy} of a positive
solution $\eta$ of (\ref{eqeta}) follow directly from \cite[Thm.
1.2]{felmerUniqueness}. (In the latter reference, it was assumed
that the potential is strictly positive but it is easy to check that
their proof works equally well for the case at hand, see also
\cite{byeonOshita}, \cite{tanaka}). The non-degeneracy property of
the solution $\tilde{\mathcal{U}}$, which follows readily, is a
\emph{new result}  and, in the context of the original singular
perturbation problem \cite{turcotte} (see also \cite{dancer-lazer},
\cite{danceryanCrtitic}) is more useful than uniqueness (recall
Remark \ref{rempotential}).

To the best of our knowledge, the similarities between the
 singularly perturbed (multi-dimensional) elliptic problem in
\cite{dancer-lazer}, arising from the study of the Lazer-Mckenna
conjecture, and the one-dimensional one in \cite{hastingsTroy},
\cite{holmespainleve}, \cite{turcotte}, arising from the problem of
vertical flow of an internally heated Boussinesq fluid with viscous
dissipation and pressure work, were previously unknown.

The proof of the Lazer-Mckenna conjecture in \cite{dancer-lazer}
consists of constructing solutions of the problem  with arbitrary
many (clustering) sharp downward spikes on top of a positive
minimizer of the corresponding energy functional, as the small
parameter $\varepsilon>0$ tends to zero. The aforementioned
minimizer has a corner layer, along the boundary of the domain,
whose profile is described by the positive solution of
(\ref{eqdancerPainle}). It is our hope that the techniques of the
present paper, together with those already developed in \cite{del
pino cpam}, \cite{malchiodi-fife}, can be used to construct new
solutions, having corner layer profile described by the unstable
solution of (\ref{eqdancerPainle}) (is this unique and non-degenrate
for $p\neq 2$?), at least when $\varepsilon$ stays away from certain
critical values $\varepsilon_1>\varepsilon_2>\cdots>\varepsilon_i\to
0$ (see \cite{karalisourdisresonance} for a related problem). These
solutions would be slightly negative in a small
($\varepsilon$-dependent) neighborhood of the boundary, with their
Morse index diverging as $\varepsilon\to 0$ (away from the critical
numbers). Then, one could use the techniques of \cite{duSpikes},
\cite{weilayerspike} in order to add downward spikes on top of that
unstable solution. One may even be able to prove the existence of
arbitrary many solutions, as $\varepsilon\to 0$, which was the
original assertion of the Lazer-Mckenna conjecture settled in
\cite{dancer-lazer}, just by the fact that the solution's Morse
index diverges (recall also the discussion following
(\ref{eqopenev2})). This would constitute a proof of the
Lazer-Mckenna conjecture, as treated in \cite{dancer-lazer},
\cite{danceryanCrtitic}, that is valid even for \emph{supercritical
exponents}.
\end{rem}

\begin{rem}\label{rempainleveNeumann}
The related boundary value problem
\begin{equation}\label{eqpalamides}
u_{xx}=u(u^2-x),\ x>0;\ \ u_x(0)=0,\ u-\sqrt{x} \to 0\ \textrm{as}\
x \to \infty,
\end{equation}
 arises
 in the study of the superheating field attached to a semi-infinite
superconductor, for the construction of a family of approximate
solutions of the Ginzburg-Landau system via the procedure of
(formally) matching inner and outer solutions (see
\cite{chapman,helfer}). Existence for the problem
(\ref{eqpalamides}) has been established by shooting arguments in
\cite{helfer}, and by topological ones in \cite{palamides}. We can
give a new proof of their results, valid for any power nonlinearity
(as in Proposition \ref{prophalf}), by slightly modifying the above
proof of Proposition \ref{proentire} as follows: One still uses
$U_+$ as a lower solution of (\ref{eqpalamides}) (it is a solution
and $U_+'(0)>0$, see \cite{Berestyckilion}); however in the
construction of the upper solution we have to be careful to chose
the function $\phi$ such that $(u_0+\phi)_x(0)\leq 0$ (and
afterwards the principal eigenfunction subject to Neumann boundary
conditions at $x=0$). It was shown in \cite{guedapainleve} that,
without any assumptions at infinity, problem (\ref{eqpalamides}) has
exactly one global positive solution. As a matter of fact, we expect
that an analogous property holds true for the problem
(\ref{eqhalf}), see also \cite{brezis} for a related result
concerning (\ref{eqkaralisoudisPain}).
\end{rem}
\begin{rem}\label{remSigal}
Since $p>1$, the solutions that we have constructed in Propositions
\ref{proentire}, \ref{prophalf} have infinite energy (more
precisely, their derivative does not belong in $L^2(0,\infty)$).
Nevertheless, we believe that one can also establish existence (and
further characterize the solutions) for problems
(\ref{eqEqappenC})--(\ref{eqEqappenCbdry}), (\ref{eqhalf}), and
(\ref{eqpalamides}) by minimizing a suitable re-normalized energy
functional, as in \cite{sigalOvch1}, or minimizing the standard
energy functional of (\ref{eqEqappenC}) in a large interval $[-R,R]$
with boundary conditions $u(-R)=0; \ u(R)=R^\frac{1}{p}$, and then
letting $R\to \infty$ (see \cite{berestycki-wei2012-2},
\cite{saddleCabreJEMS1}, \cite{delPinoscrew} for some related
situations). Similarly for the other cases.
\end{rem}
\begin{rem}\label{remlastprofile}
If $0<p\leq1$, then there exists a unique solution of
(\ref{eqEqappenC})--(\ref{eqEqappenCbdry}) such that $u>u_0$, where
$u_0$ is as in (\ref{equ0appendix}). We cannot resist to give a
short  proof of this, based on \cite{karalisourdisradial}. If
$0<p<1$, it is easy to see that $u_0$ is a lower solution, while
$u_0+\varphi$, with $\varphi>0$ defined by
\[
-\varphi_{xx}+\left[(p+1)u_0^p-x \right]\varphi=(u_0)_{xx}\geq 0,\ \
\ \varphi(\pm \infty)=0,
\]
is an upper solution of (\ref{eqEqappenC})--(\ref{eqEqappenCbdry}).
If $p=1$, then $u_0$ is a weak lower solution, while $u_0+\phi$,
with $\phi>0$ the unique continuous solution of
\[
-\phi_{xx}+\left[(p+1)u_0^p-x \right]\phi=0,\ \
\phi_x(0^-)-\phi_x(0^+)=1, \ \phi(\pm \infty)=0,
\]
is an upper solution of (\ref{eqEqappenC})--(\ref{eqEqappenCbdry})
(recall (\ref{eqpotentialAppendix})). Existence of the desired
solution follows at once. Uniqueness follows simply by taking the
difference of the equations satisfied by two pairs of solutions.

The case $p=1$, which is equivalent to (\ref{eqkaralisoudisPain}),
has received considerable attention lately, mainly since it appears
in the study of the spatial segregation limit of competitive systems
\cite{teraccini}, \cite{iida} (see also
\cite{hutson-lou-mischaikow-jde}, \cite{schecter-sourdis}).
Interestingly enough, we have noticed that problem
(\ref{eqkaralisoudisPain}) also describes the corner layer profile
of solutions in the paper \cite{hastingsJnonlinear} by
\emph{Hastings and McLeod}, in the case where the constant $c$
therein, which arises from an integration, is chosen to be zero
rather than strictly positive.
\end{rem}
\begin{rem}
Similar results should also hold true for the equation
\[
u''+x|x|^{s}u-|u|^{p}u=0,\ \ x\in \mathbb{R},
\]
where $p,s>0$.
\end{rem}
\begin{rem}\label{remuppersolNew}
As we have already seen in Remark \ref{remlastprofile}, there exists
a unique solution $U>\max\{0,x\}$ of the problem
\begin{equation}\label{eqlogisticp}
u''=pu(u-x)=0, \ \ x\in \mathbb{R},
\end{equation}
such that $U\to 0$ as $x\to -\infty$; $U-x\to 0$ as $x\to \infty$.
(We have found out in \cite{karalisourdisresonance}, by arguing as
in Remark \ref{remhastingsUniq} above, that there exists exactly one
more solution $U_-$ of (\ref{eqlogisticp}) which satisfies the
boundary conditions, and in fact $U_-<\max\{0,x\}$). Actually, the
solution $U$ is the unique (global) solution of (\ref{eqlogisticp})
such that $u\geq \frac{x}{2},\ x\in \mathbb{R}$. This follows at
once from the fact that the equation in (\ref{eqkaralisoudisPain})
has a unique nonnegative solution, see \cite{brezis}. Moreover,
since $p>1$, it is easy to see that the function $U^\frac{1}{p}$ is
an upper-solution of (\ref{eqEqappenC})--(\ref{eqEqappenCbdry}).
\end{rem}

\textbf{Acknowledgment.} The research leading to these results has
received funding from the European Union's Seventh Framework
Programme (FP7-REGPOT-2009-1) under grant agreement
$\textrm{n}^\textrm{o}$ 245749.

\end{document}